\documentclass[12pt]{amsart}

\usepackage{fullpage,url,amssymb,enumerate,colonequals}
\usepackage{mathrsfs}
\usepackage[section]{placeins}
\usepackage{MnSymbol}
\usepackage{extarrows}
\usepackage{lscape}
\usepackage[all,cmtip]{xy}
\usepackage{tablefootnote}
\usepackage{footnote}

\usepackage{makecell}

\usepackage[OT2,T1]{fontenc}
\usepackage{color}

\usepackage[
        colorlinks, citecolor=darkgreen,
        backref,
        pdfauthor={Nuno Freitas},
]{hyperref}

\usepackage{comment}
\usepackage{multirow}

\numberwithin{equation}{section}

\newtheorem{lemma}[equation]{Lemma}
\newtheorem{theorem}[equation]{Theorem}
\newtheorem{proposition}[equation]{Proposition}
\newtheorem{prop}[equation]{Proposition}
\newtheorem{corollary}[equation]{Corollary}
\newtheorem{conj}[equation]{Conjecture}

\newtheorem{claim*}{Claim}

\theoremstyle{definition}
\newtheorem{remark}[equation]{Remark}

\newtheorem{example}[equation]{Example}
\newtheorem{question}[equation]{Question}

\definecolor{darkgreen}{rgb}{0,0.5,0}
\definecolor{rem}{rgb}{0.8,0,0}
\definecolor{new}{rgb}{0.3,0.1,0.9}
\definecolor{reply}{rgb}{0,0,0.8}
\definecolor{gray}{gray}{0.7}

\renewcommand{\gcd}{\text{gcd}}
\newcommand{\F}{\mathbb{F}}
\newcommand{\PP}{\mathbb{P}}
\newcommand{\Q}{\mathbb{Q}}
\newcommand{\R}{\mathbb{R}}
\newcommand{\Z}{\mathbb{Z}}
\newcommand{\rhobar}{{\overline{\rho}}}

\newcommand{\calO}{\mathcal{O}}
\newcommand{\fp}{\mathfrak{p}}
\newcommand{\Xp}{X_E^{-}(p)}

\newcommand{\Qbar}{{\overline{\Q}}}
\newcommand{\Gal}{\text{Gal}}

\newcommand{\Qun}{\mathbb Q_\ell^{\text{un}}}
\newcommand{\Knr}{K_{\text{nr}}}

\DeclareMathOperator{\GL}{GL}
\DeclareMathOperator{\SL}{SL}
\DeclareMathOperator{\Frob}{Frob}

\newcommand{\Ebar}{{\overline{E}}}

\DeclareMathOperator{\Id}{Id}
\DeclareMathOperator{\Aut}{Aut}
\DeclareMathOperator{\id}{id}
\newcommand{\loccit}{{\it loc. cit.}}

\DeclareMathOperator{\tr}{tr}

\DeclareMathOperator{\End}{End}

\DeclareMathOperator{\Hom}{Hom}

\DeclareMathOperator{\Norm}{ Norm}

\DeclareMathOperator{\Cl}{Cl}

\DeclareMathOperator{\un}{un}

\newcommand{\oE}{\overline{E}}

\newcommand{\cC}{\mathcal{C}}

\newcommand{\cO}{\mathcal{O}}

\newcommand{\LMFDBE}[1]{\href{https://www.lmfdb.org/EllipticCurve/Q/#1}{\textsf{#1}}}
\newcommand{\LMFDBL}[1]{\href{https://www.lmfdb.org/padicField/#1}{\textsf{#1}}}

\begin{document}

\title {Local points on twists of $X(p)$ with applications}
\author{Nuno Freitas}
\address{Instituto de Ciencias Matemáticas (ICMAT),
          Nicolás Cabrera 13-15
         28049 Madrid, Spain}
\email{nunob.freitas@icmat.es}

\author{Diana Mocanu}
\address{Max Planck Institute for Mathematics,
          Vivatsgasse 7, 
         53111 Bonn, Germany
         }
\email{d.mocanu@mpim-bonn.mp.de}

\thanks{Freitas was partly supported by the PID 2022-136944NB-I00 grant of the MICINN
(Spain)}

\keywords{Weil pairing, symplectic  isomorphisms, modular curves}
\subjclass[2020]{Primary 11G05, 11G07. Secondary 14H10, 14G12}

\begin{abstract}
Let $E/\mathbb Q$ be an elliptic curve and $p \geq 3$ a prime.
The modular curve~$X_E^-(p)$ parametrizes elliptic curves with $p$-torsion modules anti-symplectically isomorphic to~$E[p]$. We give a complete classification of when $X_E^-(p)(\Q_\ell)$ is non-empty, for all primes $\ell$.

We give two different applications. First, we classify CM curves $E/\Q$ for which the modular curve $X_E^-(p)$ is a counterexample to the Hasse principle for infinitely many~$p$. Assuming the Frey--Mazur conjecture, we prove that for at least $60\%$ of rational elliptic curves $E$, the modular curve $X_E^-(p)$ is a counterexample to the Hasse principle for at least $50\%$ of primes~$p$. Secondly, we introduce a new technique to the elimination stage of the modular method and apply it to show that $x^3+y^3=5^\alpha z^p$ has no non-trivial primitive solutions for various primes $p$ satisfying $(\alpha/p)=-1$.
Moreover, as a by-product of our work, we simplify the assumptions of several local symplectic criteria due to the first author and Alain Kraus.

\end{abstract}

\maketitle
\vspace{-8mm}
\section{Introduction}

Let $E/\Q$ be an elliptic curve and $p \geq 3$ a prime. We consider the modular curves
$X_E^+(p)$ and~$X^-_E(p)$ whose non-cuspidal $\Q$-points parametrize pairs $(E', \phi)$ where $E'/\Q$ is an elliptic curve and $\phi : E[p] \to E'[p]$ is a symplectic (respectively anti-symplectic) isomorphism of $G_\Q$-modules.
Note that $X_E^+(p)(\Q)$ always contains the canonical point $(E, \id_{E[p]})$, while $\Xp(\Q)$ may be empty. It is well known that $\Xp$ always has rational $\Q$-points for $p=3$ and $p=5$
(see Theorem~\ref{thm:p=3}), so we assume $p \geq 7$.
We will answer the following question in almost full generality.
\begin{question}\label{Q}
Given an elliptic curve $E/\mathbb{Q}$ and a prime $p\geq 7$, when does~$X_E^-(p)$ have points over every completion of $\mathbb{Q}$?
\end{question}

This question fits naturally in the broader context of 
local-global principles. For example, works of Sutherland~\cite{SutherlandLocalGlobal}, Banwait--Cremona~\cite{BanwaitCremona} and Gajovi\'c--Hanselman--Koutsianas~\cite{Gajovic}
study the failure of a local-global principle for the existence
of $\ell$-isogenies for elliptic curves. These questions are related to rational points on the classical
modular curve $X_0(N)$. It is natural to consider analogous local-global
questions in the setting of twists of modular curves. For example, \"Ozman \cite{ekin, ekin2} studied local points on the twists $X^d(N)$ of $X_0(N)$ whose $\Q$-points are identified with
the $\Q(\sqrt{d})$-rational points of $X_0(N)$ that are fixed by $\sigma \circ \omega_N$ where $\sigma$ is the non-trivial element of $\Gal(\Q(\sqrt{d})/ \Q)$ and~$\omega_N$ is the Atkin--Lehner involution. 
Furthermore, the recent work of Lorenzo and Vullers~\cite{LorVul} gives twists
of~$X(7)$ that are counterexamples to the Hasse principle and, assuming Conjecture 5.12 in {\it loc.cit}, they show that there are infinitely many such counterexamples. We will extend their results 
to twists of~$X(p)$ for infinitely many primes $p$, as summarized in Theorem~\ref{thm:conterHasseCM} and Corollary~\ref{cor:bound}. 

Our main results yield a complete answer to Question~\ref{Q}. 
To state them we need to introduce some notation. An elliptic curve $E/\Q$ with potentially good reduction at a prime~$\ell$ has a semistability defect $e:=e(E/\Q_\ell)$, which is the degree of the minimal field extension $L/\Qun$ over which $E$ acquires good reduction. It is well known that
$e \in \{1, 2,3,4,6,8,12,24 \}$
and by the work of Kraus~\cite{Kraus} one can easily determine~$e$; note that $e=1$ is equivalent to $E/\Q_\ell$ having good reduction. Given a minimal model for $E/\Q_\ell$ of discriminant~$\Delta_m$, whenever $c_4 \neq 0$ or $c_6 \neq 0$,
we define $\tilde{c}_4$, $\tilde{c}_6$
and $\tilde{\Delta}$~by
\[
 c_4 = \ell^{v_\ell(c_4)} \tilde{c}_4, \qquad c_6 = \ell^{v_\ell(c_6)} \tilde{c}_6, \qquad \Delta_m = \ell^{v_\ell(\Delta_m)}\tilde{\Delta}. 
\]
Moreover, for $E/\Q_\ell$ with good reduction, we denote by $\oE$ its reduction over $\F_\ell$ and set $a_\ell:=(\ell+1) - \# \oE(\F_\ell)$. 
The Frobenius endomorphism $\pi\in \End(\oE)$ satisfies the equation $\pi^2-a_\ell \:\pi+\ell=0$ of discriminant 
$\Delta_\ell := a_\ell^2 - 4\ell$. Set also $b_E:=[\End (\oE): \Z[{\pi}]]$, and
note that $b_E$ is finite when $\oE$ is ordinary.

The following result is a complete classification of local points on $\Xp$ away from~$p$.

\begin{theorem}\label{thm:MAIN}
    Let $E/\Q$ be an elliptic curve, $p\geq 7$ a prime and $K \neq \Q_p$ a completion of $\Q$.
    Then $X_E^-(p)(K)= \emptyset$ if and only if $K = \Q_\ell$, $E/\Q_\ell$ has potentially good reduction and 
    we are in one of the cases in Table~\ref{table:main}.
\begin{table}[h!]
    \centering 
    \begin{tabular}{|c||c|p{8cm}|}
        \hline
        $e$ & $\ell$ & Additional conditions \\ \hline
        $1$  & $\ell\not\equiv1 \pmod p$ &
        $p \equiv 3 \pmod{4}$, $p \nmid b_E$,
\\
    
        & & $-p\Delta_\ell$ is a square\footnotemark[1], $|a_\ell|< p-2\sqrt{\ell}$, and \\
        & & if $q\neq \ell$ prime, $q \mid \Delta_\ell$ then $(q/p)\neq -1$, \\
        \hline
        $2$ & $\ell\not\equiv1 \pmod p$ & \makecell[l]{
            Reduce to case $e=1$, by replacing $E$ with $E^u$, \\
            where $u$ is given by Lemmas~3, 4 in~\cite{symplectic}
        }  \\ \hline
        $3,\:6$ &
        \makecell[l]{
            $\ell \equiv 2 \pmod{3}$\\
            $\ell = 3$
        } &
         \makecell[l]{
           $(3/p)=1$ and $(\ell/p)=1$\\
            $(3/p)=1$ and $\tilde{\Delta} \equiv 2 \pmod 3$
        }  \\ \hline
        $4$ &  \makecell[l]{
            $\ell \equiv 3 \pmod{4}$\\
            $\ell = 2$
        } &   \makecell[l]{
           $(2/p)=1$ and $(\ell/p)=1$\\
            $(2/p)=1$ and $\tilde{c}_4 \equiv 5\tilde{\Delta}\pmod{8}$
        } \\ \hline
        $8,\:24$ & $\ell=2$ & $(2/p)=1$\\ \hline
        $12$ & $\ell=3$ & $(3/p)=1$ \\
        \hline
    \end{tabular}
    \vspace{3mm}
    \caption{All cases of $\Xp(K)=\emptyset$, with $K \neq \Q_p$ a completion of $\Q$.}\label{table:main}
\end{table}
\footnotetext[1]{In particular, Remark~\ref{rmk:good} implies that $E/\F_\ell$ is ordinary in this case, and thus $b_E$ is finite. }
\vspace{-5mm}
\end{theorem}

\vspace{-3mm}
In the proof of this theorem, the main ideas depend on the type of reduction of~$E$ at~$\ell$ and are as follows. By Corollary~\ref{cor:isogeny}, a natural source of points on~$\Xp(\Q_\ell)$ are isogenies of $E/\Q_\ell$ with degree that is a non-square modulo~$p$. Specifically, for primes of
potentially multiplicative reduction, it follows from the theory of the Tate curve that such isogenies exist (see~Theorem~\ref{thm:multiplicative}). Moreover, in Proposition~\ref{prop:good1}, we prove that
for all primes~$\ell$ of good reduction when $p \equiv 1 \pmod{4}$ and most primes of good reduction when $p\equiv 3 \pmod{4}$ such isogenies exist. For the exceptions when $p \equiv 3 \pmod{4}$, using Kohel's thesis~\cite{Kohel} to understand $\End(\overline{E}/\F_\ell)$, we generate points on $X_\Ebar^-(p)(\F_\ell)$, and then lift them to~$\Xp(\Q_\ell)$.
For primes of potentially good reduction, the cases $e=2,6$ reduce via a quadratic twist to $e=1,3$, respectively, and all the other cases depend on whether $\Q_\ell(E[p])/\Q_\ell$ is an abelian extension or not. 
When $\Q_\ell(E[p])/\Q_\ell$ is abelian, a linear algebra argument gives an anti-symplectic automorphism of $E[p]$ (see~Theorem~\ref{thm:abelian}), and therefore a point on $\Xp(\Q_\ell)$. Otherwise,
we perform a detailed study of the $p$-torsion
module of $E/\Q_\ell$. Indeed, when $\Q_\ell(E[p])/\Q_\ell$ is a non-abelian extension and $e \in \{3,4\}$ we show there is only one isomorphism class of $G_{\Q_\ell}$-modules~$E[p]$
unless $(\ell,e)= (2,4)$ in which case there are two possibilities; for $e \in \{8,12,24\}$, we show that, up to a quadratic twist, $E[p]$ is determined by the action of (the non-abelian) inertia, for which work of the first author together with Demb\'el\'e and Voight~\cite{inertial} gives a full set of possibilities. To complete the proof, 
for each of the possible modules $E[p]$, we use the 
symplectic criteria in~\cite{symplectic} to either construct an elliptic curve $E'/\Q_\ell$ such that $E'[p]$ and $E[p]$ are anti-symplectically isomorphic, or to show that all possible $G_{\Q_\ell}$-isomorphisms of $E[p]$ must be symplectic and therefore $\Xp(\Q_\ell)$ is empty.

We highlight that, as a by-product of this approach, we 
are able to remove unnecessary assumptions in some of the local symplectic criteria in~\cite{symplectic}; see Section~\ref{sec:revisiting} for details.

We now state the classification of $\Q_p$-points on $\Xp$.

\begin{theorem}\label{thm:l=p}
Let $E/\Q$ be an elliptic curve and let $p\geq7$ be a prime. 
 Then $X_E^-(p)(\Q_p)= \emptyset$ if and only if $E/\Q_p$ has potentially good reduction and we are in one of the cases in Table~\ref{table:Qp}.

 \begin{table}[h!]
    \centering
    \begin{tabular}{|c||p{11cm}|}
        \hline
        $e$ & Additional conditions \\ \hline

        $1$
        & $p\equiv7\pmod8$, $a_p(E) = 0$ 
        \\ \hline

        $3$
        & \makecell[l]{
        $p>11, p\equiv11\pmod{12}$, \\
         $p=11$, $v_{11}(\Delta_m)\ne 8$ or $v_{11}(c_4)< 4$
        }\\ \hline
        $4$
        & \makecell[l]{$p>7$, $p\equiv7\pmod8$,\\
        $p=7$, $v_7(c_6)\neq\frac{v_7(\Delta_m)+1}{2}$
        }\\ \hline
        $2$, $6$ & \makecell[l]{
            Reduce to case $e=1$, $e=3$, by replacing $E$ with $E^{\sqrt{p}}$}\\ \hline
    \end{tabular}
    \vspace{3mm}
    \caption{All cases of $X_E^-(p)(\Q_p)=\emptyset$.}
    \label{table:Qp}
\end{table}
\vspace{-5mm}
Moreover, in the cases $e=3,4$ of the table $E/\Q_p$ has potentially supersingular reduction.
\end{theorem}
\vspace{-2mm}
In proving this theorem, a major difficulty is that the torsion prime equals the residue characteristic, so we work with the finite flat group scheme structure of the $p$-torsion. Similarly to the case $\ell \neq p$, the existence of local points is
obtained in most cases from prime to $p$ isogenies (e.g. Theorems~\ref{thm:l=pe3} and \ref{thm:l=re4}).
For $p>11$, in all cases where we cannot find an isogeny, the curve $E/\Q_p$ has potentially good supersingular reduction and we prove that $\Xp(\Q_p)=\emptyset$; see Theorem~\ref{thm:l=pss}.
Indeed,  
from am application of Kraus' formulae for the Serre weights~\cite{Kraus1997} we conclude that any $E'$ satisfying $E[p] \simeq E'[p]$ also has potentially good supersingular reduction. Using Raynaud's theorem~\cite[Corollaire~3.3.6]{Raynaud} over a common field of good reduction, we reduce to the residual $\F_p$-group schemes setting, and show that $\oE[p]$ and $\oE'[p]$ are symplectically isomorphic using Dieudonn\'e theory.
The primes $p=7$ and $p=11$ admit extra behaviours which we treat separately.

The following is an immediate consequence of our main theorems.

\begin{corollary}\label{cor:semistable}
Let $E/\Q$ be a semistable elliptic curve and $p \equiv 1 \pmod{4}$ a prime. Then $\Xp$ has local points everywhere.
\end{corollary}
This immediately raises the question of when $\Xp$ is a counterexample to the Hasse principle. In this direction, we will prove the following in Section~\ref{sec:Hasse}.
\begin{theorem}\label{thm:conterHasseCM}
Let $K=\Q(\sqrt{D})$ where $D \in \{-11,-19,-43,-67,-163 \}$. Let also $E/\Q$ be an elliptic curve with CM by~$K$. If $p > 7$ is a prime such that $p \equiv 5 \pmod{8}$ and $(D / p) = 1$, then $\Xp$ is a counterexample to the Hasse principle.
\end{theorem}
This theorem is optimal in the sense that, for all primes $p \geq 3$ and all CM curves not in its statement, the curve~$\Xp$ is not a counterexample to the Hasse principle (see Remark~\ref{rem:CM}); we refer to Example~\ref{ex:final} for a counterexample in the residually reducible case.

We will also show, conditional on the Frey--Mazur conjecture (see Conjecture~\ref{conj:FM}),
that there are infinitely many elliptic curves $E$ such that $\Xp$ is a
counterexample to the Hasse principle for infinitely many~$p$. In particular, we will prove
\begin{corollary} \label{cor:Frey-MazurSimple} Assume the Frey--Mazur Conjecture.
   Let $p \equiv 1 \pmod{4}$ be a prime greater than 17.
   Then, for all semistable elliptic curves $E/\Q$ with no non-scalar $\Q$-isogenies, the curve $X_E^-(p)$ is a counterexample to the Hasse principle.
\end{corollary}
In view of the above, it is natural to wonder how often $\Xp$ is a counterexample to the Hasse principle.
Recall that for an elliptic curve $E$ given by a Weierstrass equation with integer coefficients $\mathbf{a} = (a_1, a_2, a_3, a_4, a_6) \in \Z^5$ one defines the na\"ive height of $E$ by 
\[
\text{ht}(E):= \text{ht}(\mathbf{a})=\max_i|a_i|^{1/i}.
\]
The work of Cremona and Sadek~\cite[Theorem 1.1]{Cresad} shows that when ordered by height, just over $60\%$ of elliptic curves are semistable. Moreover, an immediate consequence of a result of Duke~\cite[Theorem 1]{Duke} shows that 
the set of elliptic curves possessing a non-scalar $\Q$-isogeny has density zero in the same ordering. Together with Corollary~\ref{cor:Frey-MazurSimple}, this yields the following.

\begin{corollary}\label{cor:bound} 
Assume the Frey--Mazur conjecture.
When ordered by na\"ive height, 
for at least $60\%$ of elliptic curves $E$, the twisted modular curve $\Xp$ is a counterexample to the Hasse principle for at least $50\%$ of primes~$p$.
\end{corollary}

Another motivation for Question~\ref{Q} is its applications to Diophantine equations.
In Section~\ref{sec:33p}, we explain how our main results allow us to introduce a new technique in the elimination step of the modular method, and derive the following Diophantine statement.
\begin{theorem}\label{thm:33p}
Let $\alpha \in \Z_{>0}$. The Fermat-type equation
$$x^3+y^3=5^\alpha z^p$$
has no non-trivial primitive integer solutions whenever $(\alpha / p ) = -1$ and
    \[
    p\in \{167, 191, 383, 431, 479, 503, 599, 647, 719, 743, 839, 863, 887, 911, 983\}.
    \]

\end{theorem}
\vspace{-2mm}
Furthermore, for an additional Diophantine application, we refer to~\cite{FMS}, where our main theorem is applied in the study of the Fermat-type equation $x^2 + y^3 = z^p$.

\subsection{Computational software}
For the computations needed in this paper, we used {\tt Magma} computer algebra system \cite{magma}, version V2.28-20.
All our
code is available at \cite{git}, and the repository instructions explain how it is used.
\subsection{Acknowledgements} 
We are grateful to Samir Siksek for useful comments at
the beginning of this work. We thank Maarten Derickx and Filip Najman for telling us about Proposition~\ref{prop:cuspsQ} and Elie Studnia for Proposition~\ref{prop:elie}. We thank Filip Najman also for discussions about Example~\ref{ex:final}. We thank Ignasi S\'anchez-Rodr\'iguez for helpful discussions related to Theorem~\ref{prop:good2}. We also thank the anonymous referee 
for an exceptionally detailed report that substantially improved the correctness, completeness and exposition of the paper, particularly for the suggestion of Theorem~\ref{thm:l=pss} in the case of good reduction.
Finally, we thank John Cremona, Alain Kraus, Ariel Pacetti for helpful discussions and remarks. 
The authors completed this work at the Max Planck Institute for Mathematics in Bonn and appreciate its hospitality and financial support.
\vspace{-3mm}
\section{Notation} \label{s:notation}
\vspace{-2mm}
In this work, $\ell$ and $p \geq 3$ are always primes.
Let $\F_\ell$ be the finite field with $\ell$ elements.

For a field $F$, we write $\overline{F}$ for an algebraic
closure and $G_F := \Gal(\overline{F}/F)$ for its absolute Galois group. For an extension $F/\Q_\ell$, we write $I_F$ for the inertia subgroup of $G_F$. We denote by $F^{\un} \subset \overline{F}$
the maximal unramified extension of $F$. Note that $F^{\un}$ is the field fixed by~$I_F$. Let $\Frob_F \in G_F$ denote a lift of Frobenius. 
Moreover, we denote by $F_{\text{nr}}$ the maximal unramified subextension of $F$. 
When $F= \Q_\ell$ we write $I_\ell$ for $I_F$ and $\Frob_\ell$ for $\Frob_F$.

For $E$ an elliptic curve defined over any field~$F$, we write $E[p]$ for its $p$-torsion
$G_F$-module and $\rhobar_{E,p} : G_F \to \Aut(E[p])$ for the corresponding Galois representation. We have $\det\: \rhobar_{E,p} = \chi_p$ the mod~$p$ cyclotomic character.
An \(F\)-isogeny \(\varphi : E\to E'\) is called {\it scalar} if there
exist an \(F\)-isomorphism \(\iota\colon E\to E'\) and an integer
\(m\neq0\) such that $\varphi=\iota\circ[m]$.

For $F$ a number field or a finite extension of $\Q_\ell$, we denote the conductor of $E/F$ by~$N_E$.

For an elliptic curve $E/\Q_\ell$,
let $\Delta_m: = \Delta_m(E)$ denote the discriminant of a minimal Weierstrass model of $E$.
Given a minimal model for $E/\Q_\ell$, whenever $c_4 \neq 0$ or $c_6 \neq 0$,
we define the quantities $\tilde{c}_4$, $\tilde{c}_6$
and $\tilde{\Delta}$~by
\[
 c_4 = \ell^{v_\ell(c_4)} \tilde{c}_4, \qquad c_6 = \ell^{v_\ell(c_6)} \tilde{c}_6, \qquad \Delta_m = \ell^{v_\ell(\Delta_m)}\tilde{\Delta}.
\]
Given an elliptic curve $E/\Q_\ell$ with potentially good reduction 
and an odd prime $q\neq \ell$, we let $L:=\Qun(E[q])$.
We call $L$ the {\it inertial field} of $E$ and we note that it is independent of the choice of $q\neq \ell$:
the field $L$ is the smallest extension of $\Qun$ where
$E$ obtains good reduction. We denote its Galois group by 
$\Phi := \Gal(L/\Qun)$, and refer to it as the {\it inertial group}.
We denote by $e = e(E)$ the order of $\Phi$, which is called the {\it semistability defect of} $E$. 
The curve $E/\Q_\ell$ has good reduction if and only if $e=1$. When $e \neq 1$, it is well known (see \cite{Kraus}) that
either $\Phi$ is cyclic of order $2,3,4,6$; or $\ell=3$ and $\Phi\cong C_3 \rtimes C_4$ is of order $12$; or $\ell=2$ and $\Phi \cong \SL_2(\F_3)$ or $\Phi \cong Q_8$,  where $Q_8$ is the quaternion group.
In the presence of two elliptic curves
$E/\Q_\ell$ and $E'/\Q_\ell$ we adapt all the
notation accordingly, namely, we will write $E'[p]$, $e'$, $c_4'$, $\tilde{c}_4'$, $c_6'$, $\tilde{c}_6'$
and $\Delta_m'$, $\tilde{\Delta}'$.

Let $v_\ell$ denote the $\ell$-adic valuation in $\Q_\ell$ satisfying $v_\ell(\ell)=1$.

For any $a \in \Z$ we will write $(a/p)$ for the Legendre symbol.

For all $p$ we will denote by $\Id$ the identity element in $\GL_2(\F_p)$ and by $\# A$ the order of an element $A$ in $\GL_2(\F_p)$. For any matrix $M\in M_{n,m}(K)$, for $m,n\geq 1$ and $K$ a field, we denote by $M^T$ the transpose of $M$.

For $A$ and $B$ subgroups of a group $G$ we write $A \cdot B$ for the smallest subgroup of $G$ containing both $A$ and $B$.

\section{Background on \texorpdfstring{$X(p)$}{} and \texorpdfstring{$X_E^\pm (p)$}{}}
\label{sec:background}

\subsection{Symplectic isomorphisms}
Let $p \geq 3$ be a prime. Let $K$ be a field of characteristic different from $p$.
Fix a primitive $p$-th root of unity $\zeta_p \in \overline{K}$.
For $E/K$ an elliptic curve, we write~$e_{E,p}$ for the Weil pairing on~$E[p]$.
We say that an $\F_p$-basis $(P,Q)$ of $E[p]$ is \emph{symplectic} if
$e_{E,p}(P,Q) = \zeta_p$.

Now let $E /K$ and $E'/K$ be two elliptic curves and
let $\phi : E[p] \to E'[p]$ be an isomorphism of $G_K$-modules.
Then there is an element $r(\phi) \in \F_p^\times$ such that
\begin{equation}\label{eq:defr}
    e_{E',p}(\phi(P), \phi(Q)) = e_{E,p}(P, Q)^{r(\phi)} \quad \text{for all $P, Q \in E[p]$.} 
\end{equation}
Note that for any $a \in \F_p^\times$ we have $r(a\phi) = a^2 r(\phi)$.
So up to scaling~$\phi$, only the class of~$r(\phi)$ modulo squares matters.
We say that $\phi$ is a \textit{symplectic isomorphism} if
$r(\phi)$ is a square in~$\F_p^\times$, and an \textit{anti-symplectic isomorphism}
if $r(\phi)$ is a non-square.
Fix a non-square~$r_p \in \F_p^\times$.
We say that $\phi$ is \emph{strictly symplectic}, if $r(\phi) = 1$, and
\emph{strictly anti-symplectic}, if $r(\phi) = r_p$.
Finally, we say that $E[p]$ and~$E'[p]$ are
\emph{symplectically} (or \emph{anti-symplectically}) \emph{isomorphic},
if there exists a symplectic (or anti-symplectic) isomorphism of~$G_K$-modules between them.
Note that $E[p]$ and~$E'[p]$ may be both symplectically
and anti-symplectically isomorphic; this will be the case if and only if
$E[p]$ admits an anti-symplectic automorphism.

Isogenies are a natural source of symplectic and anti-symplectic isomorphisms, as described by the following lemma.

\begin{lemma}\label{lem:isogeny}
		Let $\phi: E \to E'$ be a $K$-isogeny of degree $n$ coprime to $p$, so that $\phi$ restricts to an isomorphism $\phi|_{E[p]} : E[p] \to E'[p]$ of $G_K$-modules. Then, $\phi|_{E[p]}$ is symplectic if $(n/p) = 1$ and anti-symplectic otherwise.
	\end{lemma}
  \begin{proof} This is \cite[Corollary~1]{symplectic}.
  \end{proof}
\begin{remark}\label{rmk:multiplication-isog} 
    \begin{itemize}
        \item[1.] Whenever $p\mid n$, then $\phi|_{E[p]}$ is not an isomorphism of groups.
        \item[2.] Whenever $E=E'$ and $\phi=[m]$ with $p\nmid m$, we automatically have that $\phi|_{E[p]}$ is symplectic as $n=m^2$.
    \end{itemize}
\end{remark}

\subsection{The curve \texorpdfstring{$X(p)$}{}}\label{sec:Xp}
Let $K$ be a field of characteristic $0$.

The following is one of the various equivalent descriptions of $X(p)$.
The modular curve $X(p)$ is a smooth projective, geometrically irreducible curve defined
over~$\Q$ with the following property. It is the compactification of the
modular curve~$Y(p)$ whose $K$-points classify
(equivalence classes of) pairs $(E', \phi)$ such that $E'/K$ is an elliptic curve and
$\phi : \Z/p\Z \times \mu_p \to E'[p]$ is a strictly symplectic isomorphism,
where $\Z/p\Z \times \mu_p$ is viewed as a $K$-group scheme with the standard pairing
$((a,\zeta),(c,\xi)) \to \xi^a \zeta^{-c}$; see~\cite[\S 1.5.3.4]{Halberstadt} for details. 

The set $\cC$ of cusps of $X(p)$ has $\frac{1}{2}(p^2-1)$ elements and can be identified with 
\[\cC \cong \left\{\begin{pmatrix} u \\ v \end{pmatrix} \in (\Z/p\Z)^2 : (u, v) \neq (0,0) \right\} / \{ \pm 1 \};\]
see for example \cite[\S2.8.2]{Halberstadt}.
Moreover, Th\'eor\`eme 13 in \cite[p. 302]{Halberstadt} describes the Galois action of an element $\sigma \in G_\Q$ on $\cC$  as 
\begin{equation}\label{E:actionOnC}
  \begin{pmatrix} u \\ v \end{pmatrix}^\sigma = \begin{pmatrix}
				1 & 0 \\
				0 & \chi_p(\sigma)^{-1}
			\end{pmatrix}\begin{pmatrix} u \\ v \end{pmatrix}.
\end{equation}

\subsection{The twists \texorpdfstring{$X_E^\pm(p)$}{}}\label{sec:XEp}
Let $E/\Q$ be an elliptic curve, and $K$ a field of characteristic 0.

There is a smooth projective curve $X_E^+(p)$ defined over $\Q$ which is the twist of $X(p)$ with the following property. Let $\cC^+$ be the set of cusps of~$X_E^+(p)$. The $K$-points
of $Y_E^+(p) := X_E^+(p) \setminus \cC^+$
parametrize (equivalence classes of) pairs $(E', \phi)$ consisting of an elliptic curve~$E'/K$
and a strictly symplectic isomorphism $\phi : E[p] \to E'[p]$ of $G_K$-modules;
see~\cite[\S 5.3]{Halberstadt} for details.

There is also a smooth projective curve $X_E^-(p)$ over $\Q$ which is a twist of $X(p)$ with the following properties. Let $t  : X_E^-(p) \to X(p)$ be the $\Qbar$-isomorphism defined
in \cite[p.459]{Halberstadt}. The set of cusps of~$X_E^-(p)$ is
$\cC^- = t^{-1}(\cC)$ and the $K$-points of $Y_E^-(p) := X_E^-(p) \setminus \cC^-$
parametrize (equivalence classes of) pairs $(E', \phi)$ consisting of an elliptic curve~$E'/K$
and a strictly anti-symplectic isomorphism $\phi : E[p] \to E'[p]$ of $G_K$-modules (with respect to the fixed non-square~$r_p\in \F_p^\times$);
see~\cite[\S 5.4]{Halberstadt} for details. Moreover, in \cite[p. 459]{Halberstadt} 
the Galois action on the cusps is described as 
\begin{equation}\label{E:cuspAction}
  t(P^{\sigma^{-1}}) = (\tilde{g}_\sigma(t(P)))^{\sigma^{-1}},
\end{equation}
for all $\sigma \in G_\Q$ and $P \in \cC^-$. Here $\tilde{g}_\sigma \in \Aut(X(p))$ is induced by
 $g_\sigma \in \SL_2(\F_p)$ given by 
\[
 g_\sigma = \begin{pmatrix}
				1 & 0 \\
				0 & \chi_p(\sigma)^{-1} r_p^{-1}
			\end{pmatrix} \rhobar_{E,p}^T(\sigma)
			 \begin{pmatrix}
				1 & 0 \\
				0 &  r_p
			\end{pmatrix},
\]
where
$\rhobar_{E,p}(\sigma) \in \GL_2(\F_p) \simeq \GL(E[p])$ after fixing a symplectic basis of $E[p]$; see~\cite[p. 458]{Halberstadt} for details. We note that the Galois action on the cusps of $X_E^+(p)$ can be described by \eqref{E:cuspAction} with $r_p=1$ (see \cite[\S 5.3.2]{Halberstadt}).

\begin{proposition}\label{prop:cuspsQ}
Let $p \geq 3$ be a prime. Let $K$ be a field of characteristic $0$ and $E/K$ an elliptic curve. 
Then $X_E^\pm(p)$ has a $K$-rational cusp if and only if $\rhobar_{E,p} \simeq
			\left(\begin{smallmatrix}
				\epsilon \chi_p & * \\
				0 & \epsilon
			\end{smallmatrix} \right) \subset \GL_2(\F_p)$, where 
            $\epsilon$ is a character of order dividing 2. Moreover, if $K=\Q$ then $p \leq 7$.
\end{proposition}
\begin{proof}
     Suppose $P \in \cC^\pm$ is defined over $K$, that is $P^\sigma = P$ for all $\sigma \in G_K$. Equivalently, for all $\sigma \in G_K$, the equation \eqref{E:cuspAction} becomes
$$t(P) = (\tilde{g}_\sigma(t(P)))^{\sigma^{-1}}.$$
We denote $t(P) := (u,v)^T \in \cC^\pm$. From~\eqref{E:actionOnC} and the formula for $g_\sigma$, we rewrite the above as 
\[ \begin{pmatrix} u \\ v \end{pmatrix} =
 \left(\begin{matrix}
				1 & 0 \\
				0 & \chi_p(\sigma)
			\end{matrix}\right)
  \left(\begin{matrix}
				1 & 0 \\
				0 & \chi_p(\sigma)^{-1} r_p^{-1}
			\end{matrix}\right) \rhobar_{E,p}^T(\sigma)
			 \left(\begin{matrix}
				1 & 0 \\
				0 &  r_p
			\end{matrix}\right)\begin{pmatrix} u \\ v \end{pmatrix},
\]
which simplifies to
\[
			\begin{pmatrix} u \\ r_p v \end{pmatrix} =
             \rhobar_{E,p}^T(\sigma)
			  \begin{pmatrix} u \\ r_p v \end{pmatrix}.
\]
Since $(u,v)^T$ is defined modulo $\{ \pm 1 \}$, the previous equality means that there is a (possibly trivial) character $\epsilon : G_K \to \{\pm 1\}$ such that $\rhobar_{E,p}^T\otimes \epsilon$ has an eigenvector with eigenvalue $1$. Equivalently, there is a basis of $\F_p \times \F_p$ such that
$$\rhobar_{E,p}^T \otimes \epsilon =
			\begin{pmatrix}
				1 & * \\
				0 & \chi_p
			\end{pmatrix} \subset \GL_2(\F_p),$$
because $\det\: \rhobar_{E,p} = \chi_p$. 
Now transposing and swapping the basis vectors yields $\rhobar_{E,p} =
			\left(\begin{smallmatrix}
				\epsilon \chi_p & * \\
				0 & \epsilon
			\end{smallmatrix} \right)$, as desired.
This description for $\rhobar_{E,p}$ implies that $E$ has a quadratic twist that is $p$-isogenous to an elliptic curve
having a $p$-torsion point. Therefore, if $K=\Q$ we have $p\leq 7$ by Mazur's classification of torsion subgroups for rational elliptic curves.
\end{proof}

\begin{remark}\label{rem:infinitepoints}
    We note that for any $\ell$-adic field $K$, one has that $X_E^\pm(p)(K)=\emptyset$ if and only if $Y_E^\pm(p)(K)=\emptyset$. Indeed, for $K$ an $\ell$-adic field and $X$ a smooth curve with a $K$-point $P$, then $X$ contains an $\ell$-adic analytic disk around $P$, so, in particular, it will have uncountably many $K$-points. The conclusion now follows from the fact that the set of cusps $\cC^\pm$ is finite.
\end{remark}

Note that taking quadratic twists by~$u \in K$ of the pairs $(E', \phi)$
induces canonical isomorphisms $X^+_{E^{u}}(p) \simeq X^+_E(p)$ and
$X^-_{E^{u}}(p) \simeq X^-_E(p)$. In particular, the following lemma allows us to study Question~\ref{Q} up to twisting $E$.

\begin{lemma}\label{lem:qtwists}
Let $u \in K$ be a non-square and $E^{u} / K$ be the quadratic twist of~$E$ by~$u$.
Then $X_E^{-}(p)(K)\neq \emptyset$ if and only if $X_{E^{u}}^{-}(p)(K) \neq \emptyset$.
 \end{lemma}
 \begin{proof} This is \cite[Lemma~11]{symplectic} together with an easy application of Proposition~\ref{prop:cuspsQ}.
  \end{proof}

\subsection{The curve $X_E^-(p)$ in positive characteristic}
The definitions of $X(p)$ and $X_E^\pm(p)$ given in Sections~\ref{sec:Xp} and~\ref{sec:XEp} can be extended to work for an elliptic curve $E$ over $K=\F_\ell$ where $\ell \neq p$ and this is formalized in \cite[Corollary 1 (a)--(b)]{elie}. Moreover, the following proposition assures that for a prime of good reduction $\ell$, reducing modulo $\ell$ commutes with the moduli problem. 
\begin{proposition}\label{prop:elie}
Let $E/\Q$ be an elliptic curve, and $p\geq 7$ a prime.
The curves $X(p)$ and $\Xp$ have good reduction at all primes $\ell \neq p$ and $\ell \nmid pN_E$, respectively, and are geometrically irreducible over $\F_\ell$. Moreover, if $\oE$ denotes the special fiber $E/\F_\ell$, then the special fiber $X_E^-(p)/\F_\ell$ is $\F_\ell$-isomorphic to $X_{\oE}^-(p)/\F_\ell$.    
\end{proposition}
\begin{proof}
The proof follows from \cite[Corollary 1]{elie} with $G$ taken to be the group scheme $E[p]$.
\end{proof}

\begin{corollary}\label{rmk}
Let $E/\Q$ be an elliptic curve. Let $p\geq 7$ and $\ell \nmid pN_E$ be two primes. 
We have $X_E^-(p)(\Q_\ell) = \emptyset$ if and only if $X_{\oE}^-(p)(\F_\ell) = \emptyset$. Moreover, the latter occurs if and only if the following two conditions hold:  

(i) The curve $X_{\oE}^-(p)$  has no $\F_\ell$-rational cusps, or equivalently, the representation~$\rhobar_{E,p}$ is not conjugate 
in~$\GL_2(\F_p)$  to the form 
$\left(\begin{smallmatrix}
				\epsilon \chi_p & * \\
				0 & \epsilon
			\end{smallmatrix} \right)$ 
            with $\epsilon$ a character of $G_{\Q_\ell}$ or order at most 2;

(ii) for all pairs $(E',\phi)$, where $E'/\F_\ell$ is an elliptic curve and $\phi: \oE[p] \to
E'[p]$ is an isomorphism of $G_{\F_\ell}$-modules, we have that $\phi$ is symplectic, that is, $Y_{\oE}^-(p)(\F_\ell) = \emptyset$.
\end{corollary}
\begin{proof}
Using Proposition~\ref{prop:elie} and $\ell \nmid pN_E$, one gets that the curve $X_E^-(p)$ has good reduction at~$\ell$ and its reduction is $X_{\overline{E}}^-(p)/\F_\ell$. Moreover, Hensel's lifting (e.g. \cite[Lemma 1.1]{Jordan}) assures that $X_E^-(p)(\Q_\ell) = \emptyset$ if and only if $X_{\oE}^-(p)(\F_\ell) = \emptyset$.

To prove (i), note that the cuspidal subscheme $\mathcal{C}^-$ is finite \'etale over $\Z_\ell$ by \cite[Theorem 3]{elie}. Therefore, Hensel's lifting implies that there are no $\F_\ell$-rational cusps if and only if there are no $\Q_\ell$-rational cusps, and the proof concludes by Proposition~\ref{prop:cuspsQ}.

Finally, (ii) follows directly from the moduli interpretation of the $\F_\ell$-points of $Y_{\oE}^-(p)$.
\end{proof}

\begin{corollary} \label{cor:refereeRequest}
Let $\ell$ and $p \geq 3$ be different primes. Let $E/\Q_\ell$ be an elliptic curve with good reduction such that $p$ divides the order of $\rhobar_{E,p}(\Frob_\ell)$. Then $X_{\overline{E}}^-(p)$ has a
cusp defined over~$\F_\ell$ if and only if $\ell \equiv 1 \pmod{p}$. 
\end{corollary}
\begin{proof} By \cite[Lemma 8]{symplectic}, the assumption that $p$ divides the order of $\rhobar_{E,p}(\Frob_\ell)$ is equivalent to 
$\rhobar_{E,p} \simeq \left(\begin{smallmatrix}
				\delta & * \\
				0 & \delta
			\end{smallmatrix} \right) \subset \GL_2(\F_p)$  
where $\delta : G_{\Q_\ell} \to \F_p^\times$ is a character. On the other hand, from Corollary~\ref{rmk}, there is an $\F_\ell$-rational cusp of $X_{\overline{E}}^-(p)$ if and only if
$\rhobar_{E,p} \simeq \left(\begin{smallmatrix}
				\epsilon \chi_p & * \\
				0 & \epsilon
			\end{smallmatrix} \right)$.  
These two descriptions for $\rhobar_{E,p}$ are compatible if and only if $\chi_p = 1$, or equivalently, $\ell \equiv 1 \pmod{p}$. 
\end{proof}

We end this section with an observation on the notation that will be used multiple times.
Let $K$ be any field of characteristic $\neq p$. Note that for $E/K$ an elliptic curve, $X_E^+(p)(K)$ always contains the canonical point $(E, \id_{E[p]})$, while $\Xp(K)$ may be empty.
Let $\phi:E[p]\to E'[p]$ be a $G_K$-isomorphism. 
To simplify notation, we will say that a pair $(E',\phi)$ gives rise to a non-cuspidal point on $X_E^\pm(p)(K)$ if $(E', \lambda \cdot \phi)$ is a point on $X_E^\pm(p)(K)$, where $\lambda \in \F_p^\times$ is the square root of~$1/r(\phi)$ for the plus curve and the square root of~$r_p / r(\phi)$ for minus curve. Note that when $\phi$ is the restriction of an isogeny of degree coprime to $p$, then standard properties of the Weil pairing give that $r(\phi)=\deg(\phi)$.

We highlight the following case for easy future reference.

\begin{corollary}\label{cor:isogeny} Let $K$ be a field of characteristic $\neq p$.
    Suppose that $E/K$ has a $K$-isogeny of prime degree $q$ such that $(q/p)=-1$. Then 
    $Y_E^-(p)(K) \neq \emptyset$, and so $\Xp(K)\neq \emptyset$.
\end{corollary}

\section{Preliminary results}

\subsection{The case \texorpdfstring{$p=3,5$}{}}\label{sec:p=3,5}
Recall that $X(p)$ has genus zero for $p=3,5$.

\begin{theorem}\label{thm:p=3}
Let $K$ be any field of characteristic $0$ and $E/K$ be an elliptic curve.
If $p=3$ or $p=5$, then $X_E^{-}(p)(K)\neq \emptyset$.
\end{theorem}
\begin{proof}
Since $\Xp$ is a twist of $X(p)$, its genus is 0 for $p=3,5$. In fact, $\Xp$ is isomorphic to $\PP^1_K$; 
an explicit parametrization $\mathbb P^1(K)\to   X_E^{-}(p)(K)$ for $p=3$ can be found in \cite[\S 13]{fisher} and for $p=5$ in \cite[Theorem 5.8]{fisherp5}. In particular, $X_E^{-}(p)(K)\neq \emptyset$.
\end{proof}

\subsection{Real Points} \label{sectionreal}
In this section, we use the uniformization of real elliptic curves to guarantee the existence of real degree~$q$ isogenies for any prime $q$.
	\begin{theorem}\label{thm:real}
		Let $E/\R$ be an elliptic curve and $p \geq 3$ a prime. Then
		$\Xp(\R)\neq \emptyset$.
	\end{theorem}
	\begin{proof}
	From ~\cite[Ch V, Cor~2.3.1]{silverman}, there is an isomorphism of  real Lie groups
		\begin{equation*}
			E(\R)\cong
			\begin{cases}
				\R/\Z, \text{ if }\Delta(E)<0,\\
				\R/\Z \times \Z/2\Z, \text{ if } \Delta(E)>0.
			\end{cases}
		\end{equation*}
		where $\Delta(E)$ is the discriminant of a model of~$E$.
Let $q>2$ be a prime such that~$(q/p)=-1$. 
There is a $q$-torsion point $P \in E(\R)$ (the class of $1/q$ in $\R/\Z$), 
therefore $E \to E / \langle P \rangle$ is an isogeny of degree $q$ defined over $\R$, hence $\Xp(\R) \neq \emptyset$ by Corollary~\ref{cor:isogeny}.
	\end{proof}

\subsection{Potentially multiplicative reduction}
Let $\ell$ be a prime and $E/\Q_\ell$ an elliptic curve with potentially multiplicative reduction.
Similar to the real case, such elliptic curves have an $\ell$-adic uniformization, ensuring the existence $\Q_\ell$-isogenies of any prime degree~$q$.
\begin{theorem}\label{thm:multiplicative}
	Let $\ell$ and $p \geq 3$ be two (not necessarily different) primes.
	Let $E/\mathbb{Q}_\ell$ be an elliptic curve with potentially multiplicative reduction. Then $X_E^{-}(p)(\Q_\ell)\neq \emptyset$.
	\end{theorem}
	\begin{proof}
	By the theory of the Tate curve, for all primes $q$, we can conjugate $\rhobar_{E,q}$ into the Borel subgroup of $\GL_2(\F_q)$. More precisely, we have
	\begin{equation*}
			\rhobar_{E,q} \simeq
			\begin{pmatrix}
				\epsilon \chi_q & * \\
				0 & \epsilon
			\end{pmatrix} \subset \GL_2(\F_q),
		\end{equation*}
where $\epsilon$ is character of order dividing $2$ and $\chi_q$ is the mod~$q$ cyclotomic character. Thus there is a degree~$q$ isogeny
 $\phi : E/\Q_\ell \to E'/ \Q_\ell$. By taking a prime $q \neq p$ such that $(q/p) = -1$ the result follows from Corollary~\ref{cor:isogeny}.
\end{proof}

\subsection{Locally abelian Galois image}\label{sec:ab}
In view of Theorem~\ref{thm:multiplicative}, from now on, we only have to focus on elliptic curves $E/\Q_\ell$ with potentially good reduction.
In~\cite{symplectic}, it is shown that if the image of $\rhobar_{E,p}$ is abelian, then any matrix belonging to it with  non-square determinant gives an anti-symplectic automorphism $\phi$ of $E[p]$, and hence a point $(E,\phi)\in\Xp(\Q_\ell)$. 

\begin{theorem}
\label{thm:abelian}
Let $\ell$ and~$p \geq 7$ be different primes. Let $E/\Q_\ell$ be an elliptic curve with potentially good reduction 
such that $G:=\rhobar_{E,p}(G_{\Q_\ell})$ is abelian. If $E/\Q_\ell$ or a quadratic twist of~$E/\Q_\ell$ has good reduction, assume additionally that
$p \nmid \# G$.
Then $\Xp(\Q_\ell) \neq \emptyset$.
\end{theorem}
\begin{proof}
By~\cite[Lemma~6]{symplectic}, if there exists $M \in C_{\GL_2(\F_p)}(G)$ whose determinant is not a square modulo~$p$, then $M$ gives rise to an anti-symplectic automorphism $\phi_M: E[p] \to E[p]$ of $G_\Q$-modules,
giving a point $(E,\phi_M) \in X_E^{-}(p)(\Q_\ell)$, as desired.

Assume by contradiction that all matrices in $C_{\GL_2(\F_p)}(G)$ have square determinant, then $p \mid \#G$ by~\cite[Lemma~8]{symplectic}.
Now, since $p \geq 7$, it follows from \cite[Proposition 3]{symplectic} that either~$E/\Q_\ell$ or a quadratic twist of it has good reduction, contradicting the hypothesis.
\end{proof}

\subsection{Large primes of good reduction}
Recall from Proposition~\ref{prop:elie} that 
the curves $X_E^-(p)$ and $X(p)$ have good reduction at 
all primes $\ell \nmid p N_E$ and $\ell \neq p$, respectively, in which case they are geometrically irreducible over $\F_\ell$.

\begin{theorem}\label{thm:Hensel} Let $E/\Q$ be an elliptic curve.
    Let $p\geq 3$ be a prime and $g$ be the genus of~$X(p)$.
    Suppose that $\ell>4g^2$ is a prime of good reduction for $X_E^{-}(p)$.
    Then $X_E^{-}(p)(\Q_\ell)\neq \emptyset$.
\end{theorem}
\begin{proof} For a geometrically irreducible smooth projective curve $C / \F_\ell$ of genus~$g$, recall the Hasse--Weil bound $|\#C(\F_\ell) - (\ell+1) | \leq 2g \sqrt{\ell}$.
Note that $\ell > 4g^2$ implies $\ell > 2g\sqrt{\ell}$. Let $C := X_E^{-}(p) / \F_\ell$ for $\ell$ a prime of good reduction for $X_E^{-}(p)$. 
If $\# C(\F_\ell) = 0$, then $\ell + 1 \leq 2g\sqrt{\ell}$, contradicting the Hasse--Weil bound.
Hence $\# X_E^{-}(p)(\F_\ell) > 0$ and Corollary~\ref{rmk} completes the proof.
\end{proof}
\begin{remark}\label{rmk:genus}
    The genus $g$ of $X_E^{-}(p)$ is the same as that of $X(p)$ as they are twists, and it is given by $g=1+\frac{1}{24}(p^2-1)(p-6)$ (see~\cite[p. 77]{Halberstadt}).
\end{remark}
\vspace{-2mm}
\section{The case of good reduction}\label{sec:goodred}

Let $\ell$ be a fixed prime and $k/\F_\ell$ be a finite field with $\#k=\ell^f$. Let $E/k$ be an elliptic curve. We denote by 
\[
a_k(E):=\ell^f+1-|E(k)|, \qquad \Delta_k:=a_k(E)^2-4\ell^f.
\]
The Hasse--Weil bound $|a_k(E)|\leq2\sqrt{\ell^f}$ gives~$\Delta_k \leq 0$.  
Whenever $k=\F_\ell$, we abbreviate notation to $a_{\ell}$, and $\Delta_{\ell}$. 
Note that $\Delta_\ell<0$ as $f=1$. 

Suppose that $E/k$ is ordinary, that is, $a_k(E) \not\equiv 0 \pmod{\ell}$. 
The Frobenius endomorphism $\pi\in \End(E/k)$  satisfies the equation $\pi^2-a_k(E) \:\pi+\ell^f=0$. The field $K :=\Q(\sqrt{\Delta_k})$ is imaginary quadratic, and it is well known that $\End(E/k) \simeq \calO_{E}$ is an order in the ring of integers $\cO_K$ containing $\Z[\pi]$. Moreover, we set $b_E:=[\cO_E: \Z[\pi]]$ and hence
\begin{equation}\label{discriminats}
      \Delta_k = \Delta(\Z[\pi])=\Delta(\cO_E)b_E^2.
\end{equation}
Suppose now that $F/\Q_\ell$ is a finite extension with residue field $k$.
Let $E/F$ be an elliptic curve with good reduction and $q \neq \ell$ a prime. By the criterion of N\'eron–Ogg–Shafarevich, the inertia subgroup $I_F \subset G_F$ acts trivially on~$E[q]$. In particular, the representation
$\rhobar_{E,q}$
satisfies $$\rhobar_{E,q}(G_{F})= \langle \rhobar_{E,q}(\Frob_F) \rangle\subset \GL_2(\F_q).$$ 
In \cite[IV, \S 1.3]{Serre1} we are given the trace and determinant of the action of $\Frob_F$ on the $q$-adic Tate module. Reducing it modulo $q$, one gets
that the characteristic equation for $\rhobar_{E,q}(\Frob_F)$ is given by
\begin{equation}\label{Frobeq}
		t^2- a_F(E)t+\ell^f=0, \quad\text{where } a_F(E):=a_k(\Ebar/k),
\end{equation}
where $\Ebar/k$ is the reduction of a minimal model for $E/F$.
Let also $\Delta_F:=\Delta_k$ be the discriminant of this quadratic equation.
The main result of this section is the following theorem.
\begin{theorem}\label{thm:maingood}
Let $\ell$ and $p \geq 7$ be different primes.
Let $E/\Q_\ell$ be an elliptic curve with good reduction.
Then $X_E^-(p)(\Q_\ell)=\emptyset$ if and only if all of the following hold
\begin{enumerate}
\item $p \equiv 3 \pmod{4}$; \label{g1}
 \item $-p\Delta_\ell=s^2$ for some $s \in \Z$;\label{g2}
\item $p\mid \#\rhobar_{E,p}(\Frob_\ell)$ or, equivalently, $p\mid \Delta_\ell$ and $p\nmid b_E$;\label{g3}
\item  for all primes $q\neq \ell, \:q \mid \Delta_\ell \Rightarrow (q/p)\neq -1$; \label{g4}
\item $|a_\ell(E) |< p-2\sqrt{\ell}$; \label{g5}
\item $\ell \not\equiv 1 \pmod{p}$. \label{g6} 
\end{enumerate} 
\end{theorem}

\begin{remark} \label{rmk:good}
Condition \eqref{g2} implies that $\oE$ is ordinary. Indeed, $\ell\mid a_\ell$ if and only if $\ell$ divides $\Delta_\ell=a_\ell^2-4\ell$, in which \eqref{g2} gives $\ell^2\mid \Delta_\ell$. This in turn implies that $\ell^2\mid 4\ell$, a contradiction for $\ell \geq 3$, and so $a_\ell \not\equiv 0 \pmod{\ell}$, as desired.
    Moreover, if $\ell=2$ and $2\mid a_2$ then $a_2\in \{-2,0,2\}$ by the Hasse--Weil bound. In particular $\Delta_2 \in \{-4,-8\}$, contradicting \eqref{g2}. Note that \eqref{g2} implies also $p\mid \Delta_\ell$, and so condition \eqref{g3} is reduced to $p\nmid b_E$.
    Moreover, since $\oE$ is ordinary, it is not special in the sense of \cite[Definition~1]{Centeleghe}, therefore $b_E = \beta_\ell$ where~$\beta_\ell$ is defined in Theorem~2 of {\it op. cit.}, and the equivalence in condition~\eqref{g3} follows from~\cite[Corollary 4]{symplectic}. Finally, note also that if $\ell < p^2 / 16$ then condition~\eqref{g5} holds, as $p - 2 \sqrt{\ell} > p/2 > 2\sqrt{\ell}$.
\end{remark}
We start by describing when we can construct points on $\Xp(\Q_\ell)$
using isogenies (i.e. via Lemma~\ref{lem:isogeny}).
The following is a necessary and sufficient condition for $E$ to have a $\Q_\ell$-isogeny of degree $q$.
\begin{lemma}\label{good}
		Let $E/\Q_\ell$ be an elliptic curve with good reduction. Let $q \neq \ell$ be a prime.
		Then~$E$ admits a $\Q_\ell$-isogeny of degree~$q$ if and only if $(\Delta_{\ell}/q)\in \{0, 1\}$.
	\end{lemma}
	\begin{proof}
The curve $E$ admits a $\Q_\ell$-isogeny of degree~$q \neq \ell$ if and only if 
$$\rhobar_{E,q}(G_{\Q_\ell})= \langle \rhobar_{E,q}(\Frob_{\ell}) \rangle \subset \GL_2(\F_q)$$ can be conjugated into a subgroup of the Borel subgroup of $\GL_2(\F_q)$. This occurs if and only if the characteristic polynomial of
		$\rhobar_{E,q}(\Frob_{\ell})$ factors, that is $(\Delta_{\ell}/q)\in \{0, 1\}$.
\end{proof}
\begin{proposition}\label{prop:good1}
Let $\ell$ and $p \geq 3$ be two (not necessarily different) primes and $E/\Q_\ell$ an elliptic curve with good reduction.
Assume that at least one of the following holds:
\begin{enumerate}
\item $p\equiv 1 \pmod 4$;
 \item $-p\Delta_\ell$ is not a square in $\Z$;
\item  there exists a prime $q\neq \ell$ such that $q\mid \Delta_\ell$ and $(q/p)=-1$.
    
\end{enumerate}
Then $Y_E^-(p)(\Q_\ell) \neq \emptyset$ and so $X_E^{-}(p)(\Q_\ell)\neq \emptyset$.
\end{proposition}
\begin{proof}
Let $E/\Q_\ell$ be as in the statement.
Recall that $\Delta_{\ell} < 0$. 

We claim that  one can always find a prime $q \neq \ell$ such that $(\Delta_{\ell}/q) \in \{0,1\}$ and $(q/p)=-1$. Thus, by
Lemma~\ref{good}, we conclude that $E/\Q_\ell$ admits
a $\Q_\ell$-isogeny of degree~$q$ with $(q/p)=-1$ and the conclusion follows from Corollary~\ref{cor:isogeny}.

For $p \neq q$ odd primes, we
recall the quadratic reciprocity laws
\[
 \left(\frac{p}{q}\right)=(-1)^{\frac{(p-1)(q-1)}{4}} \left(\frac{q}{p}\right) \quad \text{ and } \quad \left(\frac{-1}{q}\right) = (-1)^{\frac{q-1}{2}}.
\]

We will now prove the claim. 

If assumption (3) holds, then the claim follows because $(\Delta_\ell / q) = 0$.

Suppose (1) holds, that is  $p \equiv 1 \pmod{4}$. We consider the Galois field $F=\Q(\sqrt{\Delta_{\ell}},\sqrt{p})$.

Since $p\Delta_{\ell}<0$ the extension $F/\Q$ is of degree 4  and there is $\sigma \in \Gal(F/\Q)$ satisfying
\[ \sigma(\sqrt{\Delta_\ell}) = \sqrt{\Delta_\ell}
 \quad \text{ and } \quad \sigma(\sqrt{p}) = - \sqrt{p}.
\] 
By Chebotarev's Density Theorem, there is a positive density of primes $q \nmid p\ell \Delta_\ell$ such that $\Frob_q$ acts on~$F/\Q$ as~$\sigma$, that is, $(\Delta_{\ell}/q)=1$ and $(p/q)=-1$; by quadratic reciprocity we also have $(q/p)= -1$, as desired.

Assume now that (2) holds and (1) does not, that is, $p \equiv 3 \pmod{4}$.

Consider the Galois field $F:=\Q(\sqrt{\Delta_{\ell}},\sqrt{p}, \sqrt{-1})$ whose degree is $\geq 4$.
It follows from assumption (2) that there are two cases:

(i) $F$ is of degree 8;

(ii) $F=\Q(\sqrt{p}, \sqrt{-1})$ and
$\Delta_\ell = - s^2$ with $s \in \Z$.

Assume we are in case (i). As above, there is $\sigma \in \Gal(F/\Q)$ satisfying
\[ \sigma(\sqrt{\Delta_\ell}) = \sqrt{\Delta_\ell},
\quad \sigma(\sqrt{p}) = - \sqrt{p}  \quad \text{ and } \quad  \sigma(\sqrt{-1}) =   \sqrt{-1}.
\]
By Chebotarev's Density Theorem, there is a positive density of primes $q \nmid p \ell \Delta_\ell$ such that $\Frob_q$ acts on~$F/\Q$ as~$\sigma$, that is, $(\Delta_{\ell}/q)=1$, $(p/q)=-1$ and $(-1/q) = 1$; by the quadratic reciprocity laws this yields $q \equiv 1 \pmod 4$ and $(p/q)=(q/p) = -1$ as desired.
 
Assume we are in case (ii). There is $\sigma \in \Gal(F/\Q)$ satisfying
\[
\quad \sigma(\sqrt{p}) = - \sqrt{p}  \quad \text{ and } \quad  \sigma(\sqrt{-1}) =   \sqrt{-1}.
\]
As above, there is a positive density of primes $q \nmid p \ell \Delta_\ell$ such that $\Frob_q$ acts on~$F/\Q$ as~$\sigma$, that is, $(-1/q)=1$, $(p/q)=-1$; by quadratic reciprocity this yields $q \equiv 1 \pmod 4$ and $(p/q)=(q/p) = -1$; finally,
since $\Delta_\ell = - s^2$, we have $(\Delta_\ell/q) = (-1/q) = 1$.
\end{proof}

\begin{remark} \label{rem:isogenies}
Note that, for $E/\Q_\ell$ with good reduction, the results proved so far imply that, unless conditions (1)--(4) of Theorem~\ref{thm:maingood} are satisfied, there is a $\Q_\ell$-point in $Y_E^-(p)$ induced by an isogeny of $E/\Q_\ell$ of prime degree $q$ such that $(q/p)=-1$.
\end{remark}

Recall that, given an elliptic curve $E/\Q_\ell$ with good reduction, $\oE$ denotes the reduction of a minimal model of $E/\Q_\ell$.
We know from Corollary~\ref{rmk} that 
$\Xp(\Q_\ell)$ is empty if and only if $X^-_{\oE}(p)(\F_\ell)$ is empty. We will now study the points on $X^-_{\oE}(p)(\F_\ell)$ 
in terms of~$\End(\oE/\F_\ell)$.

\begin{lemma}\label{lem:samesym} Let $\ell$ and $p \geq 3$ be different primes.
    Let $E / \F_\ell$ and $E' / \F_\ell$ be elliptic curves with $G_{\F_\ell}$-isomorphic $p$-torsion modules $E[p]\simeq E'[p]$.
    Suppose that~$p \mid \#\rhobar_{E,p}(\Frob_\ell)$.
    Then, the $G_{\F_\ell}$-modules isomorphisms $\phi : E[p] \to E'[p]$ are either all symplectic or all anti-symplectic. 
\end{lemma}
\begin{proof}
    This follows from Corollary 4, Proposition 5, and Lemma 6 in \cite{symplectic}.
\end{proof}
\begin{lemma}\label{lem:Centeleghe}
    Let $E / \F_\ell$ an ordinary elliptic curve with $p \mid \Delta_\ell$. Then there exists a suitable choice of basis for $E[p]$ such that
    \begin{equation}\label{eq:cent}
    \rhobar_{E,p}(\Frob_\ell)=\left(\begin{smallmatrix}
				 a_\ell/2 & 0 \\
				b_E & a_\ell/2\end{smallmatrix}\right) \in \GL_2(\F_p). \end{equation}
\end{lemma}
\begin{proof}
    This follows by reducing modulo~$p$ the matrix given by~\cite[Theorem 2]{Centeleghe}, together with the fact that~$\beta_E=b_E$ (see the end of Remark~\ref{rmk:good}) and $\Delta_\ell / b_E$ is divisible by~$p$.
\end{proof}

\begin{lemma} \label{lem:diagfrob}
    Let $\ell$ and $p \geq 3$ be different primes. Let $a$ be an integer such that $|a|\leq 2\sqrt{\ell}$, $\ell \nmid a$ and $a^2\equiv4\ell \pmod p$. There exists an elliptic curve $E/\F_\ell$ and a suitable choice of basis for $E[p]$ such that
    \begin{equation}\label{eq:Frob}
    \rhobar_{E,p}(\Frob_\ell)=\left(\begin{smallmatrix}
				 a/2 & 0 \\
				b & a/2\end{smallmatrix}\right) \in \GL_2(\F_p)\qquad \text{ with } \qquad b \neq 0.
    \end{equation}
\end{lemma}
\begin{proof}
    By \cite[Theorem 4.1]{Waterhouse}, there is an elliptic curve $E/\F_\ell$ with $a_\ell(E):=a$, which is ordinary as $\ell \nmid a$. By assumption we have $\Delta_\ell(E):=a^2-4\ell \equiv 0 \pmod p$, and hence $\rhobar_{E,p}(\Frob_\ell)$ has eigenvalues
    $\lambda_1=\lambda_2\equiv\frac{a}{2} \pmod{p}$.
If $p \nmid b_E$, then $E$ satisfies~\eqref{eq:Frob} by Lemma~\ref{lem:Centeleghe}.
            Suppose now that $p\mid b_E$.
            By \cite[Theorem 7]{Drew}, there is a chain of $v_p(b_E)$ descending isogenies of degree~$p$ from $E$ to an elliptic curve $E'$ satisfying $p \nmid b_{E'}$  (i.e. $E'$ belongs to the level $V_d$ in the notation of \textit{loc. cit.}). Moreover, $a_\ell(E')=a_\ell(E)=a$ since the two curves are isogenous, and thus Lemma~\ref{lem:Centeleghe} gives that $E'$ satisfies \eqref{eq:Frob}.
\end{proof}

\begin{lemma}\label{lem:eqEndnew}
     Let $\ell$ and $p \geq 3$ be different primes. Let $E/\F_\ell$ and $E'/\F_\ell$ be two ordinary elliptic curves
     with endomorphism rings $\cO_E$ and
     $\calO_{E'}$, respectively.
     Suppose there exists a cyclic isogeny $\varphi : E \to E'$ of degree $p^k$, with $k\in\Z_{\geq 0}$ and $E[p] \cong E'[p]$ as $G_{\F_\ell}$-modules.
    Assume further that $p \nmid b_E$. Then $\cO_E=\cO_{E'}$.
\end{lemma}
\begin{proof}
If $k=0$, then $\deg(\varphi)=1$, so $\varphi$ is an isomorphism and the
conclusion is immediate. We may therefore assume that $k\geq 1$.
As $E$ and $E'$ are isogenous, we have
\[
    a_\ell:=a_\ell(E)=a_\ell(E')
    \qquad\text{and}\qquad
    \Delta_\ell:=\Delta_\ell(E)=\Delta_\ell(E').
\]
 Therefore, the
characteristic equations of the Frobenius endomorphisms are the same, and
hence
\[
    \Z[\pi]=\Z[\pi'].
\]

Suppose, for a contradiction, that $\cO_E\neq\cO_{E'}$. By iterating
\cite[Proposition~21]{Kohel} along the $p$-power cyclic isogeny, the index of one order in the other is $p^r$ for some integer
$r\geq 1$. 

The inclusion $\cO_{E'}\subseteq\cO_E$ is impossible, since it would give
\[
    b_E=[\cO_E:\Z[\pi]]
        =[\cO_E:\cO_{E'}][\cO_{E'}:\Z[\pi]]
        =p^r b_{E'},
\]
contrary to the assumption that $p\nmid b_E$. Therefore,
\[
    \Z[\pi]=\Z[\pi']
    \subseteq \cO_E
    \subsetneq \cO_{E'},
\]
and
\[
    b_{E'}=[\cO_{E'}:\Z[\pi]]
           =[\cO_{E'}:\cO_E][\cO_E:\Z[\pi]]
           =p^r b_E.
\]
In particular, $p\mid b_{E'}$. It follows from~\eqref{discriminats} that
$p\mid\Delta_\ell$.

  Since $E$ and $E'$ are ordinary, Lemma~\ref{lem:Centeleghe} implies that under some choices of bases 
   \begin{equation}\label{E:FrobE}
    \rhobar_{E,p}(\Frob_\ell)=\left(\begin{smallmatrix}
				 a_\ell/2 & 0 \\
				b_E & a_\ell/2\end{smallmatrix}\right),\qquad
                \rhobar_{E',p}(\Frob_\ell)=\left(\begin{smallmatrix}
				 a_\ell/2 & 0 \\
				0 & a_\ell/2\end{smallmatrix} \right) . 
\end{equation}

   Since $b_E \not\equiv 0 \pmod p$, these matrices are not  conjugated, so $E[p] \not\cong E'[p]$, a contradiction.
\end{proof}

\begin{lemma}\label{lem:square}
    Let $\ell$ and $p \geq 3$ be different primes.
    Let $E/\F_\ell$ be an elliptic curve having an isogeny of prime degree $q\neq 2,p, \ell$. Assume that $E$ satisfies the first four conditions in Theorem~\ref{thm:maingood}. Then $(q/p)=1$.
\end{lemma}
\begin{proof} 
    An easy adaptation of Lemma~\ref{good} over finite fields gives that the existence of a $q$-isogeny implies $(\Delta_\ell/q)\in \{0,1\}$.
    If $q\mid \Delta_\ell$, then \eqref{g4} implies that
    $(q/p)=1$. If $(\Delta_\ell/q)=1$, then \eqref{g2} implies $(-p/q)=1$. Now, using quadratic reciprocity and \eqref{g1}, one gets
    \[
    \left( \frac{q}{p}\right)=(-1)^{(q-1)/2}\left( \frac{p}{q}\right)=\left( \frac{-p}{q}\right)=1.
    \]\end{proof}

The following elementary lemmas will play an important r\^ole in the proofs that follow. 
\begin{lemma} \label{lem:ala}
Let $\ell$ and $p\geq 3$ be distinct primes, and let
$a_\ell\in\mathbb{Z}$ satisfy
$|a_\ell|\leq 2\sqrt{\ell}$.

Then there exists an $a\in \Z$ satisfying
\[
a\neq \pm a_\ell,\qquad
|a|\leq 2\sqrt{\ell},
\qquad\text{and}\qquad
a\equiv a_\ell\pmod p
\]
if and only if
\[
|a_\ell|\geq p-2\sqrt{\ell}.
\]
Moreover, if such $a$ exists, one can take it to be $a:=a_\ell+p$ or $a:=a_\ell-p$. 
\end{lemma}

\begin{proof}
Suppose first that there exists an integer $a\in\mathbb{Z}$
satisfying
\[
a\neq \pm a_\ell,\qquad
|a|\leq 2\sqrt{\ell},
\qquad\text{and}\qquad
a\equiv a_\ell\pmod p.
\]
Since $a\neq a_\ell$ and $a\equiv a_\ell\pmod p$, the nonzero integer
$a-a_\ell$ is divisible by $p$. In particular,
$
p\leq |a-a_\ell|.
$
Hence triangle inequality gives
\[
p \leq |a-a_\ell|
\leq |a|+|a_\ell|
\leq 2\sqrt{\ell}+|a_\ell|.
\]

Conversely, suppose that \begin{equation}\label{eq:eq}
    |a_\ell|\geq p-2\sqrt{\ell}.
\end{equation}
We assume first that $a_\ell\geq 0$ and define
$
a:=a_\ell-p \neq a_\ell.$ It is easy to see that $a \neq -a_\ell$ as otherwise $p=2a_\ell$, a contradiction. Trivially, $a \equiv a_\ell \bmod p$, and we note that
\[
-2\sqrt \ell\leq a_\ell-p< 2\sqrt{\ell},
\]
where the first inequality comes from \eqref{eq:eq} and the second inequality from the fact that $|a_\ell|\leq 2\sqrt{\ell}$, giving $|a|\leq 2\sqrt{\ell}$. If $a_\ell<0$, an analogous argument with $a:=a_\ell+p$. 
\end{proof}

\begin{lemma}\label{lem:kraus}
    Let $\ell > 2$ and $p>3$ be different primes such that $p \equiv 3 \pmod 4$. 
    Suppose that 
    \begin{equation}\label{eq:sq}
    a^2+pt^2=b^2+pu^2 = 4 \ell
    \end{equation}
    for $(a,b,t,u)\in \Z^4$. Then $a=\pm b$. In particular, $b \neq a\pm p$.
\end{lemma}
\begin{proof}
Firstly, we show that $\ell \nmid a$. Indeed, as $\ell > 2$, if $\ell \mid a$, then 
$\ell \mid t$ and so $\ell^2 \mid a^2 + pt^2 = 4\ell$, a contradiction.  

Consider the quadratic field $K:=\Q(\sqrt{-p})$. Since $-p\equiv 1 \pmod 4$, we have $\cO_K =\Z[\frac{1+\sqrt{-p}}{2}]$.
Reducing~\eqref{eq:sq} modulo $\ell$, and taking Legendre symbols, gives $\left(\frac{-p}{\ell}\right)=\left( \frac{(a/t)^2}{\ell} \right)=1$. Thus $\ell$ splits in $\cO_K$, and the two primes above it are swapped by $\Gal(K/\Q)=\langle \sigma \rangle$.

Suppose $(a,b,t,u)\in \Z^4$ satisfies~\eqref{eq:sq} and view it in $\cO_K$. The elements
    \[
    \alpha:=\frac{a+t\sqrt{-p}}{2}, \qquad \beta:=\frac{ b+u\sqrt{-p}}{2}
    \]
belong to $\cO_K$, because \eqref{eq:sq} assures that $a, t$ and $b, u$ have the same parity. 

In particular, \eqref{eq:sq} implies that $\Norm(\alpha)=\Norm(\beta)=\ell$. Therefore,  $$(\alpha)\cdot\cO_K =:\mathfrak{P} \quad \text{ and } \quad (\sigma(\alpha))\cdot\cO_K=\sigma(\mathfrak{P})=:\mathfrak{Q}$$
are the primes above $\ell$. The same holds for $\beta,\: \sigma (\beta)$ up to permuting them. Therefore, without loss of generality, $\alpha$ and $\beta$ are both uniformizers of the prime ideal $\mathfrak{P}$. Hence, $\alpha/\beta\in \cO_K^*=\{\pm 1\}$. This implies that $a=\pm b$.
\end{proof}
\begin{lemma}\label{lem:isgprime}
Let $E/\F_\ell$ and $E'/\F_\ell$ be two ordinary elliptic curves with the same endomorphism ring~$\cO$ and $\phi: E \to E'$ an isogeny of degree coprime to $\ell$. Let $M\in \Z_{> 0}$. Then there exists an isogeny $\psi: E \to E'$ of degree coprime to $M\ell$.
\end{lemma}

\begin{proof}
    By \cite[Section 2.8]{Drew},  every isogeny $\phi:E\to E'$ of degree coprime to $\ell$ corresponds to a fractional ideal $I_\phi$ in $\cO$ with $\deg(\phi)=\Norm(I_\phi)$. Any other isogeny between $E$ and $E'$ corresponds to another ideal in the class $[I_\phi]\in \Cl(\cO)$. The conclusion follows from~\cite[Corollary 7.17]{cox}.
\end{proof}

\begin{lemma}\label{lem:cyclic-isog}
   Let $E/\F_\ell$ and $E'/\F_\ell$ be elliptic curves, and $N\in\Z_{>0}$. Let $\varphi: E\to E'$ be an isogeny of minimal degree $n$ among all $\F_\ell$-isogenies from $E$ to $E'$ whose degree is coprime to~$N$. Then $\varphi$ is cyclic and $\ell \nmid n$.
\end{lemma}
\begin{proof}
Let $n=\deg \varphi$.
    Firstly, we note that minimality implies that $\varphi$ is separable, and therefore $\ell \nmid n$. 
Suppose that $\ker\varphi$ is not cyclic. Then there exists a prime $r\nmid\ell N$ with  $r\mid n$ such that \[ (\ker\varphi)[r]\simeq (\Z/r\Z)^2 \simeq E[r]. \]
so $\varphi$ factors as $ \varphi=\psi\circ[r] $ for an $\F_\ell$-rational isogeny $\psi: E\to E'$. Consequently, $\deg\psi=\frac{n}{r^2}<n$, contradicting minimality.
\end{proof}

\begin{theorem}\label{prop:good2}
    Let $\ell$ and $p \geq 7$ be two different primes and $E/\Q_\ell$ an elliptic curve with good reduction, with special fibre $\oE$ and $a_\ell:=a_\ell(\oE)$.
    Assume that hypotheses \eqref{g1}--\eqref{g4} of Theorem~\ref{thm:maingood} hold. Then
    \begin{enumerate}[(a)]
        \item if $|a_\ell |< p-2\sqrt{\ell}$, then $Y^-_{\oE}(p)(\F_\ell)=\emptyset$;
        \item otherwise, $Y^-_{\oE}(p)(\F_\ell)\neq\emptyset$.
    \end{enumerate}
\end{theorem}
\begin{proof}

We will first prove (a). Assume that $|a_\ell |< p-2\sqrt{\ell}$. 
Let $E' / \F_\ell$ be an elliptic curve and $\psi : \oE[p] \to E'[p]$ a $G_{\F_\ell}$-isomorphism. We will show that 
$\psi$ is symplectic. Indeed,
we claim there is an isogeny $\varphi :\oE \to E'$ of degree $n'$ coprime to $\ell$~and~$p$ such that $(n' /p)=1$.
The isogeny $\varphi$ induces a symplectic isomorphism
$\varphi|_{\oE[p]} : \oE[p] \to E'[p]$ by Lemma~\ref{lem:isogeny}.
Therefore, by Lemma~\ref{lem:samesym}, the isomorphism
$\psi$ is also symplectic.

We will now prove the claim.

The existence of~$\psi$ implies that $a_\ell\equiv a_\ell(E') \pmod p$.
From Lemma~\ref{lem:ala} it follows that $a_\ell = a_\ell(E')$ under our assumptions. Therefore, there exists an isogeny $\varphi:\oE \to E'$. We choose it to be of minimal degree $n$. By Lemma~\ref{lem:cyclic-isog} with $N=1$, we have $\varphi$ is cyclic and $\ell \nmid n$.
Write $n=p^k m$ with $ p\nmid m.$
By considering the prime-to-$p$ part of the kernel of $\varphi$, we may
factor $\varphi$ as
\[
    \overline E\xrightarrow{\eta}E_0\xrightarrow{\pi_0}E',
\]
where $\deg\eta=m$ and $\deg \pi_0 =p^k$.
Since $p\nmid m$, the restriction of $\eta$ to the $p$-torsion induces an
isomorphism
\[
    \overline E[p]\simeq E_0[p]
\]
of $G_{\F_\ell}$-modules. Together with the assumption
$\overline E[p]\simeq E'[p]$, this gives $E_0[p]\simeq E'[p].$
Moreover, condition~\eqref{g3}, which is determined by the isomorphism class of
the $p$-torsion Galois module, gives $p\nmid b_{E_0}$. 
We may therefore
apply Lemma~\ref{lem:eqEndnew} to the $p^k$-isogeny $\pi_0$ and conclude that
\[
    \cO_{E_0}=\cO_{E'}.
\]

Applying Lemma~\ref{lem:isgprime} with $M=p$ shows that there exists an isogeny $\theta: E_0 \to E'$ of degree coprime to $p \ell$. Therefore, the isogeny $\theta\circ\eta: \oE\to E'$ has degree coprime to $p\ell$. Among all such isogenies, we choose $\varphi'$ such that $n'=\deg \varphi'$ is minimal, and by Lemma~\ref{lem:cyclic-isog} with $N=p\ell$, we conclude that $\varphi'$ is cyclic and $\ell \nmid n'$.

We will now show that $(n'/p)=1$.
Let $q$ be a prime dividing $n'$, we will show that $(q/p)=1$. As $\varphi'$ is cyclic, it has a unique cyclic group of order $q$ which is invariant under $G_{\F_\ell}$ and hence gives rise to a $q$-isogeny. When $q$ is odd, Lemma~\ref{lem:square} gives $(q/p)=1$.

It remains to consider $q=2$. The existence of the $2$-isogeny implies that  
\[
    2\mid\#\overline E(\F_\ell)=\ell+1-a_\ell.
\]
Since $\ell\nmid n'$, the assumption $2\mid n'$ implies that $\ell$ is odd.
It follows that $a_\ell$ is even, and therefore
\[
    2\mid\Delta_\ell=a_\ell^2-4\ell.
\]
Condition~\eqref{g4}, applied with $q=2$, gives $\left(\frac{2}{p}\right)=1$, establishing the claim, and hence (a). 

We will now prove (b). 
Suppose that $|a_\ell| \geq p-2\sqrt{\ell}$. Note that $\ell>2$ by the Hasse-Weil bound and the fact that $p\geq 7$.
By Lemma~\ref{lem:ala} it follows that $a := a_\ell + p$ or $a:=a_\ell-p$ satisfies the Weil bound $|a|\leq 2\sqrt{\ell}$ and $a^2-4\ell \equiv \Delta_\ell \equiv 0 \pmod p$. 
Therefore, by Lemma~\ref{lem:diagfrob}, there exists an elliptic curve $E'/\F_\ell$ satisfying $a_\ell(E')=a$ and with $\rhobar_{E',p}(\Frob_\ell)$ given by~\eqref{eq:Frob}. Moreover, Lemma~\ref{lem:Centeleghe} implies that $\rhobar_{\oE,p}(\Frob_\ell)$ is given by~\eqref{eq:cent}.

As the matrices describing $\rhobar_{E',p}(\Frob_\ell)$ and $\rhobar_{\oE,p}(\Frob_\ell)$ are conjugated inside $\GL_2(\F_p)$, it follows
that there is a $G_{\F_\ell}$-isomorphism $\psi: \oE[p] \to E'[p]$. If $\psi$ is anti-symplectic, then $(E',\psi)$ gives a point on $Y_{\oE}^-(p)(\F_\ell)$. If $\psi$ is symplectic, then we claim that $E'$ admits an anti-symplectic isomorphism $\varphi :E'[p] \to E''[p]$, and therefore $(E'', \varphi \psi)$ gives a point on~$Y_{\oE }^-(p)(\F_\ell)$
and part (b) follows.

We will now prove the claim.

Note first that hypothesis~\eqref{g2} implies that $\Delta_\ell:=a_\ell^2-4\ell=-p t^2$, for some $t\in \Z_{>0}$.
Suppose that $\Delta_\ell(E'):=a^2-4\ell =-pu^2$, some $u\in \Z_{>0}$. By rearranging, one gets that 
\[
a_\ell^2+p t^2 = 4\ell = (a_\ell\pm p)^2+pu^2,
\]
contradicting Lemma~\ref{lem:kraus} (recall that $\ell > 2)$. Therefore, 
for any elliptic curve $\widetilde{E}'/\Q_\ell$ with good reduction and residual curve~$E'$, we know that it 
does not satisfy hypothesis~\eqref{g2}. 
Hence $Y_{\widetilde{E}'}^-(p)(\Q_\ell) \neq \emptyset$ by Proposition~\ref{prop:good1} and, in fact, at least one point in $Y_{\widetilde{E}'}^-(p)(\Q_\ell)$ arises from a $\Q_\ell$-isogeny $\widetilde{E}' \to \widetilde{E}''$ of degree $q$ such that $(q/p)=-1$ (see Remark~\ref{rem:isogenies}). Thus $\widetilde{E}''/\Q_\ell$ also has good reduction 
and there is an anti-symplectic isomorphism $\widetilde{E}'[p] \to \widetilde{E}''[p]$. Since the reduction morphism is Galois equivariant, taking mod~$\ell$ reductions yields 
an anti-symplectic $G_{\F_\ell}$-isomorphism $\varphi : E'[p] \to E''[p]$ as claimed (here $E''/\F_\ell$ denotes the reduction of $\widetilde{E}''$).    
\end{proof}

\begin{proof}[Proof of Theorem~\ref{thm:maingood}]
 
Suppose first that all conditions in the statement are satisfied. It follows by Theorem~\ref{prop:good2} and Corollary~\ref{cor:refereeRequest} that $X_{\oE}^-(p)(\F_\ell)=\emptyset$, and hence  $\Xp(\Q_\ell)=\emptyset$ by Corollary~\ref{rmk}.

Conversely, assume first that \eqref{g3} fails. The group $G$ is abelian (cyclic) and it satisfies $p\nmid \#G$ if and only if $p$ does not divide the order of $\rhobar_{E,p}(\Frob_\ell)$. Therefore, $\Xp(\Q_\ell)\neq\emptyset$ by Theorem~\ref{thm:abelian}. 
If one of the assumptions~\eqref{g1}, \eqref{g2} or \eqref{g4} fails, then the conclusion follows from Proposition~\ref{prop:good1}. Thus we can assume \eqref{g1}--\eqref{g4} hold.

If \eqref{g5} fails then Theorem~\ref{prop:good2} gives $Y_{\oE}^-(p)(\F_\ell)\neq \emptyset$ and if \eqref{g6} fails, then Corollary~\ref{cor:refereeRequest} implies there is an $\F_\ell$-rational cusp in $X_{\oE}(p)$. In both of these cases we can conclude $\Xp(\Q_\ell)\neq\emptyset$ by Corollary~\ref{rmk}.
\end{proof}

We finish this section with the case of twist of good reduction. Let $E/\Q_\ell$ an elliptic curve with $e=2$. From \cite[Lemmas 3 and 4]{symplectic} there is $u$ such that the quadratic twist $E^u/\Q_\ell$ of~$E/\Q_\ell$ has good reduction.
\begin{corollary}\label{cor:e2}
Let $\ell$ and $p \geq 7$ be two different
primes and $E/\Q_\ell$ an elliptic curve with $e=2$. Let $E^u/\Q_\ell$ be its quadratic twist with good reduction. 
Then $X_E^{-}(p)(\Q_\ell)= \emptyset$ if and only if all of the following hold
\begin{enumerate}
\item $p \equiv 3 \pmod{4}$; 
 \item $-p\Delta_\ell(E^u)=s^2$ for some $s \in \Z$;
\item $p\mid \#\rhobar_{E^u,p}(\Frob_\ell)$;
\item  for all primes $q\neq \ell, \:q \mid \Delta_\ell(E^u) \Rightarrow (q/p)\neq -1$; 
\item $|a_\ell(E^u) |< p-2\sqrt{\ell}$; 
\item $\ell \not\equiv 1 \pmod{p}$. 
\end{enumerate} 
\end{corollary}
\begin{proof} 
The conclusion follows from Theorem~\ref{thm:maingood} applied to $E^u/\Q_\ell$ and Lemma~\ref{lem:qtwists}.
\end{proof}
\vspace{-5mm}
\section{The case of semistability defect \texorpdfstring{$e=3$ or $4$}{}}
\label{sec:non-ab}

In view of Theorem~\ref{thm:p=3}, we may assume $p \geq 7$. Nonetheless, all proofs below are valid for $p \geq 5$, and many arguments also apply for $p=3$ and elliptic curves with $e=4$.

Let $\ell \neq p$ be a fixed prime and $E/\Q_\ell$ be an elliptic curve with potentially good reduction and semistability defect~$e \in \{3,4\} $. Let $F \subset \Qbar_\ell$ be a field such that $F/\Q_\ell$ is totally ramified of degree~$e$ and $E/F$ has good reduction.
Let $\Ebar / \F_\ell$ be the elliptic curve
obtained by reduction of a model of $E / F$ with good reduction.
Let $E_F[p]$ denote $E[p]$ as a $G_F$-module.
We write $\varphi : E_F[p] \rightarrow \overline{E}[p]$ for
the reduction morphism, which is a  $G_F$-isomorphism.
Let $\Aut(\Ebar)$ denote the group of $\overline{\F_\ell}$-automorphisms of $\Ebar$.

For primes $p$ and $\ell$ as above, the field $L:=\Qun(E[p])$
is independent of $p$ and is the minimal finite extension of
$\Qun$ over which $E$ acquires good
reduction; see
\cite[Theorem~2(ii), and Corollary~2]{SerreTate}. We denote its Galois group by 
$\Phi := \Gal(L/\Qun)$, and refer to it as the inertial group.

We will write $K:=\Q_\ell(E[p])$ for the $p$-torsion field of $E$ and $G:=\Gal(K/\Q_\ell)$ for its Galois group, and we note that $G \simeq \rhobar_{E,p}(G_{\Q_\ell})$. We say that $K/\Q_\ell$ is abelian/cyclic if $G$ is.
In the presence of a second curve $E'/\Q_\ell$ we will use the notations $e'$, $K'$ and $\varphi'$.

\subsection{The field of good reduction when \texorpdfstring{$e=3$}{} or \texorpdfstring{$4$}{}}

Recall that Theorem~\ref{thm:abelian} provides a point in~$\Xp(\Q_\ell)$ when $G $ is abelian.
Therefore, we will focus on the case of non-abelian~$G$ or equivalently non-abelian $K/\Q_\ell$.
We recall the following result from~\cite{symplectic}.

\begin{theorem}[Freitas--Kraus]\label{thm:F}
Let $E/\Q_\ell$ be an elliptic curve with $e \in \{3,4\} $.
Then,
\begin{enumerate}
    \item If $\gcd(\ell,e)=1$, then $K/\Q_\ell$ is non-abelian if and only if $\ell \equiv -1 \pmod{e}$.
    \item If $(\ell,e)=(3,3)$, then $K/\Q_\ell$ is non-abelian if and only if $\tilde{\Delta} \equiv 2 \pmod 3$.
    \item If $(\ell,e)=(2,4)$, then $K/\Q_\ell$ is non-abelian if and only if $\tilde{c}_4\equiv 5\tilde{\Delta} \pmod{8}$.
\end{enumerate}

Moreover, in all cases,
there is a degree~$e$, non-Galois, totally ramified extension $F/\Q_\ell$ such that $E$ has good reduction over~$F$. In correspondence with each of the above cases, the  field~$F$ is given by
\begin{enumerate}
    \item $F=\Q_\ell(\ell^{1/e})$;
    \item $F$ is defined by the polynomial $x^3+3x^2+3$;
    \item $F = F_i$ is defined by $f_1:=x^4 + 12x^2 + 6$ or $f_2:= x^4 + 4x^2 + 6$;
    furthermore, $E$ has good reduction over exactly one of the $F_i$ and if $E/F_2$ has good reduction then the quadratic twist of~$E$ by~$-1$ has good reduction over~$F_1$.
\end{enumerate}
\end{theorem}
\begin{proof}
This is a summary of the relevant entries in Theorem~17 together with Corollary~5, Proposition~7, and Proposition~8 in \cite{symplectic}.
\end{proof}
We observe that the property of $K/\Q_\ell$ being abelian is independent of $p\geq 3$, even though $K$ itself depends on $p$. This is to be expected, since $K\cdot \Qun = L$.

\subsection{The \texorpdfstring{$p$}{}-torsion module}
The goal of this section is to study the possibilities for the
$G_{\Q_\ell}$-modules $E[p]$ arising from elliptic curves $E/\Q_\ell$ with non-abelian $p$-torsion field~$K/\Q_\ell$ and semistability defect $e \in \{3,4\}$. It turns out that they are uniquely determined by~$(\ell,e)$
except when $(\ell,e) = (2,4)$ in which case there are two possibilities (see Theorem~\ref{thm:uniqueTorsion}). 

The following is a refinement of~\cite[Proposition~9]{symplectic} and it gives the presentation of $G$ in terms of $\ell$ and $p$ only, whenever $G=\Gal(K/\Q_\ell)$ is non-abelian.

\begin{proposition}\label{prop:G} 
    Let $\ell$ and $p \geq 5$ be two different primes and $E/\Q_\ell$ an elliptic curve with $e \in \{3,4\} $.
    Let $r$ be the order of $\ell \pmod p$.
     Suppose that $K/\Q_\ell$ is non-abelian. Let $F$ be the field of good reduction of~$E$ given by Theorem~\ref{thm:F}. 
     Let $f=[\Knr:\Q_\ell]$. We have two cases:

(1) Suppose $e=3$. Then  $G = \langle \tau, \sigma  : \tau ^f=1, \sigma ^3=1, \tau \sigma \tau ^{-1} = \sigma ^{-1} \rangle$ and $K\cap F= F$.

(2) Suppose $e=4$. Set $K_2 := \Q_\ell(\sqrt{\ell})$ if $\ell \geq 3$ and $K_2 := \Q_2(\sqrt{-2})$ if $\ell=2$.
	\begin{itemize}
        \item[(a)] If $r\equiv 2 \pmod 4$, then
        $G = \langle \tau, \sigma  : \tau ^f=1, \sigma ^4=1, \tau \sigma \tau ^{-1} = \sigma ^{-1} \rangle$ and $K\cap F= F$;
        \item[(b)] If $r\not\equiv 2 \pmod 4$, then $G  =\langle \tau , \sigma  : \sigma ^4=1, \tau \sigma \tau ^{-1} = \sigma ^{-1}, \tau^f=\sigma^2 \rangle$ and $K\cap F= K_2$,
    \end{itemize}
We also have $p\nmid \# G$.

Furthermore, in all cases, $\tau$ can be any generator of the cyclic group $\Gal(K/K\cap F)$ and we can identify $\sigma$ with any generator of the inertial group $\Phi$. In particular, we can assume that $\tau$ acts on~$K$ as~$\Frob_F$.
\end{proposition}
\begin{proof} 
From $\tau \sigma \tau^{-1} = \sigma^{-1}$ it follows by a simple induction that $\tau\sigma^k = \sigma^{-k}\tau$. This relation
 shows that the elements of the presented group can be written in the
form
\[
\sigma^i\tau^j,
\qquad
0\leq i<e,\quad 0\leq j<f
\]
where in case~\textup{(2)(b)} we used
\(\tau^f=\sigma^2\) to reduce the possible exponent of \(\tau\).
Thus the order of the presented groups is at most~\(ef\); on the other hand,
we have $\#G=[K:\Q_\ell]=ef$, so we are left to show that 
the corresponding elements $\tau, \sigma \in G$ satisfy the same presentations.

From \cite[Lemma 9 and Proposition 9]{symplectic} we have one of the following cases:

(i) $K\cap F= F$ and $\tau$ and $\sigma$ satisfy the following relations
\[
\tau ^f=1, \qquad \sigma ^e=1,\qquad \tau \sigma \tau ^{-1} = \sigma ^{-1}. 
\]
(ii) $e=4$, $K\cap F= K_2$ and 
$\tau$ and $\sigma$ satisfy the following relations
\[
\tau ^{2f}=1, \qquad \sigma ^e=1,\qquad \tau \sigma \tau ^{-1} = \sigma ^{-1}. 
\]
Moreover, \textit{loc. cit.} also gives $p\nmid \# G$ and the claims regarding $\sigma$ and $\tau$ in the last paragraph of the statement.

Suppose first $e=3$. Then case (i) is satisfied and part (1) follows directly. 

Suppose now $e=4$. If we are in case (i), then the presentation in part (2)(a) holds. Suppose we are in case (ii).
In this case,
$\langle\tau\rangle=\Gal(K/K_2)$ has order $2f$, while
$\langle\sigma\rangle$ is the inertia subgroup of order $4$.
Moreover, the field $M:=K_{\mathrm{nr}}\cdot K_2$ satisfies
\[
 \langle\tau\rangle\cap\langle\sigma\rangle
 =
 \Gal(K/M).
\]
Since $K_2/\Q_\ell$ is ramified quadratic and
$K_{\mathrm{nr}}/\Q_\ell$ is unramified, we have
$[M:\Q_\ell]=2f.$
As $[K:\Q_\ell]=4f$, it follows that
$\#\bigl(\langle\tau\rangle\cap\langle\sigma\rangle\bigr)=2.$
The unique element of order $2$ in $\langle\tau\rangle$ is $\tau^f$,
whereas the unique element of order $2$ in
$\langle\sigma\rangle$ is $\sigma^2$. Therefore, $$ \tau^f=\sigma^2,$$ 
giving the presentation in (2)(b). In particular, the order of $\tau$ is $2f$.

To finish the proof, we will establish the relation between $r$ and the two cases in part (2).

After a choice of basis, we denote by $N\in \GL_2(\F_p)$ the matrix corresponding to  $\rhobar_{E,p}(\Frob_F)$ (or equivalently, to $\tau$).
 By \cite[Theorem 4.11]{torsf}, we have
    \begin{equation}\label{eq:d}
d:=[K:\Q_\ell]=\begin{cases}
    4r, & \text{ if }r \text{ is even},\\
    8r, & \text{ if }r \text{ is odd},
\end{cases}
    \end{equation}
    where $r$ is the order of $\ell \pmod p$. As
    $[K:\Q_\ell]=4f$, we get that 
     \[
f=\begin{cases}
    r, & \text{ if }r \text{ is even},\\
    2r, & \text{ if }r \text{ is odd}.
\end{cases}
    \]
    Recall that $E/F$ has good reduction. We claim that $a_F(E)=0$. Therefore, the characteristic
    polynomial of $N$ is $t^2+\ell$.
By the Cayley--Hamilton theorem, we have $N^2=-\ell \cdot\Id$.
Note that we are in case (i) if and only if the order of $N$ is $f$. We will now investigate the order of $N$.

Suppose first that $r$ is odd, then $f=2r$. One gets that 
\[
N^f=(N^2)^r=(-\ell)^r \cdot\Id=-\ell^r\cdot\Id= -\Id,
\]
showing that $f$ is not the order of $N$, and hence we must be
in case~(ii).

Suppose now that $r$ is even. We have $\ell^{r/2}=-1$.
If $r=4k$, then $f=r=4k$, and hence
\[
N^f= (N^2)^{2k}=(-\ell)^{2k}\:\Id= \ell^{r/2} \:\Id =- \:\Id,
\]
and again, we are in case (ii).
Finally, if $r=4k+2$, then $f=r=4k+2$, and hence 
\[
N^f= (N^2)^{2k+1}=(-\ell)^{2k+1}\:\Id= -\ell^{r/2} \:\Id = \:\Id,
\]
showing that $N$ is of order~$f$ and hence we are in case (i).

Finally, we prove the claim.
Let $S:=\rhobar_{E,p}(\tau)$ and $A:=\rhobar_{E,p}(\sigma)$. 
Over $\overline{\F}_p$, the matrix $A$ is
diagonalizable with distinct eigenvalues~$\{\zeta_e, \zeta_e^{-1}\}$. The relation
$SAS^{-1}=A^{-1}$ shows that $S$ interchanges the two eigenspaces of $A$, and therefore it must have trace $0$. 
This shows that $a_F(E) \equiv 0 \pmod{p}$. The property of $K/\Q_\ell$ being non-abelian and the definition of $F$ are independent of $p$ by Theorem~\ref{thm:F}. Therefore, we can repeat the argument using the $q$-torsion for an auxiliary prime~$q$ large enough compared to $\ell$ which together with the Hasse--Weil bound forces $a_F(E) = 0 \in \Z$, as claimed. 
\end{proof}

The claim established in the previous proof refines~\cite[Theorem~3.2]{doks} under additional hypotheses, so we highlight it as a corollary.

\begin{corollary} \label{cor:al0}
    Let $\ell$ be a prime and $E/\Q_\ell$ an elliptic curve with $e \in \{3,4\} $. Suppose that $\ell = -1 \pmod e$, so that $E$ has good reduction over~$F:=\Q_\ell(\ell^{1/e})$.
    Then  $a_F(E)=0$.
\end{corollary}
 
The following shows that the $p$-torsion field extension of $E/\Q_\ell$ is essentially uniquely determined by the semistability defect if it is non-abelian.

\begin{proposition}\label{prop:samefields}
Let $\ell$ and~$p \geq 5$ be different primes.
    Let $E/\Q_\ell$ and $E'/\Q_\ell $ be elliptic curves with $e=e'\in \{3,4\}$. 
    Suppose that $K/\Q_\ell$ and $K'/\Q_\ell$ are non-abelian, and that $E$ and~$E'$ obtain good reduction over the same~$F$ given by Theorem~\ref{thm:F}.  Then $K=K'$.
\end{proposition}
\begin{proof}
	Let $E/\Q_\ell$ be as in the statement, that is,
    $E$ has semistability defect $e \in \{3,4\} $, non-abelian
    $p$-torsion field $K=\Q_\ell(E[p])/\Q_\ell$ and obtains good reduction over a field~$F$ prescribed by Theorem~\ref{thm:F}.
We will show there is a unique possibility for~$K$, determined by~$F$, $\ell$, and~$p$.
The result then follows because of the assumption that both curves have good reduction over the same~$F$.
    Observe that, from Theorem~\ref{thm:F}, this assumption is automatically satisfied  except when $(\ell,e) = (2,4)$.

   Recall that $e=[K : \Knr]$ where $\Knr/\Q_\ell$ is the maximal unramified subextension of~$K/\Q_\ell$.

    Suppose that $e=3$. By \cite[Theorem 4.11]{torsf}, we have $[K:\Q_\ell]=6\delta,$ where $\delta$ is the order of
    $-\ell$ in $\F_p^*$. By Proposition~\ref{prop:G}, we have $F \subseteq K$ and recall that $F$ is totally ramified of degree~$3$. As $e=3$ is the ramification degree of $K$, it follows that
    $K$ is the unique unramified extension of $F$ of degree~$2\delta$.

   Suppose that $e=4$. Let $r$ be the order of~$\ell$ mod~$p$.

   If $r \equiv 2 \pmod 4$, then $K \cap F = F$ by Proposition~\ref{prop:G} and the conclusion follows exactly as for $e=3$. So assume $r \not\equiv 2 \pmod 4$. Then $K \cap F = K_2 = \Q_\ell(\sqrt{s})$ is ramified and given by $s=\ell$ for $\ell > 2$ and $s=-2$ for $\ell=2$.
Let $K_3:=\Knr \cdot K_2= \Knr(\sqrt{s})$, which is ramified. 

Recall from~\eqref{eq:d} that, when $e=4$, the degree
$d=[K:\Q_\ell]$
depends only on $r$. Therefore, $K_{\mathrm{nr}}/\Q_\ell$ is the unique
unramified extension of degree $d/4$, and hence $K_3$ is uniquely
determined.
Let
$ \alpha,\beta:G_{K_3}\to\{\pm1\}
$ 
be the quadratic characters corresponding, respectively, to the
quadratic extensions $K/K_3$ and $F \cdot K_3/K_3$. Both characters are ramified. Moreover, they are distinct: otherwise
$K=FK_3$, which would imply $F\subseteq K$, contradicting $K\cap F=K_2.$

Let $I_{K_3}$ denote the inertia subgroup of $G_{K_3}$. By the
minimality of the inertial field of $E$, we have
$  L=K_3^{un} K
     =K_3^{un} F.$
Consequently,
\[
    \ker \alpha|_{I_{K_3}}
    =
    G_L
    =
    \ker\beta|_{I_{K_3}}.
\]
Since $\alpha|_{I_{K_3}}$ and $\beta|_{I_{K_3}}$ are nontrivial
quadratic characters with the same kernel, they are equal. It follows
that the quadratic character $\alpha/\beta$
is unramified. Since $\alpha\neq\beta$, it is the unique nontrivial
unramified quadratic character of $G_{K_3}$, which we denote by $\chi_{\mathrm{un}}$.
Thus,
\[
    \alpha=\beta\chi_{\mathrm{un}}.
\]
 Since $F$ is fixed and $K_3$ uniquely determined, the character 
 $\beta$ is determined and hence $\alpha$ is uniquely determined. Thus its fixed field $K/K_3$ also is, therefore, $K$ is uniquely determined.
\end{proof}

\begin{lemma}
    \label{lem:faithful} 
    Let $\ell$, $p$, $\sigma$, $\tau$ and $G$ be as in Proposition~\ref{prop:G}. Up to conjugation, there is only one faithful representation $\rho : G \to \GL_2(\F_p)$ satisfying
    \[
    \det \rho(\sigma) = 1 \quad \text{ and } \quad \det \rho(\tau) = \ell.
    \]
\end{lemma}
\begin{proof}
    From Proposition~\ref{prop:G} we know that $p$ does not divide the order of $G$, hence all representations of $G$ into $\GL_2(\F_p)$ are semisimple by Maschke’s theorem. 
Since any $\rho$ satisfying the hypothesis have the same determinant, they have the same
characteristic polynomials if and only if they have the same traces. In that case, it follows from Brauer--Nesbitt theorem~\cite[Theorem~30.16]{CurtisReiner} that all such $\rho$ are isomorphic.
Thus, to prove the result, it suffices to show that the trace of $\rho$ is fully determined.  

    Recall that $\sigma$, $\tau$ generate $G$ and $\sigma$ has order $e \in \{3,4\}$. Moreover, from the first paragraph of the proof of Proposition~\ref{prop:G}
    we know that $\tau\sigma^k = \sigma^{-k}\tau$ and all elements of $G$ can be written as 
    $\sigma^i \tau^j$ with $1 \leq i \leq e$ and $1 \leq j \leq f$, where $f=[K_{nr} : \Q_\ell]$.  
    
    Since \(\rho\) is faithful, \(\rho(\sigma)\) has two distinct
eigenvalues.     Arguing as in the last paragraph of the proof of Proposition~\ref{prop:G}, we conclude that $\rho(\tau)$ exchanges the two eigenspaces of $\rho(\sigma)$. Therefore,  $\tr \rho(\tau) = 0$, hence $\rho(\tau)^2 + \ell \cdot\Id = 0$ (by the Cayley--Hamilton theorem) and thus $\rho(\tau^2) = -\ell \cdot \Id$. We conclude that 
$$\tr \rho(\sigma^i \tau^j) =      
\begin{cases}
 (-\ell)^\frac{j}{2} \cdot \tr \rho(\sigma^i) \quad \text{if }j\text{ is even,} \\
(-\ell)^\frac{j-1}{2} \cdot \tr \rho(\sigma^i \tau) \quad \text{if } j\text{ is odd.}
\end{cases} $$

Suppose $e=3$. Since $\rho$ is faithful, the matrices 
$\rho(\sigma)$ and $\rho(\sigma^2)$ have order $3$, therefore their characteristic polynomial is $x^2 + x + 1$, so $\tr \rho(\sigma) = \tr \rho(\sigma^2) = -1$. Moreover, it is easy to check that $\sigma^i \tau \sigma^{-i} = \sigma^{-i} \tau$, thus $\tr \rho(\sigma^{-i} \tau) = \tr \rho(\sigma^i \tau \sigma^{-i}) = \tr \rho(\tau) = 0$, showing that the trace of $\rho$ is completely determined in this case.

Suppose $e=4$.  Since $\rho$ is faithful,  
$\rho(\sigma)$ and $\rho(\sigma^3)$ have order $4$ and determinant~1, hence their characteristic polynomial is $x^2 + 1$, so $\tr \rho(\sigma) = \tr \rho(\sigma^3) = 0$. Moreover, by the Cayley--Hamilton theorem $\rho(\sigma^2)=-\Id$ thus $\tr \rho(\sigma^2) = -2$ and we have
$$\tr \rho(\sigma^i \tau) =      
\begin{cases}
 (-1)^\frac{i}{2} \cdot \tr \rho(\tau) \quad \text{if }i\text{ is even,} \\
(-1)^\frac{i-1}{2} \cdot \tr \rho(\sigma \tau) \quad \text{if }i\text{ is odd.}
\end{cases} $$
Finally, from $\sigma (\sigma \tau) \sigma^{-1} = \sigma^2 (\sigma \tau)$ we have $\tr \rho(\sigma \tau) = - \tr \rho(\sigma \tau) = 0$ and thus $\tr \rho(\sigma^i \tau) = 0$ in all cases.
\end{proof}
We can finally prove the main result of this section.

\begin{theorem} \label{thm:uniqueTorsion}
    Let $\ell$ and~$p \geq 5$ be different primes. Let $E/\Q_\ell$, $E'/\Q_\ell $ be elliptic curves with potentially good reduction 
    and $e = e' \in \{3,4\}$. 
    Suppose one of the following holds:
\begin{enumerate}[(1)]
    \item $\ell \equiv -1 \pmod{e}$; 
    \item $(\ell,e) = (3,3)$ and $\tilde{\Delta} \equiv \tilde{\Delta}' \equiv 2 \pmod{3}$; 
    \item $(\ell,e) = (2,4)$, $\tilde{c}_4 \equiv 5 \tilde{\Delta} \pmod{8}$, $\tilde{c}_4' \equiv 5 \tilde{\Delta}' \pmod{8}$ and both curves obtain good reduction over the same field $F_i$ given in Theorem~\ref{thm:F} (3).
\end{enumerate}
Then $E[p]$ and $E'[p]$ are isomorphic as $G_{\Q_\ell}$-modules.
\end{theorem}
\begin{proof}
By Theorem~\ref{thm:F}, the fields $K=\Q_\ell(E[p])$ and $K'=\Q_\ell(E'[p])$ are non-abelian extensions of $\Q_\ell$ and both curves $E$ and $E'$ obtain good reduction over the same field $F$; note the latter is automatic in cases (1) and (2) and by assumption in case (3). 

It follows from Proposition~\ref{prop:samefields} that $K = K'$ is the field fixed by $\ker \rhobar_{E,p} = \ker \rhobar_{E',p} \subset G_{\Q_\ell}$. We conclude that  $\rhobar_{E,p}$ and 
$\rhobar_{E',p}$ factor through faithful representations of  $G = \Gal(K/\Q_\ell)$ which satisfies one of the cases of Proposition~\ref{prop:G}. Since $\rhobar_{E,p}$ and 
$\rhobar_{E',p}$ send inertia into $\SL_2(\F_p)$ and $\Frob_F$ to a matrix of determinant~$\ell$, we conclude $\rhobar_{E,p} \simeq \rhobar_{E',p}$
by Lemma~\ref{lem:faithful}.
\end{proof}

\subsection{Revisiting symplectic criteria}
\label{sec:revisiting}
We show that for some local symplectic criteria in~\cite{symplectic}, the hypothesis that $E[p]$ and $E'[p]$ are isomorphic as $G_{\Q_\ell}$-modules is unnecessary. Specifically, 
for Theorems 1, 4, 5, and~6 in {\it loc. cit.}, we can include this isomorphism in the conclusion, as a consequence of our Theorem~\ref{thm:uniqueTorsion}. 

All the elliptic curves considered in this section have potentially good reduction with semistability 
defects $e \in \{3,4\}$. Moreover, from \cite[Table 5]{symplectic} we see that all elliptic curves covered by the next statement 
satisfy $v_\ell(\Delta_m) \in \{4,8\}$.   

\begin{theorem} Let $\ell$ and~$p \geq 5$ be different primes such that $\ell \equiv 2 \pmod{3}$.
Let $E/\Q_\ell$ and $E'/\Q_\ell$ be elliptic curves with $e=e'=3$.

Set $t=1$ if exactly one of $E$, $E'$ has a $3$-torsion point defined over $\Q_\ell$ and $t=0$ otherwise.

Set $r=0$ if $\upsilon_\ell(\Delta_m) = \upsilon_\ell(\Delta_m')$ and $r=1$ otherwise.

Then $E[p]$ and $E'[p]$ are isomorphic $G_{\Q_{\ell}}$-modules.
Moreover,
\[
 E[p] \text{ and } E'[p] \quad \text{are symplectically isomorphic} \quad \Leftrightarrow \quad \left(\frac{\ell}{p}\right)^r \left (\frac{3}{p} \right)^t = 1.
\]
\label{T:mainTame3}
Moreover, $E[p]$ and $E'
[p]$ are not both symplectically and anti-symplectically isomorphic.
\end{theorem}
\begin{proof} This follows from Theorem~\ref{thm:uniqueTorsion} (1) 
and \cite[Theorem 1]{symplectic}. 
\end{proof}

\begin{theorem} Let $p \geq 5$ be a prime.
Let $E/\Q_3$ and $E'/\Q_3$ be elliptic curves
with $e=e'=3$.
Suppose that $\tilde{\Delta} \equiv 2 \pmod{3}$ and $\tilde{\Delta}' \equiv 2 \pmod{3}$.

Let $r=0$ if $\tilde{c}_6 \equiv \tilde{c}_6' \pmod{3}$ and $r=1$ otherwise.

Then $E[p]$ and $E'[p]$ are isomorphic $G_{\Q_3}$-modules.
Moreover,
\[
 E[p] \text{ and } E'[p] \quad \text{are symplectically isomorphic} \quad \Leftrightarrow \quad \left (\frac{3}{p} \right)^r = 1.
\]
\label{T:mainWild3}
Moreover, $E[p]$ and $E'
[p]$ are not both symplectically and anti-symplectically isomorphic.
\end{theorem}
\begin{proof} This follows from Theorem~\ref{thm:uniqueTorsion} (2) 
and \cite[Theorem 4]{symplectic}. 
\end{proof}

We observe again from \cite[Table 5]{symplectic} that all elliptic curves covered by the next statement satisfy $v_\ell(\Delta_m) \in \{3,9\}$.

\begin{theorem} Let $\ell$ and~$p \geq 5$ be different primes such that $\ell \equiv 3 \pmod{4}$.
Let $E/\Q_\ell$ and $E'/\Q_\ell$ be elliptic curves with  $e=e'=4$.

Set $r=0$ if $\upsilon_\ell(\Delta_m) = \upsilon_\ell(\Delta_m')$  and $r=1$ otherwise.

Set $t=1$ if exactly one of $\tilde{\Delta}$, $\tilde{\Delta}'$ is a square mod~$\ell$ and $t=0$ otherwise.

Then $E[p]$ and $E'[p]$ are isomorphic $G_{\Q_{\ell}}$-modules. Moreover,
\[
 E[p] \text{ and } E'[p] \quad \text{are symplectically isomorphic} \quad \Leftrightarrow \quad \left(\frac{\ell}{p}\right)^r \left (\frac{2}{p} \right)^t = 1.
\]
\label{T:mainTame4}
Moreover, $E[p]$ and $E'
[p]$ are not both symplectically and anti-symplectically isomorphic.
\end{theorem}
\begin{proof} This follows from Theorem~\ref{thm:uniqueTorsion} (1) 
and \cite[Theorem 5]{symplectic}. 
\end{proof}

\begin{theorem} Let $p \geq 5$ be a prime.
Let $E/\Q_2$ and $E'/\Q_2$ be elliptic curves with $e=e'=4$.
Assume $\tilde{c}_4 \equiv 5\tilde{\Delta} \pmod{8}$ and $\tilde{c}_4' \equiv 5\tilde{\Delta}' \pmod{8}$. Assume also that $E$ and~$E'$ have good reduction over the same field $F$ given in Theorem~\ref{thm:F}.

Let $r=0$ if $\tilde{c}_6 \equiv \tilde{c}_6' \pmod{4}$ and $r=1$ otherwise.

Then $E[p]$ and $E'[p]$ are isomorphic $G_{\Q_2}$-modules. Moreover,
\[
 E[p] \text{ and } E'[p] \quad \text{are symplectically isomorphic} \quad \Leftrightarrow \quad \left (\frac{2}{p} \right)^r = 1.
\]
\label{T:mainWild4}
Moreover, $E[p]$ and $E'
[p]$ are not both symplectically and anti-symplectically isomorphic.
\end{theorem}
\begin{proof} This follows from Theorem~\ref{thm:uniqueTorsion} (3) 
and \cite[Theorem 6]{symplectic}. 
\end{proof}

\subsection{Points on \texorpdfstring{$\Xp$}{}}

In this section, we will generate points on $\Xp(\Q_\ell)$, by giving explicit examples of elliptic curves $E'/\Q_\ell$ with the same semistability defect $e=e' \in \{3,4\}$ and non-abelian $p$-torsion field, such that they satisfy the anti-symplectic criteria given in Section~\ref{sec:revisiting}. Moreover, the same criteria also identify when it is not possible to find such $E'$, hence offering converse statements. The case $e=6$ reduces to $e=3$ by passing to a quadratic twist of $E$, and it is included in this section for completeness. 

We will need the following auxiliary lemma.

\begin{lemma}\label{lem:same-defect}
Let $\ell\neq p$ be primes with $p\geq 5$, and let
$E/\Q_\ell$ be an elliptic curve with potentially good reduction and
semistability defect $e\geq 3$. Suppose there is an isomorphism
$E[p]\simeq E'[p]$ of $G_{\Q_\ell}$-modules  for some elliptic curve $E'/\Q_\ell$. Then $E'$
has potentially good reduction, its semistability defect satisfies
$e'=e$, and $E$ and $E'$ have the same inertial field.
\end{lemma}
\begin{proof}
We have $\rhobar_{E,p}\simeq \rhobar_{E',p}$ by assumption. 
In particular, the field fixed by the restrictions to inertia is
\[
    L:=\Qun (E[p])
      =\Qun (E'[p]).
\]
Since $E$ has potentially good reduction, $L$ is its inertial field and $[L:\Qun ]=e.$

Suppose that $E'$ has potentially multiplicative reduction. 
Then, by the theory of the Tate curve, the image
of inertia of $\rhobar_{E',p}(I_\ell)$ has order dividing $2p$.
This is impossible because
\[
    [\Qun (E'[p]):\Qun ]
    =[L:\Qun ]=e,
\]
where $e \geq 3$ and $p\nmid e$. 
Thus $E'$ has potentially good reduction, and
$e'=e=[L:\Qun]$. Moreover, the proof shows that $E$ and $E'$ have the same inertial field.
\end{proof}

We will now start the classification of $\Q_\ell$-points on $\Xp$, starting with the case $e=3$.

\begin{theorem}\label{thm:e3l}
	Let $\ell$ and~$p \geq 5$ be different primes such that $\ell \equiv 2 \pmod{3}$.
	Let $E/\Q_\ell$ be an elliptic curve with $e=3$.
	Then $\Xp(\Q_\ell)\neq \emptyset$
	if and only if $(3/p)= -1$ or $(\ell / p) = -1$.
\end{theorem}

\begin{proof}
Let $E/\Q_\ell$ be as in the statement.
If $X_E^-(p)(\Q_\ell)\neq\emptyset$, then there is an elliptic curve
$E'/\Q_\ell$ such that $E[p]$ and $E'[p]$ are anti-symplectically
isomorphic as $G_{\Q_\ell}$-modules. By
Lemma~\ref{lem:same-defect}, the curve $E'$ has potentially good
reduction and $e'=e=3.$
Therefore, it follows from Theorem~\ref{T:mainTame3} that $(3/p)= -1$ or $(\ell/p)= -1$.

To prove the converse, suppose first $(3/p)=-1$.
It follows from \cite[Proposition 10]{symplectic} that $E/\Q_\ell$ admits a $\Q_\ell$-isogeny of degree~3, therefore $\Xp(\Q_\ell) \neq \emptyset$ by Corollary~\ref{cor:isogeny}.

Assume now $(3/p) = 1$ and $(\ell / p) = -1$. Let $E'/\Q_\ell$ be an elliptic curve with $e'=e=3$.
From Theorem~\ref{T:mainTame3} we know that
$E[p]$ and $E'[p]$ are isomorphic $G_{\Q_\ell}$-modules
and that such an isomorphism is anti-symplectic if and only if
$v_{\ell}(\Delta_m)\neq v_{\ell}(\Delta'_m) \pmod{3}$.

To complete the proof, we show that such a curve $E'$ exists, and hence it will give rise to a point on $\Xp(\Q_\ell)$.

It is well-known (e.g.~\cite[Table 5]{symplectic}) that if an elliptic curve $E/\Q_\ell$ has semistability defect $e=3$ then $v_\ell(\Delta_m)\in \{4,8\}$, so we have two cases.

If $E$ satisfies $v_\ell(\Delta_m)=4$, we take
\[
E'/\Q_{\ell} \; : \; y^2+\ell xy+ \ell^2 y = x^3
\]
satisfying
\[
(v_{\ell}(c_4'), v_{\ell}(c_6'), v_{\ell}(\Delta_m'))=
\begin{cases}
    (3,4,8), & \text{ if } \ell \geq 5,\\

(4,6,8),\: \tilde{c}_6' \equiv 3 \pmod 4, \: \tilde{\Delta}'\equiv 3 \pmod 4 & \text{ if } \ell=2;
\end{cases}
\]
In these cases, we get $e'=3$ by \cite[Proposition 1]{Kraus} for $\ell \geq 5$ and \cite[p. 358]{Kraus} for $\ell=2$.

If $E$ satisfies $v_\ell(\Delta_m)=8$, we take
\[
E'/\Q_{\ell} \; : \; y^2+\ell xy+ \ell y = x^3
\]
satisfying
\[
(v_{\ell}(c_4'), v_{\ell}(c_6'), v_{\ell}(\Delta_m'))=
\begin{cases}
    (2, 2, 4), & \text{ if } \ell \geq 5,\\
     (4,5,4),\: \tilde{c}_4' \equiv 3 \pmod 4, \: \tilde{c}_6' \equiv 1 \pmod 4,& \text{ if } \ell=2;
\end{cases}
\]
this curve also has $e'=3$
by \cite[Proposition 1 and p. 358]{Kraus}.
\end{proof}

\begin{theorem}\label{thm:e3l3}
	Let $p \geq 5$ be a prime. Let $E/\Q_3$ be an elliptic curve with $e=3$ and satisfying $\tilde{\Delta}\equiv 2 \pmod 3$.
  Then $\Xp(\Q_3)\neq \emptyset$ if and only if $(3/p) = -1$.
\end{theorem}
\begin{proof}
Let $E/\Q_3$ be as in the statement.
If $X_E^-(p)(\Q_3)\neq\emptyset$, then there is an elliptic curve
$E'/\Q_3$ such that $E[p]$ and $E'[p]$ are anti-symplectically
isomorphic as $G_{\Q_3}$-modules. By
Lemma~\ref{lem:same-defect}, the curve $E'$ has potentially good
reduction and $e=e'=3$. Moreover, the isomorphism
$E[p]\simeq E'[p]$ implies $\Q_3(E[p])=\Q_3(E'[p]).$
This extension is non-abelian by Theorem~\ref{thm:F}~(2), since $\widetilde\Delta_m\equiv2\pmod3.$
Applying Theorem~\ref{thm:F}~(2) to $E'$, we obtain
\[
    \tilde\Delta'_m\equiv2\pmod3.
\]

Therefore, it follows from Theorem~\ref{T:mainWild3} that $(3/p)= -1$.

To prove the converse, suppose $(3/p)=-1$.
Let $E'/\Q_3$ be an elliptic curve with $e'=e=3$ and $\tilde{\Delta}'_m \equiv 2 \pmod 3$. From Theorem~\ref{T:mainWild3} we know that
$E[p] \simeq E'[p]$ as $G_{\Q_3}$-modules
and that such an isomorphism is anti-symplectic if and only if
$\tilde{c}_6 \not\equiv \tilde{c}'_6 \pmod 3$.

To complete the proof, we show that such a curve $E'$ exists.

As illustrated in the proof of Theorem~\ref{thm:e3l}, it suffices to give an example of $E'$ with $e'=3$ for each value of $\tilde{c}_6 \pmod 3 \in\{1, 2\}$.
These are given in the last two rows of Table~\ref{tab:abelian_examples}.
\end{proof}

We have the following consequence for the case of semistability defect $e=6$.

\begin{corollary}\label{cor:e6}  Let $\ell$ and~$p \geq 5$ be different primes.
 Let $E/\Q_\ell$ be an elliptic curve with $e=6$.
 Assume that either $\ell \equiv 2 \pmod 3$ or $\ell = 3$ and $\tilde{\Delta} \equiv 2 \pmod 3$.
Then $\Xp(\Q_\ell)\neq \emptyset$ if and only if $(3/p)= -1$ or $(\ell / p) = -1$.
\end{corollary}
\begin{proof} 
Let $E/\Q_\ell$ be as in the statement.
From \cite[Lemmas 1 and 2]{symplectic} there is $u$ such that the quadratic twist $E^u/\Q_\ell$ of~$E/\Q_\ell$ satisfies $e(E^u) = 3$.

If $\ell \equiv 2 \pmod 3$ then the claim follows from Theorem~\ref{thm:e3l} applied to $E^u/\Q_\ell$ and Lemma~\ref{lem:qtwists}.

Suppose $\ell = 3$ and $\tilde{\Delta} \equiv 2 \pmod 3$. The discriminant $\Delta_m(E^u)$ of a minimal model for $E^u$
satisfies $\Delta_m(E^u) = s^2 \Delta_m$, hence $\tilde{\Delta}(E^u) \equiv \tilde{\Delta} \equiv 2 \pmod{3}$. The conclusion now follows from Theorem~\ref{thm:e3l3} applied to $E^u/\Q_\ell$ and Lemma~\ref{lem:qtwists}.
\end{proof}

Finally, we cover the case of semistability defect $e=4$.

\begin{theorem}\label{thm:e4l}
	Let $\ell \equiv 3 \pmod{4}$ and~$p \geq 5$ be different primes.
	Let $E/\Q_\ell$ be an elliptic curve with $e=4$.
	Then $\Xp(\Q_\ell)\neq \emptyset$ if and only if $(2/p)= -1$ or $(\ell / p) = -1$.
\end{theorem}

\begin{proof}
Let $E/\Q_\ell$ be as in the statement. 
If $\Xp(\Q_\ell)\neq \emptyset$ then there is $E'/\Q_\ell$ such that
$E[p]$ and $E'[p]$ are anti-symplectically isomorphic $G_{\Q_\ell}$-modules. 
By Lemma~\ref{lem:same-defect}, the curve $E'$ has potentially good
reduction and $e=e'=4$.
Therefore, it follows from Theorem~\ref{T:mainTame4} that $(2/p)= -1$ or $(\ell / p) = -1$.

To prove the converse, suppose first $(2/p)= -1$.
From \cite[Lemma 19 (i)]{symplectic} we know that $E/\Q_\ell$ has a 2-torsion point, hence admits a $\Q_\ell$-isogeny of degree~2. Therefore, $\Xp(\Q_\ell)\neq \emptyset$ by Corollary~\ref{cor:isogeny}.

Suppose now $(2/p) = 1$ and $(\ell /p)= -1$.
Let $E'/\Q_\ell$ be any elliptic curve with $e'=e=4$. 
From Theorem~\ref{T:mainTame4}
we know that the $G_{\Q_\ell}$-modules $E[p]$ and $E'[p]$ are isomorphic and that such an isomorphism is
anti-symplectic if and only if $v_{\ell}(\Delta_m)\neq v_{\ell}(\Delta'_m)\pmod 4$.

To complete the proof, we show that such a curve $E'$ exists.

It is well known (e.g.~\cite[Table 5]{symplectic}) that if $\ell \neq 2$ and an elliptic curve $E/\Q_\ell$ has semistability defect $e=4$ then $v_\ell(\Delta_m(E))\in \{3,9\}$, so we have two cases. 

If
$E$ satisfies $v_\ell(\Delta_m)=3$, we take
\[
E'  / \Q_{\ell} \; : \;  y^2=x^3+\ell^2x^2-\ell^3x
\]
satisfying
\[
(v_{\ell}(c_4'), v_{\ell}(c_6'), v_{\ell}(\Delta_m'))=
\begin{cases}
(3,5,9), & \text{ if }\ell \geq 5,\\
(4,6,9),\: \tilde{\Delta}' \equiv 4 \pmod 9 & \text{ if }\ell=3,
\end{cases}
\]
In these cases, we get $e'=4$ by \cite[Proposition 1]{Kraus} for $\ell \geq 5$ and \cite[p. 356]{Kraus} for $\ell=3$.

If $v_\ell(\Delta_m)=9$, we take
\[
E' / \Q_{\ell} \; : \; y^2=x^3+\ell x^2-\ell x
\]
satisfying
\[
(v_{\ell}(c_4'), v_{\ell}(c_6'), v_{\ell}(\Delta_{m}'))=
\begin{cases}
    (1, 2, 3), & \text{ if }\ell\geq 5,\\

(2,3,3),\: \tilde{\Delta}'\equiv 4 \pmod 9 & \text{ if }\ell=3;
\end{cases}
\]
this curve also has $e'=4$
by \cite[Proposition 1, and p. 356]{Kraus}.

\end{proof}

\begin{theorem}\label{thm:e4l2}
	Let $p \geq 5$ be a prime. Let $E/\Q_2$ be an elliptic curve with~$e=4$ and satisfying $\tilde{c}_4 \equiv 5\tilde{\Delta} \pmod 8$.
  Then $\Xp(\Q_2)\neq \emptyset$ if and only if $(2/p) = -1$.
\end{theorem}
\begin{proof}
Let $E/\Q_2$ be as in the statement.
If $\Xp(\Q_2)\neq \emptyset$ then there is $E'/\Q_2$ such that
$E[p]$ and $E'[p]$ are anti-symplectically isomorphic $G_{\Q_2}$-modules. 
By Lemma~\ref{lem:same-defect}, the curve $E'$ has potentially good
reduction, $e'=e=4$, and the same inertial field as $E$. Moreover, $\Q_2(E[p])=\Q_2(E'[p]).$
Since this extension is non-abelian, Theorem~\ref{thm:F}~(3), applied to $E'$,
gives
\[
    \widetilde c'_4\equiv
    5\widetilde\Delta'_m\pmod8.
\]
The equality of the inertial fields also shows that $E$ and $E'$ acquire
good reduction over the same field $F_i$ appearing in
Theorem~\ref{thm:F}~(3). Therefore, it follows from Theorem~\ref{T:mainWild4} that $(2/p)= -1$.

To prove the converse, suppose $(2/p)=-1$.
From Theorem~\ref{thm:F}, either $E$ or $E^{(-1)}$ has good reduction over $F_1$.
By Lemma~\ref{lem:qtwists}, after replacing $E$ by $E^{(-1)}$ if needed, we can assume that $E$ has good reduction over $F = F_1$.

Let $E'/\Q_2$ be an elliptic curve such that $e'=e=4$, $\tilde{c}_4 \equiv 5\tilde{\Delta} \pmod 8$ and $E'/F$ has good reduction.
From Theorem~\ref{T:mainWild4} we have that $E[p]$ and $E'[p]$ are isomorphic $G_{\Q_2}$-modules and
such an isomorphism is anti-symplectic if and only if $\tilde{c}_6 \not\equiv \tilde{c}'_6 \pmod 4$.

To complete the proof, we show that such an elliptic curve $E'$ exists.

As illustrated in the proof of Theorem~\ref{thm:e4l}, it suffices to give an example of $E'$ with $e'=4$ for each value of $\tilde{c}_6 \pmod 4 \in\{1, 3\}$. These can be found in the first two rows
of Table~\ref{tab:abelian_examples}.
\end{proof}
\begin{table}[ht]
    \centering
    \begin{tabular}{|c||c|c|c|}
        \hline
        $(\ell, e)$ & $E'/\Q_\ell$ & $F$ & Additional conditions \\ \hline
        (2,4)& \LMFDBE{6912j1}  & $x^4 + 12x^2 + 6$ &$\tilde{c}_6'\equiv 1 \pmod 4$, $\tilde{c}_4' \equiv 5\tilde{\Delta}' \pmod 8$ \\ \hline
         (2,4)& \LMFDBE{6912l1}  & $x^4 + 12x^2 + 6$  & $\tilde{c}_6'\equiv 3 \pmod 4$, $\tilde{c}_4' \equiv 5\tilde{\Delta}' \pmod 8$ \\ \hline
        (3,3)&  \LMFDBE{25920z1} & $x^3+3x^2+3$ & $\tilde{c}_6'\equiv 1 \pmod 3$, $\tilde{\Delta}'\equiv 2 \pmod 3$ \\ \hline
        (3,3)&  \LMFDBE{25920v1} &$x^3+3x^2+3$  & $\tilde{c}_6'\equiv 2 \pmod 3$, $\tilde{\Delta}'\equiv 2 \pmod 3$ \\ \hline
    \end{tabular}
    \vspace{2mm}
   \caption{Examples
   of $E'/\Q_\ell$ with non-abelian $p$-torsion field and $\gcd(\ell,e)\neq 1$ used in the proofs of Theorems~\ref{thm:e3l3} and~\ref{thm:e4l2}; this table can be verified with the code given in \cite[Table 3]{git}.}
    \label{tab:abelian_examples}
\end{table}
\vspace{-8mm}
\section{The case of non-abelian inertia}

We keep the notations from Section~\ref{sec:non-ab}.
Let $E/\Q_\ell$ be an elliptic curve with potentially good reduction and non-abelian inertial group~$\Phi$. Thus $(\ell,e)\in \{(2,8),(2,24),(3,12)\}$ (see~\cite{Kraus}).

\subsection{The case \texorpdfstring{$\ell=2$ and semistability defect $e=8$ or $24$}{}} \label{s:e8e24}

Let $E/\Q_2$ be an elliptic curve of discriminant $\Delta_E$ with potentially good reduction and semistability defect~$e=8$ or $e=24$.
Recall that the inertial field~$L$ of $E$ satisfies
$L = \Q_2^{\text{un}}(E[p])$ for all $p \geq 3$.
Let $L_3=\Q_2(E[3])$ and $U_3 := (L_3)_{nr}$ be the maximal unramified subextension of $L_3$.
In particular, $E/L_3$ has good reduction.

\begin{lemma}\label{lem:l2non-ab}
    Let $E/\Q_2$, $L_3$ and $U_3$ as above. Then
    \[ U_3=\Q_2(\zeta_3), \qquad G_{U_3}=G_{L_3} \cdot I_2 \quad \text{ and } \quad \rhobar_{E,p}(\Frob_{L_3})=-2 \cdot \Id.\]
\end{lemma}
\begin{proof}

From \cite[Proposition 2]{doks} we see that
$\Gal(L_3/\Q_2)$ has order 16 if $e=8$ and order 48 if $e=24$, hence $[U_3 : \Q_2] = 2$ and $U_3 = \Q_2(\zeta_3)$ by the uniqueness of the unramified quadratic extension of $\Q_2$. In particular, $L_3$ has residue field $\F_4$.


Let $F$ be the field obtained from $\Q_2$ by adjoining the coordinates of one point of exact
order $3$ and a cube root of the discriminant~$\Delta_E$. From \cite[Lemma 2.1]{Cop2adic} we know that $E/F$ has good reduction and admits a model with residue curve $\Ebar: y^2+y=x^3$.

We have $\#\overline{E}(\F_4)=9$, thus $a_{L_3} = -4$ and $\Delta_{L_3} = a_{L_3}^2-4\cdot 4 = 0$.
Since $E/L_3$ has good reduction and $F \subset L_3$ by \cite[Lemma 6]{doks}, 
 we conclude $\rhobar_{E,p}(\Frob_{L_3})=-2 \cdot \Id$ by \cite[Theorem 2]{Centeleghe}.

The equality $G_{U_3}=G_{L_3} \cdot I_2$ follows directly from the definition of $U_3$.
\end{proof}

Let $\chi_3 : G_{\Q_2} \to \{ \pm 1 \}$ be the mod~$3$ cyclotomic character; it is the unique quadratic unramified character of $G_{\Q_2}$ 
and it satisfies $G_{U_3} = \ker \chi_3$.

\begin{proposition}\label{prop:l2non-ab}
    Let $E$ and $E'$ be two elliptic curves over $\Q_2$ with $e=e' \in \{8,24\}$
    and the same inertial field $L$.
    Then, after twisting $E[p]$ by $\chi_3$ if needed, we have that $E[p]$ and $E'[p]$ are isomorphic as $G_{\Q_2}$-modules.
\end{proposition}
\begin{proof} Since $E$ and $E'$ have the same inertial field, then $E[p]$ and $E'[p]$ are isomorphic as $I_2$-modules by \cite[Theorems 7 and 8]{symplectic}. Using the notation above, and applying
Lemma~\ref{lem:l2non-ab} to both curves, we get $U_3 = U_3'$ and
\[
  \quad G_{U_3}=G_{L_3} \cdot I_2  = G_{L'_3} \cdot I_2, \quad \text{ and } \quad
  \rhobar_{E,p}(\Frob_{L_3})=\rhobar_{E',p}(\Frob_{L'_3})=-2\cdot\Id.
\]
Both images of Frobenius commute with all matrices in $\GL_2(\F_p)$.
 Thus, the isomorphism of $I_2$-modules is also an isomorphism of $G_{U_3}$-modules.
The representations $\rhobar_{E,p}$ and $\rhobar_{E',p}$ are irreducible (because the action of inertia is already irreducible) and their restrictions to~$G_{U_3}$ are equal. Therefore, either
    \[
    \rhobar_{E,p} \simeq \rhobar_{E',p}\qquad\text{ or } \qquad\rhobar_{E,p} \otimes \chi_3 \simeq \rhobar_{E',p},
    \]
concluding the proof.
\end{proof}
\begin{theorem}\label{thm:l2n-abinertia}
    Let $p\geq 7$ and $E/\Q_2$ an elliptic curve with $e\in \{8,24\}$. Then, we have $\Xp(\Q_2)\neq \emptyset$ if and only if $(2/p)=-1$.
\end{theorem}
\begin{proof}

Let $E$ be as in the statement, and $L$ be its inertial field.
If $\Xp(\Q_2)\neq \emptyset$ then there is $E'/\Q_2$ such that
$E[p]$ and $E'[p]$ are anti-symplectically $G_{\Q_2}$-isomorphic. By Lemma~\ref{lem:same-defect}, the curve $E'$ has potentially good
reduction, semistability defect
$e'=e\in\{8,24\},$
and the same inertial field $L$ as $E$. Therefore, it follows from \cite[Theorem 7 and 8]{symplectic} that $(2/p)=-1$.

To prove the converse, suppose that $(2/p)= -1$.
Let $E'/\Q_2$ be any elliptic curve with $e'=e$ and the same inertial field $L$. From Proposition~\ref{prop:l2non-ab} we know that $E[p]$ is isomorphic to $E'[p]$
as $G_{\Q_2}$-modules up to a twist by $\chi_3$. So, by Lemma~\ref{lem:qtwists}, after replacing $E$ by $E^{-3}$ if needed, we can assume that $E[p] \simeq E'[p]$ as $G_{\Q_2}$-modules.
We divide into cases.

{\sc Case $e=8$.} By \cite[Theorem 9]{symplectic}, after twisting both curves by $2$ if necessary, we can assume that $E$ and $E'$ have either both conductor $2^5$ or both conductor $2^8$. Moreover, they are anti-symplectically isomorphic if and only if one of the following holds
\begin{itemize}
    \item[(A)] Both curves have conductor $2^5$ and are in different cases of \cite[Table 3]{symplectic};
    \item[(B)] Both curves have conductor $2^8$ and $\tilde{c}_4 \not\equiv \tilde{c}_4' \pmod4$.
\end{itemize}
We will now show that such a curve $E'$ exists for each possible inertial field $L$. These fields are classified by \cite[Table 10]{inertial}. We note that in {\it{loc. cit.}} $L'/\Q_2$ stands for the totally ramified field of degree $8$ over which $E$ acquires good reduction. Since $L=\Q_2^{\text{un}} \cdot L'$, these completely classify all the inertial fields. 
As noted above, it is enough to consider only the elliptic curves $E$ with conductor $2^5$ or $2^8$. Moreover,
by the same table and Lemma~\ref{lem:qtwists}, if $N_E=2^8$, after twisting $E$ by~$2$ if needed,
we can assume that  $E$ has $L'=\LMFDBL{2.8.24.66}$. Hence, the possibilities of inertial types that we need to consider are given by $L' \in \{\LMFDBL{2.8.16.65}, \LMFDBL{2.8.16.66}, \LMFDBL{2.8.24.66}\}$.

For example, if $E$ has $L'= \LMFDBL{2.8.16.65}$ and it is in Case~(a) of \cite[Table 3]{symplectic} we take $E'=\LMFDBE{2592f1}$. One checks that $E'$ has $e'=8$ by \cite[p. 358-359]{Kraus}, acquires good reduction over $L'$ and belongs to Case~(b) of \cite[Table 3]{symplectic}.

From the above it suffices to give examples of $E'$ with $e'=8$ and
\begin{itemize}
    \item[(A')]  $N_{E'}=2^5$, field $L'\in \{\LMFDBL{2.8.16.65}, \LMFDBL{2.8.16.66}\}$, one in each case of \cite[Table 3]{symplectic},
    \item[(B')] $N_{E'}=2^8$, field $L'=\LMFDBL{2.8.24.66}$, one for each value of $\tilde{c}_4'  \pmod 4\in \{\pm1\}$.
\end{itemize}
Such examples can be found in Table~\ref{tab:l2examples}, concluding the proof of this case.

{\sc Case $e=24$.} By \cite[Theorem 7]{symplectic}, the curves $E$ and $E'$ are anti-symplectically isomorphic if and only if $v_2(\Delta_m)\not\equiv v_2(\Delta_m') \pmod 3$.

We will now show that such a curve $E'$ exists for each possible inertial field $L$. These fields are classified in \cite[Table 17]{inertial}. We note that in this case, $L'/\Q_2$ stands for the totally ramified field of degree $8$ whose splitting field, call it $F$, is totally ramified of degree $24$ and is such that $E/F$ has good reduction. Since $L=\Q_2^{\text{un}}\cdot F$, these completely classify all the inertial fields. The same table shows that, after twisting $E$ and $E'$ by a quadratic character if needed, we can assume that both curves have $L'\in\{\LMFDBL{2.8.10.2}, \LMFDBL{2.8.22.132}\}$.

As illustrated before, for each of these $L'$, it suffices to give an example of $E'$ with $e'=24$ for each value of $v_2(\Delta_m') \pmod 3 \in\{1, 2\}$.
These are also given in Table~\ref{tab:l2examples}.
\end{proof}
\begin{table}[h]
    \centering
    \begin{tabular}{|c||c|c|c|c|}
        \hline
        $e$ & $L'$ & $N_E = N_{E'}$ & $E'/\Q_2$ & Additional  conditions  \\ \hline
        8& \LMFDBL{2.8.16.65} &$2^5$ & \LMFDBE{96a1} & Case~(a) of $(*)$
        \\ \hline
        8& \LMFDBL{2.8.16.65} & $2^5$&  \LMFDBE{2592f1}& Case~(b) of $(*)$
        \\ \hline
           8&\LMFDBL{2.8.16.66} & $2^5$&\LMFDBE{288a1} & Case~(a) of $(*)$
           \\ \hline
        8&\LMFDBL{2.8.16.66} & $2^5$& \LMFDBE{288e2}& Case~(b) of $(*)$\\ \hline
        8&\LMFDBL{2.8.24.66} & $2^8$& \LMFDBE{256b2} & $\tilde{c}_4'\equiv 1\pmod4$ \\ \hline
        8& \LMFDBL{2.8.24.66}& $2^8$& \LMFDBE{2304a2} &$\tilde{c}_4'\equiv -1\pmod4$ \\ \hline
        24& \LMFDBL{2.8.10.2} &$2^3$& \LMFDBE{648b1} & $v_2(\Delta_m' )\equiv 1\pmod 3$\\ \hline
        24& \LMFDBL{2.8.10.2} &$2^3$ & \LMFDBE{6696q1} & $v_2(\Delta_m' )\equiv -1\pmod 3$\\ \hline
        24& \LMFDBL{2.8.22.132}  &$2^7$ & \LMFDBE{3456a1} & $v_2(\Delta_m' )\equiv 1\pmod 3$\\ \hline
        24& \LMFDBL{2.8.22.132}   &$2^7$ & \LMFDBE{289152l1} & $v_2(\Delta_m' )\equiv -1\pmod 3$\\ \hline
    \end{tabular}
    \caption{Examples of $E'/\Q_2$ with $e=8$ or $24$ used in the proof of Theorem~\ref{thm:l2n-abinertia}. Here $(*)$ stands for \cite[Table 4]{symplectic}; this table can be verified with the code given in \cite{git}.}
    \label{tab:l2examples}
\end{table}

\subsection{The case \texorpdfstring{$\ell=3$ and semistability defect $e=12$}{}}

Let $E/\Q_3$ be an elliptic curve of discriminant $\Delta_E$ having potentially good reduction with semistability defect~$e=12$.
The inertial field~$L$ of $E$ satisfies
$L = \Q_3^{\text{un}}(E[p])$ for all $p \geq 5$.
We define also
\[
F:=\Q_3(E[2], \Delta_E^{1/4}), \quad F_0:=F(\zeta_4) \quad \text{ and } \quad U:=\Q_3(\zeta_4).
\]
Let $\chi_4 : G_{\Q_3} \to \{ \pm 1 \}$ be the mod~$4$ cyclotomic character; it is the unique quadratic unramified character of $G_{\Q_3}$ and it satisfies $G_U = \ker \chi_4$.
\begin{lemma}\label{lem:l3non-ab}
    Let $E/\Q_3$, $F_0$ and $U$ be as above. Then
    \[
     G_{U}=G_{F_0} \cdot I_3 \qquad \text{ and } \quad \rhobar_{E,p}(\Frob_{F_0})=-3 \cdot \Id.
    \]
\end{lemma}
\begin{proof}
    By the work of Kraus \cite[Corollaire to Lemme 3]{Kraus}, it follows that $L=\Q_3^{\text{un}}\cdot F$.
    Since $L/F$ is unramified, it follows that $E/F$ has good reduction and
    $F$ is totally ramified of degree~$12$ (see \cite[Remark 3.1]{Cop3adic}). Thus $F_0$ has residue field~$\F_9$ and $G_{U}=G_{F_0} \cdot I_3$.

    From \cite[Lemma 3.4]{Cop3adic}, there is a model of $E/F$ with good reduction and residual curve $\overline{E}: y^2=x^3-x$. We have $\# \overline{E}(\F_9)=16$, so $a_{F_0} = -6$ and $\Delta_{F_0} = 0$, 
thus $\rhobar_{E,p}(\Frob_{F_0})=-3 \cdot \Id$ by \cite[Theorem 2]{Centeleghe}.
\end{proof}

\begin{proposition}\label{prop:l3non-ab}
Let $E$ and $E'$ be two elliptic curves over $\Q_3$ with $e=e'=12$
    and the same inertial field $L$.
    Then, after twisting $E[p]$ by $\chi_4$ if needed, we have that $E[p]$ and $E'[p]$ are isomorphic as $G_{\Q_3}$-modules.
\end{proposition}

\begin{proof}
This follows as the proof of Proposition~\ref{prop:l2non-ab} with $L_3$ and $U_3$ replaced by $F_0$ and $U$.
\end{proof}
\begin{theorem}\label{thm:l3n-abinertia}
    Let $p\geq 7$ and $E/\Q_3$ be an elliptic curve with 
    semistability defect $e=12$. Then, $\Xp(\Q_3)\neq \emptyset$ if and only if $(3/p)=-1$.
\end{theorem}
\begin{proof}
Let $E$ be as in the statement, and $L$ its inertial field.
If $\Xp(\Q_3)\neq \emptyset$ then there is $E'/\Q_3$ such that
$E[p]$ and $E'[p]$ are anti-symplectically isomorphic $G_{\Q_3}$-modules. By Lemma~\ref{lem:same-defect}, the curve $E'$ has potentially good
reduction, semistability defect $e'=e=12,$
and the same inertial field $L$ as $E$.
Therefore, it follows from \cite[Theorem 10]{symplectic} that $(3/p)=-1$.

To prove the converse, suppose that $(3/p)= -1$.
Let $E'/\Q_3$ be any elliptic curve with $e'=e=12$ and the same inertial field $L$. From Proposition~\ref{prop:l3non-ab} we know that $E[p] \simeq E'[p]$ as $G_{\Q_3}$-modules up to a twist by $\chi_4$. So, by Lemma~\ref{lem:qtwists}, after replacing $E$ by $E^{-1}$ if needed, we can assume that $E[p] \simeq E'[p]$ as $G_{\Q_3}$-modules. 

By \cite[Theorem 11]{symplectic}, the curves $E$ and $E'$ are anti-symplectically isomorphic if and only if they are in different cases of \cite[Table 4]{symplectic}.

We will now show that such a curve $E'$ exists for each possible inertial field $L$. These fields are classified in \cite[Table 4]{inertial}. We note that in {\it{loc. cit.}} $L'/\Q_3$ stands for the totally ramified field of degree $12$ over which $E$ acquires good reduction. Since $L=\Q_3^{\text{un}} \cdot L'$, these completely classify all the inertial fields.

As in the proof of Theorem~\ref{thm:l2n-abinertia}, for each such $L'$, it is enough to give two examples of curves $E'$ with $e=12$ and good reduction over $L'$, in different cases of \cite[Table 4]{symplectic}; we included these examples in Table~\ref{tab:l3examples}.
\begin{table}[h]
    \centering
    \begin{tabular}{|c||c|c|c|c|}
        \hline
        $e$ & $L'$ & $N_E = N_{E'}$ & $E'/\Q_3$ & Case of $(*)$
        \\ \hline
        12& \LMFDBL{3.12.15.1} &$3^3$ & \LMFDBE{1728v1} & Case~(a) \\ \hline
       12& \LMFDBL{3.12.15.1} &$3^3$ & \LMFDBE{27a1} & Case~(b) \\ \hline
    12& \LMFDBL{3.12.15.12} &$3^3$ & \LMFDBE{12096dd1} & Case~(a) \\ \hline
    12& \LMFDBL{3.12.15.12} &$3^3$ & \LMFDBE{54a1} & Case~(b) \\ \hline
 12& \LMFDBL{3.12.23.122} &$3^5$ & \LMFDBE{972b1} & Case~(c)  \\ \hline
 12& \LMFDBL{3.12.23.122} &$3^5$ & \LMFDBE{388800oh1} & Case~(d)  \\ \hline
12& \LMFDBL{3.12.23.20} &$3^5$ & \LMFDBE{15552c2} & Case~(c)  \\ \hline
    12& \LMFDBL{3.12.23.20} &$3^5$ & \LMFDBE{243b1} & Case~(d)  \\ \hline
12& \LMFDBL{3.12.23.14} &$3^5$ & \LMFDBE{243a1} & Case~(c)  \\ \hline
    12& \LMFDBL{3.12.23.14} &$3^5$ & \LMFDBE{15552bp2} & Case~(d)  \\ \hline
    \end{tabular}
    \caption{Examples of $E'/\Q_3$ with $e=12$ used in the proof of Theorem~\ref{thm:l3n-abinertia}. Here $(*)$ stands for \cite[Table 4]{symplectic}; this table can be verified with the code given in \cite[Table 5]{git}.}
    \label{tab:l3examples}
\end{table}
\end{proof}
\newpage
\section{The case of \texorpdfstring{$\ell$}{l} = \texorpdfstring{$p$}{p}}

Some of the results proved so far imply the existence of $\Q_p$-points on $\Xp$ in many cases when $E/\Q_p$ is semistable. In this section, we will analyze the cases where $E/\Q_p$ has potentially good reduction, which will allow us to give a complete classification. 

Let $p\geq 7$ be a fixed prime. 
In this section, we consider $E/\Q_p$ an elliptic curve with potentially good reduction. We set $\delta:=v_p(\Delta_m)$, where $\Delta_m$ is as the minimal discriminant of $E$.
Since $p \geq 7$, the semistability defect $e\in \{1,2,3,4,6\}$. Recall that $e$ is the denominator of ${\delta}/{12}$.
We fix $p^{1/12}\in \overline \Q_p$ a $12$th root of $p$ and $M:=\Q_p(p^{1/12})$, which is totally ramified of degree $12$. Note that $E$ acquires good reduction over $M$.
We denote the maximal ideal of $\mathcal{O}_M$ by~$\mathfrak{p}$.

Let $I_p\subseteq G_{\Q_p}$ denote the inertia subgroup. We denote by $\omega:=\chi_p|_{I_p}$, where $\chi_p$ denotes the mod~$p$ cyclotomic character. Moreover, let $\omega_2$ and $\omega_2'$ be the fundamental characters of level~$2$; see \cite[p. 7]{Kraus1997}. We have $\omega_2'=\omega_2^p$ and $\omega_2 \omega_2'=\omega.$

\subsection{The Weil pairing} 
We recall that the Weil pairing may be defined intrinsically over
an arbitrary base scheme. More precisely, for $\mathcal A/B$ an
elliptic scheme, the Cartier duality together with the canonical principal
polarisation gives a canonical bilinear, alternating and perfect (in the sense of Cartier duality) pairing,
\[
    e_{\mathcal A,p}\colon
    \mathcal A[p]\times_B\mathcal A[p]
    \longrightarrow
    \mu_{p,B};
\]
see \cite[\S2.8]{KatzMazur}. This construction is compatible with
arbitrary base change. In particular, for an elliptic curve
$\overline E/\F_p$, the pairing
\[
    e_{\overline E,p}\colon
    \overline E[p]\times_{\F_p}\overline E[p]
    \longrightarrow
    \mu_{p,\F_p}
\]
is defined intrinsically, independently of any lift to characteristic
zero. 

Now let $E/M$ have good reduction and let $\mathcal E/\mathcal O_M$
be its N\'eron model. Since $\mathcal E$ is then an elliptic scheme, the
same construction gives a perfect alternating pairing
\[
    \widetilde e_{E,p}\colon
    \mathcal E[p]\times_{\mathcal O_M}\mathcal E[p]
    \longrightarrow
    \mu_{p,\mathcal O_M}.
\]
By compatibility with base change, its generic fibre is the usual Weil
pairing on $E[p]$, while its special fibre is precisely $e_{\overline E,p}$. 

Let
\(
    \psi\colon
    \overline E[p]\to\overline E'[p]
\)
be an isomorphism of finite flat group schemes over $\F_p$. 

There is a unique $r(\psi)\in\F_p^\times$ such that
\[
    e_{\overline E',p}\circ(\psi\times\psi)
    =
    [r(\psi)]\circ e_{\overline E,p}.
\]

We call $r(\psi)$ the \emph{multiplier} of $\psi$. We say that $\psi$
is \emph{symplectic} if $r(\psi)$ is a square in $\F_p^\times$, and
\emph{anti-symplectic} if $r(\psi)$ is a nonsquare. 
\begin{proposition}
\label{prop:reduction}
Let \(p>13\) be a prime.
Let \(E,E'/M\) be elliptic curves with good reduction, with N\'eron models \(\mathcal{E}\) and \(\mathcal{E}'\). Then every \(G_M\)-module isomorphism
\(
\varphi:E[p]\to E'[p]
\)
extends uniquely to an isomorphism of finite flat group schemes
\(
\widetilde{\varphi}:\mathcal{E}[p]\to\mathcal{E}'[p]
\)
over \(\mathcal{O}_M\). Its special fibre is an isomorphism
\[
\overline{\varphi}:\overline{E}[p]\to\overline{E}'[p]
\]
of finite group schemes over \(\mathbb{F}_p\).
Let $r(\varphi)\in\F_p^\times$ be the multiplier defined as in
\eqref{eq:defr}, with $K=M$. Then
\[
    e_{\overline E',p}
    \circ
    (\overline\varphi\times\overline\varphi)
    =
    [r(\varphi)]\circ e_{\overline E,p}
\]
as bilinear pairings of finite flat group schemes over $\F_p$.
In particular, \(\overline{\varphi}\) and \(\varphi\) have the same symplectic type.
\end{proposition}

\begin{proof}
Because \(E\) and \(E'\) have good reduction, their N\'eron models \(\mathcal{E}\) and \(\mathcal{E}'\) are abelian schemes over the DVR \(\mathcal{O}_M\). Consequently the \(p\)-torsion group schemes \(\mathcal{E}[p]\) and \(\mathcal{E}'[p]\) are finite flat over \(\mathcal{O}_M\) of rank \(p^2\) \cite[Proposition~20.7]{MilneAbelianVarieties}.

Since $M$ has characteristic zero, the generic fibres $E[p]$ and
$E'[p]$ are finite \'etale group schemes.
Thus, the \(G_M\)-module isomorphism \(\varphi\) is therefore an isomorphism between the generic fibres of these two finite flat group schemes. Since the ramification index of $M$ is \(e_{\text{ram}}(M/\mathbb{Q}_p)=12<p-1\), \cite[Corollaire~3.3.6]{Raynaud} implies that \(\varphi\) extends uniquely to a morphism of finite flat group schemes
\[
\widetilde{\varphi}:\mathcal{E}[p]\longrightarrow\mathcal{E}'[p]
\]
over \(\mathcal{O}_M\). Applying the same theorem to \(\varphi^{-1}\) produces a morphism in the opposite direction. By uniqueness the two compositions are the identity morphisms, so \(\widetilde{\varphi}\) is an isomorphism. Base change to the residue field \(\mathbb{F}_p\) yields the isomorphism \(\overline{\varphi}\) of special fibres.

Consider the two morphisms of finite flat group schemes
\[
f:=\widetilde{e}_{E',p}\circ(\widetilde{\varphi}\times\widetilde{\varphi})
\qquad\text{and}\qquad
g:=[r(\varphi)]\circ\widetilde{e}_{E,p}
\]
from \(\mathcal{E}[p]\times_{\mathcal{O}_M}\mathcal{E}[p]\) to \(\mu_{p,\mathcal{O}_M}\). By the pointwise definition of \(r(\varphi)\) these two morphisms agree on the generic fibre. 
Since
$\mathcal E[p]\times_{\mathcal O_M}\mathcal E[p]$
is flat over $\mathcal O_M$, its generic fibre is schematically dense.
As the generic fibre is quasi-compact and $\mu_{p,\mathcal O_M}$ is
separated, the morphisms $f$ and~$g$ agree over $\mathcal O_M$; 
see \cite[Tags~\texttt{081H}, \texttt{01RD}, \texttt{01RH}]{Stacks}.

Taking special fibres and using compatibility of the Weil pairing with
base change gives
\[
    e_{\overline E',p}
    \circ
    (\overline\varphi\times\overline\varphi)
    =
    [r(\varphi)]\circ e_{\overline E,p}.
\]
Thus the multiplier of $\overline\varphi$ is $r(\varphi)$. In
particular, $r(\overline\varphi)$ and $r(\varphi)$ have the same square
class, so $\overline\varphi$ and $\varphi$ have the same symplectic
type. 
\end{proof}
\begin{remark} \label{rmk:reduction}
    Note that the proof of Proposition~\ref{prop:reduction} works also for
    \begin{itemize}
        \item $p=7$, $M=\Q_7$ and both curves $E$,$E'$ with good reduction, because $e=1<p-1=6$; 
        \item $p=11$, $M=\Q_{11}(\sqrt[3]{11})$ and both curves $E$,$E'$ with $e=e'=3$ because $e=3<p-1=10$.
    \end{itemize}
\end{remark}

We will need a version of Lemma~\ref{lem:isogeny} that applies for the $p$-torsion group schemes of elliptic curves over~$\F_p$. We record it here for future reference.

\begin{lemma}\label{lem:isogenySchemes} 
Let $p \geq 3$ be a prime. Let $E/\F_p$ and $E'/\F_p$ be elliptic curves and $\phi:E\to E'$ an $\F_p$-isogeny of degree $n$ coprime to $p$. Then $\phi$ induces an
isomorphism of finite flat group schemes
$\phi|_{E[p]}:E[p]\to E'[p]$. Moreover, $\phi|_{E[p]}$
is symplectic 
if $\left(\frac{n}{p}\right)=1$
and anti-symplectic otherwise.
\end{lemma}
\begin{proof} The proof is analogous to that of Lemma~\ref{lem:isogeny}, where we use functoriality of the 
scheme-theorectic Weil pairing (see~\cite[\S2.8]{KatzMazur}) and the fact that $p \nmid n$. 
\end{proof}

\subsection{Dieudonn\'e Modules}\label{sec:Dieudonne}
In this section, we recall a few facts about Dieudonn\'e Theory; a standard reference is \cite[\S 1.4]{CCO}. We will state the facts only in the simpler case of~$\F_p$ which is enough for our purpose.
The ring of Witt vectors $W(\F_p)$ is isomorphic to $\Z_p$ and its Frobenius is the identity. We let $\mathbb{D} \simeq \Z_p[F,V]/\langle FV-p \rangle$ be the Dieudonn\'e ring of $\F_p$.

By~\cite[Theorem~1.4.3.2]{CCO}, there is an exact additive anti-equivalence
\[
 G\longmapsto D(G)
\]
between the category of finite commutative group schemes over $\F_p$ of
$p$-power order and the category of left $\mathbb D$-modules of
finite $\Z_p$-length. In particular, for finite commutative group
schemes of $p$-power order $G$ and $H$, there is a natural isomorphism
\[
 \Hom_{\F_p}(G,H)
 \simeq
 \Hom_{\mathbb D}\bigl(D(H),D(G)\bigr).
\]
Thus the Dieudonn\'e functor is fully faithful, with the direction of
morphisms reversed.

If $G$ is killed by $p$, then $D(G)$ is also killed by $p$ and is
therefore naturally a vector space over $\F_p$.
In this case, if $G$ has order $p^r$, then
\(
 \dim_{\F_p} D(G)=r.
\)
The functor is also compatible with Cartier duality: there is a
natural identification
\[
 D(G^\vee)\simeq D(G)^\vee.
\]

We apply this correspondence 
\(
 G=\overline E[p],
\)
where $\overline E/\F_p$ is an elliptic curve. Since
$\overline E[p]$ has order $p^2$ and is killed by $p$, its
Dieudonn\'e module
\[
 D:=D(\overline E[p])
\]
is a two-dimensional $\F_p$-vector space. 

The principal polarization identifies $\overline E[p]$ with its
Cartier dual. Applying the contravariant Dieudonn\'e functor and
using the natural identification
\(
 D(\overline E[p]^\vee)\simeq D(\overline E[p])^\vee
\)
gives a nondegenerate alternating pairing
\begin{equation}\label{eq:l=ppair}
 \langle\ ,\ \rangle:D\times D\longrightarrow D(\mu_p)\simeq\F_p.
\end{equation}
Compatibility with Cartier
duality implies that the multiplier of an
automorphism of $\overline E[p]$ agrees with the
multiplier of the corresponding automorphism of $D$; with the
contravariant convention, the multiplier may be replaced by its
inverse.

\begin{prop}\label{prop:ss-automorphism}
Let $p\geq5$ and let $ E/\F_p$ be a supersingular elliptic curve.
Every automorphism of the finite group scheme $ E[p]$ defined
over~$\F_p$ is symplectic.  In particular, $ E[p]$ has no
anti-symplectic $\F_p$-automorphisms.
\end{prop}

\begin{proof}
Let $D( E[p])$
be the Dieudonn\'e module of $ E[p]$.  It is a $2$-dimensional
$\F_p$-vector space, as discussed above. Since $a_p=0$, we have
$\pi^2=[-p]$.  Consequently, the endomorphism $N=D(\pi|_{E[p]})$ is nonzero (e.g. $\#\ker(\pi|_{E[p]})=p$) and
satisfies $N^2=D([-p]|_{E[p]})=0$.  In a suitable basis,
\[
 N=\begin{pmatrix}0&1\\0&0\end{pmatrix}.
\]
Let $\psi \in \Aut_{\F_p}(E[p])$ and note that it commutes with $\pi|_{E[p]}$. Therefore its matrix $A:=D(\psi)$ commutes with $N$ and so has the form
\[
 A=\begin{pmatrix}a&b\\0&a\end{pmatrix},
 \qquad a\in\F_p^*.
\]
In dimension~$2$, the multiplier of $A$ is
$\det(A)=a^2$.
Under the contravariant Dieudonn\'e correspondence, the multiplier
$r(\psi)$ is equal to $\det(A)$ or its inverse, depending on conventions.
In either case, $r(\psi)$ is a square in $\F_p^\times$, and hence
$\psi$ is symplectic.
\end{proof}

\subsection{A lemma in representation theory} We will need the following result below.

\begin{lemma}\label{lem:niveau-two-representations}
Let $p > 2$ be a prime. Let
$\rho_1,\rho_2:G_{\Q_p}\longrightarrow \GL_2(\F_p)$
be continuous representations with equal determinants.
Suppose that, for some  $k \in \Z$, we have $\omega_2^k\neq\omega_2^{pk}$ and
\[
 \left.\rho_i\right|_{I_p}\otimes_{\F_p}\F_{p^2}
 \simeq
 \omega_2^k\oplus\omega_2^{pk},
 \qquad i=1,2.
\]
Then $\rho_1\simeq\rho_2$.
\end{lemma}

\begin{proof}
Set $\theta=\omega_2^k$. By assumption, we have $\theta\neq\theta^p$ and, after extending scalars to $\F_{p^2}$, there is a basis such that, for all $\sigma \in I_p$, we have
\[
 \rho_i(\sigma)=
 \begin{pmatrix}
  \theta(\sigma)&0\\
  0&\theta^p(\sigma)
 \end{pmatrix}.
\]
Let $\sigma\in I_p$ be a generator of the tame inertia
quotient, and let $\tau\in G_{\Q_p}$ be a lift of arithmetic
Frobenius. Recall the tame relation
$\tau\sigma\tau^{-1}=\sigma^p$. Since $\theta \neq \theta^p$, it follows that conjugation by $\rho_i(\tau)$ exchanges the 
eigenspaces corresponding to $\theta$ and $\theta^p$. Therefore,
\[
 \rho_i(\tau)=
 \begin{pmatrix}
  0&a_i\\
  b_i&0
 \end{pmatrix}
\]
for some $a_i,b_i\in\F_{p^2}^{\times}$. From the equality of determinants, we have $a_1b_1=a_2b_2$.
It follows that the matrix
$
 D=\left(
 \begin{smallmatrix}
  a_2/a_1&0\\
  0&1
 \end{smallmatrix}\right)
$ 
centralizes image of inertia and satisfies
$D\rho_1(\tau)D^{-1}=\rho_2(\tau)$.
Thus
\[
 D\rho_1(g) =  \rho_2(g) D, \quad \text{ for all } g \in G_{\Q_p},
\]
showing the representations are isomorphic over $\F_{p^2}$.
It remains to descend this isomorphism to $\F_p$. 
Let $P_i\in\GL_2(\F_{p^2})$ be the matrix that transforms $\rho_i$ from the original
$\F_p$-basis to the matrix description above. Then
 $A:=P_2DP_1^{-1}\in\GL_2(\F_{p^2})$ satisfies
\[
 A\rho_1(g)=\rho_2(g)A
 \qquad\text{for all }g\in G_{\Q_p}
\]
where now $\rho_i(g) \in \GL_2(\F_p)$ for all $g \in G_{\Q_p}$.
Let $\delta$ be the nontrivial element of
$\Gal(\F_{p^2}/\F_p)$. We also have  $\delta(A)\rho_1(g)=\rho_2(g)\delta(A)$ 
for all $g\in G_{\Q_p}$. 
Note that the matrix description above also shows the $\rho_i$ are absolutely irreducible. Therefore, by Schur's
lemma, there exists $c\in\F_{p^2}^{\times}$ such that
$\delta(A)=cA$. Applying $\delta$ again gives
$c\,\delta(c)=1$. By Hilbert's Theorem~90, there is
$\lambda\in\F_{p^2}^{\times}$ such that
$c=\frac{\lambda}{\delta(\lambda)}$.
Consequently,
$\delta(\lambda A)=\lambda A$,
so that $\lambda A\in\GL_2(\F_p)$. Hence $\lambda A$ gives an
isomorphism $\rho_1\simeq\rho_2$ over $\F_p$.
\end{proof}

\subsection{Non-existence of $\Q_p$-points on $\Xp$}
We will prove the following result.

\begin{theorem}\label{thm:l=pss}
Let $p>7$ be a prime and let $E/\Q_p$
have potentially good supersingular reduction such that
\begin{itemize}
    \item[(1)] $p\equiv7\pmod8$, and $e \in \{1,3,4\}$;
    \item[(2)] $p \equiv 11 \pmod {12}, p>11$, and $e=3$;
    \item[(3)] $p=11$, $e=3$ and $v_{11}(\Delta_m)\ne 8$ or $v_{11}(c_4)<4$.
\end{itemize}
Then \(X_E^-(p)(\Q_p)=\emptyset.\)
\end{theorem}

We will need various preparatory lemmas.

\begin{lemma}\label{lem:l=prep}
Let $p\geq7$ be a prime and $E/\Q_p$ an elliptic curve
with potentially good supersingular reduction. Then, for any character $\epsilon : G_{\Q_p} \to \F_p^\times$ of order dividing $2$, we have
\[
 \rhobar_{E,p}\not\simeq
 \begin{pmatrix}
  \epsilon\chi_p&*\\
  0&\epsilon
 \end{pmatrix},
\]
 where $\chi_p$ is the usual mod~$p$ cyclotomic character.
\end{lemma}
\begin{proof} For a contradiction, suppose that
$\overline{\rho}_{E,p}\simeq
 \left(\begin{smallmatrix}
  \epsilon\chi_p&*\\
  0&\epsilon
 \end{smallmatrix}\right)$ 
for some character $\epsilon$ at most quadratic.
Observe that the restriction to inertia of the quotient of the diagonal characters is $\omega$ or $\omega^{-1}$. We will show that this is incompatible with $E$ having potentially good supersingular reduction. 

If \(E\) has good supersingular reduction, then
\(\overline{\rho}_{E,p}|_{I_p}\) is irreducible  
by~\cite[Proposition~12]{Serre2}, contradicting our assumption. 
Suppose now that \(E\) has additive potentially good supersingular
reduction.
By \cite[Proposition~2 and Lemme 2]{Kraus1997}, the action of $I_p$ is
either irreducible, in which case we again obtain an immediate
contradiction, or \(\delta \in\{2,3,4,8,9,10\}\) and
\begin{equation} \label{eq:Ipss2}
 \left(\overline{\rho}_{E,p}|_{I_p}\right)^{\mathrm{ss}}
 \simeq
 \begin{cases}
   \omega^{1-\alpha}\oplus\omega^\alpha,
       & \delta<6,\\[2mm]
   \omega^{2-\alpha}\oplus\omega^{\alpha-1},
       & \delta>6,
 \end{cases}
 \qquad
 \alpha=\frac{(p+1)\delta}{12}.
\end{equation}
Here the integrality of \(\alpha\) follows from the assumption that
the acquired good reduction is supersingular. Indeed, by \cite[Lemme 1]{Kraus1997}, for \(\delta\in\{2,4,8,10\}\), we also have $p \equiv 2 \pmod{3}$ whilst for \(\delta\in\{3,9\}\), we have $p \equiv 3 \pmod{4}$.
For each~$\delta$, computing the difference of the exponents of~$\omega$ modulo $p-1$ yields
\[
\begin{cases}
-\dfrac{p-2}{3}, & \text{if } \delta\in\{2,8\},\\[2mm]
 \dfrac{p-2}{3}, & \text{if } \delta\in\{4,10\},\\[2mm]
 \dfrac{p-1}{2}, & \text{if } \delta\in\{3,9\}.
\end{cases}
\]
Clearly all the possibilities are not congruent to $\pm 1$ modulo~$p-1$, a contradiction.
\end{proof}

\begin{lemma}\label{lem:good-companion}
Let $p\geq11$ and let $E/\Q_p$ have potentially good supersingular reduction.  Let
$E'/\Q_p$ be an elliptic curve satisfying $E[p]\simeq E'[p]$ as $G_{\Q_p}$-modules, then $E'$
also has potentially good supersingular reduction.
\end{lemma}
\begin{proof}
Suppose first that $E'$ has potentially multiplicative reduction. By
Tate uniformization, we have $\overline{\rho}_{E',p}\simeq
 \left(\begin{smallmatrix}
  \epsilon\chi_p&*\\
  0&\epsilon
 \end{smallmatrix}\right)$ 
with $\epsilon$ at most quadratic. Since
$E[p]\simeq E'[p]$, the same would hold for
$\overline{\rho}_{E,p}$, contradicting
Lemma~\ref{lem:l=prep}. Hence $E'$ has potentially good
reduction.

Suppose, for a contradiction, that $E'$ has potentially
good ordinary reduction. 
If $E'$ has good ordinary reduction, it is well known that 
$\left(\overline{\rho}_{E',p}|_{I_p}\right)^{\mathrm{ss}}
 \simeq \omega \oplus 1$; if it has additive potentially good ordinary reduction, then
$\overline{\rho}_{E',p}$ is also reducible and \cite[Proposition 1]{Kraus1997} gives  
\begin{equation} \label{eq:Ipss3}
\left(\overline{\rho}_{E',p}|_{I_p}\right)^{\mathrm{ss}}
 \simeq \omega^{1-\alpha'}\oplus\omega^{\alpha'},
 \end{equation}
 where $\alpha'=(p-1)\delta'/12$.
Since $\overline{\rho}_{E,p}|_{I_p}$ is irreducible when $E$ has good supersingular reduction, we may therefore assume that $E$ has
additive potentially good reduction.
Moreover, $\overline{\rho}_{E,p}|_{I_p}$ is reducible and must have the shape given in~\eqref{eq:Ipss2} 
by~\cite[Proposition 2 and Lemme 2]{Kraus1997}. 
We will show this is impossible for $p\geq 11$.

Indeed, taking 12th powers of the characters in \eqref{eq:Ipss2} and \eqref{eq:Ipss3} we conclude 
\[
\{\omega^{4}, \omega^{8} \} = \{1,\omega^{12}\}  \quad \text{ for } \delta\in\{2,4,8,10\}
\]
and
\[
\{\omega^{6}, \omega^{6} \} = \{1,\omega^{12}\}  \quad \text{ for } \delta\in\{3,9\}
\]
which is impossible for $p \geq 11$ as the order of $\omega$ is $p-1$.\end{proof}

\begin{lemma}\label{lem:ss-isogeny-degree}
Let $p\geq5$ be a prime such that $p\equiv7\pmod8$.  Let
$E, E'/\F_p$ be supersingular elliptic curves, and let
$\psi: E\rightarrow E'$ be an $\F_p$-isogeny of
degree $n$ coprime to~$p$.  Then
\(
 ({n}/{p})=1.
\)
\end{lemma}
\begin{proof}

Note that $(2/p)=1$ because $p\equiv7\pmod8$. We will show that $(q/p) = 1$
for all odd primes~$q$ such that $v_q(n)$ is odd, and the result follows.  
Let $q$ be such a prime.

Write
$C=\ker(\psi)$ and factor $\psi$ through multiplication by~$q$ for as
long as $ E[q]\subseteq C$.  Each such factorization removes
a square $q^2$ from the degree.  Since $v_q(n)$ is odd, after these
factorizations the intersection $C_q:=C\cap E[q]$ is a nonzero
proper subgroup of $ E[q]$.  It is defined over~$\F_p$ and
hence it is fixed by Frobenius. Thus $E$ admits a cyclic $q$-isogeny over~$\F_p$.

Since $p\geq5$, every supersingular elliptic curve over~$\F_p$ has $a_p=0$. Thus the characteristic polynomial of Frobenius is $T^2+p$, and it factors mod~$q$, equivalently $ 
(-p/q)=1$.
Since $q$ is odd and $p\equiv3\pmod4$, quadratic reciprocity gives
\(
 \left(\frac{q}{p}\right)
 =
 \left(\frac{-p}{q}\right)
 =1.
\)
\end{proof}

\begin{lemma}\label{lem:ss-isogeny-degree2}
Let $p\geq11$ be a prime such that $p\equiv11\pmod{12}$.  Let
$E, E'/\F_p$ be elliptic curves with $j=j'=0$, then either $E\simeq E'$ or they are related by a $3$-isogeny.
\end{lemma}
\begin{proof}
Since $j=0$ and $p\equiv2\pmod3$, we have $E\simeq E_a:y^2=x^3+a$ and the map $x\mapsto x^3$ is an automorphism of
$\F_p^\times$. Consequently, there are exactly two $\F_p$-isomorphism
classes of elliptic curves with $j$-invariant~$0$, distinguished by
the square class of~$a$.

The standard $3$-isogeny
\(
\psi : E_a
\to
E_{-27a}
\)
is defined over~$\F_p$. Moreover, $(-27/p)=-1$
because $p\equiv11\pmod{12}$. Thus $E_a$ and $E_{-27a}$ represent the
two different $\F_p$-isomorphism classes. Hence, if $E$ and $E'$ are
not $\F_p$-isomorphic, then $\psi$ composed with a
suitable $\F_p$-isomorphism, gives an $\F_p$-rational $3$-isogeny
from $E$ to~$E'$.
\end{proof}

\begin{lemma}\label{lem:e3-cross-defect}
Let $p\equiv11\pmod{12}$ and let $E/\Q_p$ have potentially good
supersingular reduction with $e=3$. Then there exists an elliptic
curve $E'/\Q_p$ such that
\(E[p]\simeq E'[p]\)
as $G_{\Q_p}$-modules and $e'\neq3$ if and only if
\(
 p=11, \:v_{11}(\Delta_m)=8,
 \text{ and } v_{11}(c_4)\geq4.
\)

When these conditions hold, $E'$ can be chosen to have semistability
defect~$4$.
\end{lemma}

 \begin{proof}
Suppose first that there exists an elliptic curve $E'/\Q_p$ such that
\( E[p]\simeq E'[p] \)
and $e'\neq3$. By Lemma~\ref{lem:good-companion}, the curve $E'$
has potentially good supersingular reduction.

Since $e=3$, we have $\delta\in\{4,8\}$. By
\cite[Th\'eor\`eme~1]{Kraus1997,Kraus},
\[
 k(E)=
 \begin{cases}
 k_0:=\dfrac{p^2+3p+2}{3},
 & \delta=4,\ \text{or }\delta=8
   \text{ and }v_p(c_4(E))=3,\\[2mm]
 k_1:=\dfrac{p^2+5}{3},
 & \delta=8\text{ and }v_p(c_4(E))\geq4.
 \end{cases}
\]
Since $e'\in\{1,2,4,6\}$, Kraus's theorem, together with the
good-reduction case, gives
\[
 \mathcal W_{\neq3}:=
 \left\{
 2,\frac{p^2+3}{2},
 \frac{p^2+6p+5}{6},\frac{p^2+11}{6},
 \frac{p^2+4p+3}{4},\frac{p^2+7}{4}
 \right\}
\]
as the set of possible values of $k(E')$. Since
$E[p]\simeq E'[p]$, we have $k(E)=k(E')$. A comparison of the
above expressions gives
         $k_0\notin\mathcal W_{\neq3}$,
whereas, for $p\geq11$,
\[
 k_1\in\mathcal W_{\neq3}
 \quad\Longleftrightarrow\quad
 \frac{p^2+5}{3}=\frac{p^2+4p+3}{4}.
\]
The latter equality gives $p=11$.
Consequently,
\(v_{11}(\Delta_m(E))=8,
 \text{ and } v_{11}(c_4(E))\geq4.\)

Conversely, suppose that these conditions hold and consider
\(
E_1 : y^2=x^3+11x+11^3.
\)
One computes $(v_{11}(\Delta_m),v_{11}(c_4),v_{11}(c_6))=(3,1,3)$ and thus $e(E_1)=4$, and \cite[Lemme~1(b)]{Kraus1997} implies that $E_1$ has potentially good supersingular reduction.

The identity
$c_4(E)^3-c_6(E)^2=1728\Delta_m(E)$
and the hypotheses on $E$ give $v_{11}(c_6(E))=4$. Therefore the
triples
$\bigl(v_{11}(\Delta_m),v_{11}(c_4),v_{11}(c_6)\bigr)$
for $E$ are $(8,\geq4,4)$.
Thus, the hypothesis of
\cite[Proposition~2(b)]{Kraus1997} holds for both curves with parameters $\alpha_E=8$, and $\alpha_{E_1}=3$.

It follows that
\[
 \rhobar_{E,p}|_{I_{11}}
 \otimes_{\F_{11}}\F_{11^2}
 \simeq
 \omega_2^{41}\oplus\omega_2^{91}
 \simeq
 \rhobar_{E_1,p}|_{I_{11}}
 \otimes_{\F_{11}}\F_{11^2}.
\]
The two characters are distinct and conjugate and both representations have cyclotomic determinant, therefore 
Lemma~\ref{lem:niveau-two-representations} gives
$E[11]\simeq E_1[11]$ as $G_{\Q_{11}}$-modules. Taking $E'=E_1$ proves the converse and the final assertion.
\end{proof}

\begin{proof}[Proof of Theorem~\ref{thm:l=pss}]
The curve $X_E^-(p)$ has no $\Q_p$-rational cusps by Proposition~\ref{prop:cuspsQ} and
Lemma~\ref{lem:l=prep}. Suppose $(E',\varphi)\in Y_E^-(p)(\Q_p)$, i.e. $E'/\Q_p$ is an elliptic curve and $\varphi:E[p] \to E'[p]$ is an anti-symplectic $G_{\Q_p}$-isomorphism. 
The curve $E'$ has potentially good supersingular reduction by Lemma~\ref{lem:good-companion}.

In case (1), both $E/M$ and $E'/M$ have good supersingular reduction. 
As $p > 3$, we have $a_p(\Ebar) = a_p({\Ebar'})=0$ and so there is an isogeny 
$\psi : \Ebar \to \Ebar'$ of degree~$n$ which we can assume to be separable. Since $E$ is supersingular, $E[p]$ has no nontrivial \'etale subgroup.
Hence the kernel of this separable isogeny has order prime to $p$, 
implying $p \nmid n$. 
From Lemma~\ref{lem:ss-isogeny-degree} we have $(n/p)=1$. 

In cases (2) and (3), by Lemma~\ref{lem:e3-cross-defect}, we have $e'=e=3$ under the assumptions. 
For each curve, its inertial group generator gives rise to an automorphism of order $3$ on the special fiber.  Thus, $\overline{E}, \overline{E}'/\F_p$ have $j=j'=0$.
Hence, Lemma~\ref{lem:ss-isogeny-degree2} gives an isogeny $\psi : \Ebar \to \Ebar'$ of degree~$n \in \{1,3\}$, and note that $(n/p)=1$.

Therefore, $\psi|_{\Ebar[p]} : \Ebar[p] \to \Ebar'[p]$ is a symplectic isomorphism by Lemma~\ref{lem:isogenySchemes}.
On the other hand, by Proposition~\ref{prop:reduction} or Remark~\ref{rmk:reduction}, there 
is $\overline{\varphi} : \Ebar[p] \to \Ebar'[p]$
an anti-symplectic isomorphism  of $\F_p$-group schemes. 
 Hence, $\overline{\varphi}^{-1} \circ \psi|_{\Ebar[p]}$  is an anti-symplectic automorphism of~$\Ebar[p]$, contradicting Proposition~\ref{prop:ss-automorphism}.
\end{proof}

\begin{corollary}\label{cor:p7-good-ss}
Let $E/\Q_7$ have good supersingular reduction. Then
\(
 X_E^-(7)(\Q_7)=\emptyset.
\)
\end{corollary}

\begin{proof}
Exactly as in the proof of Theorem~\ref{thm:l=pss} we reduce the proof to showing that there is no 
$(E',\varphi)\in Y_E^-(7)(\Q_7)$, with $E'/\Q_7$ of potentially good reduction. Assume for a contradiction that there is $\varphi:E[7] \to E'[7]$ an anti-symplectic $G_{\Q_7}$-isomorphism.

The representation \(\overline{\rho}_{E,7}|_{I_7}\)  is irreducible by
\cite[Proposition~12]{Serre2}. Hence $E'$ cannot have potentially
ordinary reduction, since the inertia representation of a potentially
ordinary elliptic curve is reducible. Thus $E'$ has potentially
supersingular reduction. Moreover, $E[7]$ has Serre weight~$2$, and
therefore so does $E'[7]$. By \cite[Th\'eor\`eme~1 and Lemme 2]{Kraus1997},
this implies that $E'$ has good supersingular reduction: the exceptional additive
case of weight~$2$ when $p=7$ is potentially ordinary. 
We may now repeat the last two paragraphs of the proof of
Theorem~\ref{thm:l=pss}, using Remark~\ref{rmk:reduction} in place of Proposition~\ref{prop:reduction}.
\end{proof}

\subsection{Existence of $\Q_p$-points on $\Xp$} 

\begin{theorem}\label{thm:l=pe3}
Let $p\geq 7$ be a prime such that $p\not\equiv 11 \pmod{12} $. Let $E/\Q_p$
be an elliptic curve with potentially good reduction with $e=3$. Then $Y_E^-(p)(\Q_p) \neq \emptyset$.
\end{theorem}
\begin{proof}
Suppose $p \equiv 2 \pmod{3}$. Then, by \cite[Proposition~10]{symplectic}, the curve~\(E\) admits an isogeny of degree \(3\) defined over $\Q_p$. 
Therefore, for \(p\equiv5 \pmod{12}\), quadratic reciprocity gives
$(3/p)=-1$
and Corollary~\ref{cor:isogeny} concludes the proof.

It remains to consider \(p\equiv1, 7\pmod{12}\). 
By Chebotarev's Theorem, one can choose a prime~$q$ such that $q\equiv1\pmod3$ and $(q/p)=-1$.
Since \(q\neq p,3\), then $\rhobar_{E,q}(I_p)$ is cyclic of
order \(3\). Let $A$ generate this image. Because
\(q\equiv1\pmod3\), then \(\zeta_3 \in\F_q^*\), and hence there exists a basis such that 
$ A=\begin{pmatrix}
  \zeta_3&0\\
  0&\zeta_3^2
 \end{pmatrix}$.
Since \(p\equiv1\pmod3\), Theorem~\ref{thm:F}~(1) gives that $\rhobar_{E,q}(G_{\Q_p})$ is abelian,
therefore the image of Frobenius must also be diagonal, and hence $E$ admits a $ q$-isogeny. The conclusion now follows by Corollary~\ref{cor:isogeny}.
\end{proof}
\begin{theorem}\label{thm:l=re4}
    Let $p\geq 7$ be a prime such that $p\not\equiv 7 \pmod{8} $, and let $E/\Q_p$
have potentially good reduction with $e=4$. Then $Y_E^-(p)(\Q_p) \neq \emptyset$.
\end{theorem}

\begin{proof}
By \cite[Lemma~19~(i)]{symplectic}, the curve~$E$ has a $\Q_p$-point of
order~$2$, and hence admits an isogeny of
degree~$2$ defined over~$\Q_p$. Therefore, for
$p\equiv3,5\pmod8$, 

Corollary~\ref{cor:isogeny} concludes the proof.
For $p\equiv1\pmod8$, Chebotarev's theorem shows that one
can choose a prime~$q$ such that
\(
    q\equiv1\pmod4
    \text{ and }
    ({q}/{p})=-1.
\)
The rest of the proof follows exactly as the one of Theorem~\ref{thm:l=pe3}.
\end{proof}

\begin{theorem}\label{thm:l=p=11e=3}
Let $E/\Q_{11}$ have potentially good
supersingular reduction. Suppose that $v_{11}(\Delta_m)=8$ and  $v_{11}(c_4)\geq4$. Then $Y_E^-(11)(\Q_{11})\ne \emptyset$.
\end{theorem}
\begin{proof} From $v_{11}(\Delta_m)=8$ we have $e=3$.
    By Lemma~\ref{lem:e3-cross-defect}, there is an elliptic curve $E'/\Q_{11}$ with $e'=4$ and a $G_{\Q_{11}}$-isomorphism $\varphi:E[11]\to E'[11]$. If $\varphi$ is anti-symplectic, then we are done. If $\varphi$ is symplectic, by Theorem~\ref{thm:l=re4}, applied for $E'$, there exists an anti-symplectic isomorphism $\psi:E'[11]\to E''[11]$. Therefore $(E'',\psi \circ\varphi) \in Y_E^-(11)(\Q_{11})$.
\end{proof}

\begin{proposition}\label{prop:prime7defect4}
Let $E/\Q_7$ be an elliptic curve with potentially good reduction 
and $e=4$, so that $\delta=v_7(\Delta_m)\in\{3,9\}$.   
Then $X_E^-(7)(\Q_7)\ne\emptyset$ if and only if $v_7(c_6)=\frac{\delta+1}{2}.$
\end{proposition}

\begin{proof}
The isomorphism class of $X_E^-(7)$ is unchanged by replacing $E$ by a quadratic twist (see~\cite[\S4.5]{PSS}).  
If $E$ has Kodaira type $\mathrm{III}^*$, then $\delta=9$ and its quadratic twist $E' := E^{(7)}$ has Kodaira type $\mathrm{III}$ and $\delta' = 3$. This twisting operation also replaces $c_6$ by $c_6' = c_6/7^3$, therefore we can assume $E$ has Kodaira type III and $\delta = 3$.  

We consider a short Weierstrass model  
\(
 E:y^2=x^3+ax+b.
\)
From the Kodaira type assumption and Tate's algorithm we have $v_7(a)=1$ and $v_7(b)\geq2$.  We may therefore
write
\(
 a=7A\text{ and }b=49B,
 \)
where $A\in\Z_7^\times$ and $B\in\Z_7$.  Notice that
$v_7(c_6)=v_7(b)$, since $c_6=-864b$.
After reordering the terms in~\cite[(7.6)]{PSS}, the curve $X_E^-(7)$ is the plane
quartic $G(x,y,z)=0$, where
\begin{align*}
G={}&-a^2x^4+2abx^3y-12bx^3z-(6a^3+36b^2)x^2y^2
       +6ax^2z^2 +2a^2bxy^3-12abxy^2z\\
 & +18bxyz^2 +(3a^4+19ab^2)y^4 
 -(8a^3+42b^2)y^3z+6a^2y^2z^2-8ayz^3+3z^4.
\end{align*}
Suppose first that $(x:y:z)\in X_E^-(7)(\Q_7)$.  Scale its coordinates
so that $x,y,z\in\Z_7$ and at least one of them is a unit.  Reduction of
$G$ modulo~$7$ gives $3z^4$, and hence $7\mid z$.  After writing
$z=7z_1$, dividing by~$7^2$ and reducing again, one obtains
$-A^2x^4$; hence $7\mid x$.  It follows that $y$ is a unit, so we may
take $y=1$ and write $x=7X$, $z=7Z$.

Dividing by~$7^4$ and reducing modulo~$7$ now gives
\begin{align*}
0&=3A^4-8A^3Z+6A^2Z^2-8AZ^3+3Z^4=3(Z+A)^2(Z^2+A^2)\pmod 7.
\end{align*}
Since $-1$ is not a square modulo~$7$, this forces $Z\equiv-A\pmod7$.
Write $Z=-A+7T$.  A final substitution shows that
\[
 H(X,T):=\frac{G(7X,1,7(-A+7T))}{7^5}\in\Z_7[X,T]
\]
and gives
\[
 H(X,T)
 \equiv-3A\bigl(A^3+ABX+3B^2\bigr)\pmod7.                 \tag{*}
\]
If $7\mid B$, the right-hand side of~$(*)$ is congruent to
$-3A^4\ne0$, a contradiction.  Thus a rational point can exist only if
$B$ is a unit.

Conversely, suppose that $B\in\Z_7^\times$.  Set $T=0$ and choose
$X_0\in\F_7$ such that
\[
 A^3+ABX_0+3B^2=0.
\]
Such an $X_0$ exists and is unique because $AB$ is a unit.  Thus,
\[
\frac{\delta H(X,0)}{d X}(X_0,0)\equiv-3A^2B\not\equiv 0 \pmod 7
\]
Hensel's lemma therefore
lifts $X_0$ to an $X\in\Z_7$ satisfying
\(
 G(7X,1,-7A)=0,
\)
giving a point of $X_E^-(7)(\Q_7)$.  We have proved that, in type
$\mathrm{III}$, $X_E^-(7)(\Q_7)\ne \emptyset$ is equivalent to $B$ being a unit, or
equivalently to $v_7(c_6)=2$.
\end{proof}
\begin{remark}
The two cases in Proposition~\ref{prop:prime7defect4} also admit a
representation-theoretic interpretation. By
\cite[Proposition~2, Lemme~2 and Proposition~4]{Kraus1997}, when
$E$ has type $\mathrm{III}$, the condition $v_7(c_6)=2$ is precisely
the exceptional case in which the action of inertia on $E[7]$ is
reducible and wild. If $v_7(c_6)\geq3$, the inertia representation is
irreducible and given by fundamental characters of level~$2$. Thus the valuation condition appearing in
the proposition is exactly the boundary between the reducible wild
and irreducible inertia cases.
\end{remark}
 
\section{Proof of the main theorems}

\begin{proof}[Proof of Theorem~\ref{thm:MAIN}]
For $K=\R$ this follows from Theorem~\ref{thm:real}. 
Let $K = \Q_\ell$ with $\ell \neq p$.
When $E/K$ has potentially multiplicative reduction, the result follows from Theorem~\ref{thm:multiplicative}, so assume $E/K$ has potentially good reduction. 
From Theorem~\ref{thm:abelian}, the problem reduces to (i) elliptic curves $E/\Q_\ell$ with good reduction up to quadratic twist and $\rhobar_{E,p}(\Frob_\ell)$ of order divisible by $p$; and (ii) elliptic curves with semistability defect $e \geq 3$ and $\rhobar_{E,p}(G_{\Q_\ell})$ non-abelian.
Now, the statement for $E/\Q_\ell$ having good reduction ($e=1$) follows from Theorem~\ref{thm:maingood}  and the statement for $e=2$ follows from Corollary~\ref{cor:e2}. The claims for $e=4$ follow from Theorems~\ref{thm:e4l} and~\ref{thm:e4l2}, for $e=3$ they follow from Theorems~\ref{thm:e3l} and~\ref{thm:e3l3}, and for 
$e=6$ from Corollary~\ref{cor:e6}. Lastly, for $e = 8, 24$ the result follows from Theorem~\ref{thm:l2n-abinertia}, and for $e=12$ it follows from Theorem~\ref{thm:l3n-abinertia}. 
\end{proof}
 
\begin{proof}[Proof of Theorem~\ref{thm:l=p}]
When $E/\Q_p$ has potentially multiplicative reduction, the result follows from Theorem~\ref{thm:multiplicative}. 

Assume first that $E/\Q_p$ has good reduction. If $a_p=0$ and $p \equiv 7\pmod 8$, then $\Xp(\Q_p) = \emptyset$ by Theorem~\ref{thm:l=pss} and Corollary~\ref{cor:p7-good-ss}. 
In all of the other instances, we have $\Xp(\Q_p)\neq \emptyset$.
 This claim follows from Proposition~\ref{prop:good1}. Indeed, note that $\Delta_p=a_p^2-4p$, and suppose that assumptions (1) and (2) in that proposition fail. Therefore, $p\equiv 3 \pmod 4$, and $-p\Delta_p$ being a square implies that $p\mid \Delta_p$ giving $p \mid a_p$. By the Hasse bound for $p\geq 7$ this is only possible if $a_p=0$. Then, the failure of (3) can only happen if $(2/p)=1$ (as $\Delta_p=-4p$), i.e., precisely when $p\equiv7 \pmod 8.$ 
 
 Suppose now that $E/\Q_p$ has potentially good reduction with $e=4$. If $p \not\equiv 7 \pmod 8$, then $\Xp(\Q_p)\neq \emptyset$ by Theorem~\ref{thm:l=re4}. If $p=7$ then the conclusion follows by Proposition~\ref{prop:prime7defect4}. If $p>7$, and $p\equiv 7 \pmod 8$ then $\Xp(\Q_p) = \emptyset$ by Theorem~\ref{thm:l=pss}, since $p=-1\pmod 4$ implies $E$ has potentially good supersingular reduction by Corollary~\ref{cor:al0}. 

Suppose next that $E/\Q_p$ has potentially good reduction with $e=3$. If $p \not \equiv 11 \pmod {12}$, then $\Xp(\Q_p)\neq \emptyset$ by Theorem~\ref{thm:l=pe3}. If $p$ is as in Table~\ref{table:Qp} then $\Xp(\Q_p) = \emptyset$ by Theorem~\ref{thm:l=pss}, since $p=-1\pmod 3$ implies $E$ has potentially good supersingular reduction by Corollary~\ref{cor:al0}. Finally the case $p=11$, with $v_{11}(\Delta_m)=8,\: v_{11}(c_4)\geq4$ is established by Theorem~\ref{thm:l=p=11e=3}.

Finally, the cases of semistability defect $e=2$ and $e=6$ follow by 
Lemma~\ref{lem:qtwists} applied to the twist $E^{\sqrt{p}}$ which has $e=1$ and $e=3$, respectively, by \cite[Lemma 3]{symplectic} and \cite[Lemma 1]{symplectic}.
\end{proof}

\begin{proof}[Proof of Corollary~\ref{cor:semistable}]
This is an immediate consequence of Theorems~\ref{thm:MAIN} and \ref{thm:l=p}.
\end{proof}
\section{Counterexamples to the Hasse principle}
\label{sec:Hasse}

In this section, we use our results to produce, for infinitely many~$p$, counterexamples to the Hasse principle of the form $\Xp$. Assuming the Frey--Mazur conjecture, we further show there are counterexamples for infinitely many $E$ and~$p$.

\begin{lemma}\label{lem:irred}
Let $p \geq 3$ be a prime and $E/\Q$ be an elliptic curve.
Suppose that $\rhobar_{E,p}$ is irreducible.
Then all automorphisms of $E[p]$ are multiplication by a scalar.
Moreover, if $\psi : E\to E'$ is a  $\Q$-isogeny of degree $d$ with $p\nmid d $ and $\phi : E'[p] \to E[p]$ is a $G_\Q$-isomorphism, then $\phi$ is symplectic if and only if $(d/p)=1$.
\end{lemma}

\begin{proof}
Since $\overline{\rho}_{E,p}$ is odd and irreducible, it is absolutely
irreducible~\cite[Lemma~24]{DahmenThesis}.
Therefore, Schur's Lemma gives
$\End_{G_{\Q}}(E[p])=\F_p$,
which proves the first assertion.

Since $p\nmid d$, the restriction $\psi|_{E[p]}:E[p]\longrightarrow E'[p]$
is an isomorphism. By \cite[Lemma~6]{symplectic}, the isomorphisms
$\phi^{-1}$ and $\psi|_{E[p]}$ have the same symplectic type.
Moreover, $\psi|_{E[p]}$ is symplectic if and only if
$({d}/{p})=1$ by Lemma~\ref{lem:isogeny}, proving the second assertion.
\end{proof}

\begin{proposition} \label{prop:conductor}
Let $p > 7$ be a prime. Let $E/\Q$ and $E'/\Q$ be elliptic curves having integral $j$-invariants and isomorphic $p$-torsion modules. 
Then $E$ and~$E'$ have the same conductor.
\end{proposition}
\begin{proof}

From $\rhobar_{E,p} \simeq \rhobar_{E'p}$, we have equality of Serre levels and Serre weights
\[
 N(\rhobar_{E,p}) = N(\rhobar_{E',p}),  \qquad k(\rhobar_{E,p}) = k(\rhobar_{E',p}).
\]
Since $j(E)$ and $j(E')$ belong to $\Z$, we know that neither $E$ nor $E'$ has any primes of potentially multiplicative reduction. Since $p>3$, it follows from \cite[page~28]{Kraus1997} that $N(\rhobar_{E,p})$ is equal to~$N_E$ away from~$p$ and similarly for $E'$. We are left to show that
$v_p(N_E) = v_p(N_{E'})$.

As $p > 7$, it follows from \cite[Th\'eor\`eme 1]{Kraus1997} and the discussion preceding it that  $k(\rhobar_{E,p}) = 2$ if and only if $E$ has good reduction. So, the above equality of weights shows that $E$ has good reduction if and only if $E'$ has good reduction. In this case, $v_p(N_E) = v_p(N_{E'})=0$ otherwise both curves have additive potentially good reduction and so $v_p(N_E) = v_p(N_{E'})=2$ (as~$p > 3)$.
\end{proof}

\begin{theorem} \label{thm:HasseSplit}
Let $K=\Q(\sqrt{D})$ where $D \in \{-11,-19,-43,-67,-163 \}$. Let also $E/\Q$ be an elliptic curve with CM by~$K$. If $p > 7$ is a prime such that $p \equiv 5 \pmod{8}$ and $(D / p) = 1$, then $\Xp$ is a counterexample to the Hasse principle.
\end{theorem}
\begin{proof}
For each fixed field $K$ in the statement, all elliptic curves over $\Q$ with CM by~$K$ have the same
$j$-invariant $\neq 0, 1728$, therefore the curve $E/\Q$ is uniquely determined up to a quadratic twist. Note that, by Lemma~\ref{lem:qtwists}, the curve $\Xp$ is a counterexample to the Hasse principle if and only if $X_{E^u}^-(p)$ is a counterexample to the Hasse principle for any quadratic twist~$E^u /\Q$ of $E$. Thus,
we can assume that $E$ has minimal conductor among its twists. That is, for each $D = -11,-19,-43,-67$ or $-163$, we take $E$ to be the elliptic curve with Cremona label
$121b1$, $361a2$, $1849a2$, $4489a2$ or $26569a2$, respectively.

Note that $q := -D$ is a prime satisfying $q \equiv 3 \pmod{4}$.
For each $E$ in the above list, a quick consultation of LMFDB~\cite{LMFDB} shows the following facts: 
\begin{enumerate}
    \item[(i)] $E$ has conductor $N_E=q^2$ and potentially good reduction at $q$ with semistability defect $e = 4$;\label{CM1}
        \item[(ii)] the isogeny class of $E$ is the unique one of conductor $q^2$ containing a CM curve; \label{CM3}
    \item[(iii)] the isogeny class of $E$ contains precisely two elliptic curves i.e. $E$ and $E^D$ related by a cyclic isogeny of degree $q$; \label{CM2}
    \item[(iv)] the mod $p$ representation of $E$ is irreducible for all primes $p \neq q$. \label{CM4}
\end{enumerate}

From (iv) and \cite[p. 422, Th\'eor\`eme 7]{Halberstadt} we know that the image of $\rhobar_{E,p}$ is irreducible and contained in the normalizer of a Cartan subgroup of $\GL_2(\F_p)$. Moreover, it is a split Cartan if $p$ splits in $K$ and a non-split Cartan if $p$ is inert in $K$.
In particular, $p \nmid \#\rhobar_{E,p}(G_\Q)$.

The fact that $X_{E}^-(p)(\R) \neq \emptyset$ and $\Xp(\Q_\ell) \neq \emptyset$ for all primes~$\ell$ follows from Theorems~\ref{thm:MAIN} and~\ref{thm:l=p}. Note that for all good primes~$\ell \nmid N_E = q^2$, we use the fact that $p \nmid \#\rhobar_{E,p}(\Frob_\ell)$ for CM curves. For $\ell=q$, by~(i), we are in case $e=4$ and the hypothesis $p \equiv 5 \pmod{8}$ assures that $(2/p)=-1$.

To complete the proof, we will show that $\Xp(\Q) = \emptyset$. We start by noting that Proposition~\ref{prop:cuspsQ} assures that the cusps of $\Xp$ are not defined over $\Q$. Thus, we are left with showing that $Y^-_E(p)(\Q) = \emptyset$. Assume there is a rational point $(E', \phi) \in Y^-_E(p)(\Q)$, where $E'/\Q$ is an elliptic curve and 
$\phi : E[p] \to E'[p]$ is an anti-symplectic $G_\Q$-isomorphism.

The assumption $(D/p) = 1$ implies that the image of $\rhobar_{E,p}$ is contained in the normalizer of a split Cartan subgroup (since $p$ splits in~$K$), hence the same holds for
$\rhobar_{E',p}$. By \cite[Theorem 1.2]{Xsplit13}, it follows that $E'$ has CM. In particular, both $E$ and $E'$ have integral $j$-invariants and so
$N_{E'}= N_{E}$ by Proposition~\ref{prop:conductor}. Thus $E$ and $E'$ are isogenous by~(ii) above.

By~(iii), we let $\psi : E' \to E$ be either an isomorphism or a cyclic isogeny of degree~$q = -D$. By quadratic reciprocity and our assumptions $(q/p)=(-D/p)=(D/p)=1$. Hence Lemma~\ref{lem:irred} implies that $\phi$ is symplectic, a contradiction.
\end{proof}
\begin{remark}\label{rem:CM} The previous theorem is optimal in the following sense. The CM curves $E/\Q$ not covered by Theorem~\ref{thm:HasseSplit} do not give rise to counterexamples to the Hasse Principle for any prime $p\geq 7$.
Indeed, up to quadratic twist, the remaining CM curves are (given by their Cremona labels)
\[
 49a1, \; 49a4, \; 32a3, \; 32a2, \; 27a2, \; 27a1, \; 36a4, \; 256a1.
\]
The curve $36a4$ admits a $2$-isogeny and a $3$-isogeny, so we need $(2/p) = (3/p) = 1$ to avoid a $\Q$-point on $\Xp$ via Lemma~\ref{lem:isogeny}. Moreover, it has semistability defect $e=4$ at $\ell=3$, hence $\Xp(\Q_3) = \emptyset$ by Theorem~\ref{thm:MAIN}. We divide the remaining 7 curves into 3 groups:

(i) 27a1, 27a2. These curves admit a $3$-isogeny and have semistability defect $e=12$ at~$\ell=3$. Thus $(3/p) = 1$ to avoid a $\Q$-point on $\Xp$ but then $\Xp(\Q_3) = \emptyset$ by Theorem~\ref{thm:MAIN}.

(ii) 256a1, 32a2, 32a3. These curves admit a $2$-isogeny, thus $(2/p) = 1$ to avoid a $\Q$-point on $\Xp$. Moreover, the first curve has semistability defect $e=4$ at~$\ell =2$ and the last two $e=8$, but then $\Xp(\Q_2) = \emptyset$ by Theorem~\ref{thm:MAIN}.

(iii) 49a1, 49a4. These curves admit a 2-isogeny and a 7-isogeny and have semistability defect $e = 4$ at~$\ell=7$. The conclusion follows as for $36a4$.
\end{remark}

The assumption $(D/p) = 1$ in Theorem~\ref{thm:HasseSplit} ensures $\rhobar_{E,p}$ lies in the normalizer of a split Cartan subgroup of $\GL_2(\F_p)$; this is necessary due to the absence of an analogous result to \cite[Theorem 1.2]{Xsplit13} for the non-split case. Assuming the following consequence of Serre's uniformity conjecture allows for the same proof in the non-split case.

\begin{conj}\label{conj:nonsplit}
Let $p$ be a prime.
Then there exists an elliptic curve $E/\Q$ without CM such that
the image of $\rhobar_{E,p}$ is contained in the normalizer of a non-split Cartan subgroup
of $\GL_2(\F_p)$ if and only if $p \leq 11$.
\end{conj}

\begin{theorem} \label{thm:HasseNonSplit} Assume Conjecture~\ref{conj:nonsplit}.
Let $D$, $K$ and~$E$ be as in Theorem~\ref{thm:HasseSplit}.
If $p \geq 11$ is a prime such that $p \equiv 3 \pmod{8}$ and $(D / p) = -1$, then $\Xp$ is a counterexample to the Hasse principle.
\end{theorem}

We recall the following well-known conjecture (see \cite{FreyMazur} for its original, weaker formulation).
\begin{conj}[Frey--Mazur] \label{conj:FM}
Any two elliptic curves $E/\Q$ and $E'/\Q$ with isomorphic $p$-torsion modules for some $p > 17$ are $\Q$-isogenous.
\end{conj}
Using Theorem~\ref{thm:MAIN}, we can conjecturally create counterexamples to the Hasse principle for infinite choices of $E$ and $p$. The following theorem provides a non-exhaustive list of such counterexamples.
\begin{theorem} \label{thm:Frey-Mazur}
    Let $p>17$ be a prime. Let $E/\Q$ be an elliptic curve semistable at $p$. Assume the Frey--Mazur Conjecture holds for $E$. Assume also that one of the following holds: 
   \begin{enumerate}
    \item $p \equiv 1 \pmod{4}$, $\rhobar_{E,p}$  is irreducible, and all $\Q$-isogenies of $E$ of degree $d$ coprime to $p$ satisfy $(d/p)=1$;
    \item $p \equiv 5 \pmod{8}$, $E[2](\Q) =\{0\}$ and $E/\Q_2$ has potentially good reduction with $e \in \{8,24\}$ or $e=4$ and $\tilde{c}_4\equiv 5\tilde{\Delta} \pmod{8}$;
    \item $p \equiv 5 \pmod{12}$, $E$ has no $2$-and $3$-isogenies and $E/\Q_3$ has potentially good reduction with
    $e =12$ or $e=3$ and~$\tilde{\Delta} \equiv 2 \pmod{3}$;
    \item $p \equiv 5 \pmod{24}$, $E$ has no non-scalar $\Q$-isogenies;
   \end{enumerate}
   Moreover, for the primes $\ell \mid N_E$ whose reduction type is not specified in each case, assume that the reduction type of $E/\Q_\ell$ does not belong to Table~\ref{table:main}. 
   
   Then $X_E^-(p)$ is a counterexample to the Hasse principle.
\end{theorem}
\begin{proof}
Let $E$ be as in the statement. Firstly, we have $\Xp(\Q_p) \neq \emptyset$
by Theorem~\ref{thm:l=p}.

The fact that $X_{E}^-(p)(\R) \neq \emptyset$ and $\Xp(\Q_\ell) \neq \emptyset$ for all~$\ell\neq p$ follows from Theorem~\ref{thm:MAIN}. Indeed, since $p \equiv 1 \pmod{4}$ in all cases, we have $X_{E}^-(p)(\Q_\ell) \neq \emptyset$ for all primes $\ell \nmid N_E$ and quadratic reciprocity gives $(3/p) = (p/3)$.
In case (2), from $p \equiv 5 \pmod{8}$ we have $(2/p) = -1$, and  Theorem~\ref{thm:MAIN} gives $X_{E}^-(p)(\Q_2) \neq \emptyset$. Similarly in case (3), from $p \equiv 2 \pmod{3}$, it follows that $(3/p) = (p/3)=-1$, and Theorem~\ref{thm:MAIN} gives $X_{E}^-(p)(\Q_3) \neq \emptyset$ once again. Finally, the last assumption assures that Theorem~\ref{thm:MAIN} applies to all other primes $\ell$.

Note that, in case (4), $E$ has no non-scalar $\Q$-isogenies by assumption. We claim the same is true for cases~(2) and~(3). 
Indeed, in case (2),
if \(e=4\) and
\(\tilde{c}_4\equiv5\tilde{\Delta}\pmod 8\), then the
\(q\)-torsion field is non-abelian for every prime \(q\geq3\):
this follows from Theorem~\ref{thm:F} for \(q\geq5\), together with
the fact that the abelianity of the torsion field is independent of
the torsion prime. Furthermore, the image of \(\rhobar_{E,q}\)
contains no element of order \(q\), by
Proposition~\ref{prop:G} for \(q\geq5\) and by
\cite[Proposition~9]{symplectic} for \(q=3\).
Hence $\rhobar_{E,q}(G_{\Q_2})$ is not contained in a Borel subgroup, and thus 
$\rhobar_{E,q}(G_\Q)$ is irreducible for all $q\geq3$. 
If $e\in \{8,24\}$, the restriction of $\rhobar_{E,q}$ to the inertia subgroup $I_2$ is irreducible for all primes $q\geq 3$, and in particular $\rhobar_{E,q}(G_\Q)$ shares the same property. This, together with the absence of rational $2$-torsion points, implies that $\rhobar_{E,q}(G_\Q)$ is irreducible for all primes $q$ in all of the cases. Therefore, $E$ has no non-scalar $\Q$-isogenies. Indeed, if a \(\Q\)-isogeny were non-scalar, then, after factoring
out all possible multiplication maps, its kernel would contain a
nonzero proper \(G_\Q\)-stable subgroup of \(E[q]\) for some prime
\(q\), contradicting the irreducibility of \(\rhobar_{E,q}\).
 
A similar argument applies in case~\textup{(3)}. If \(e=3\) and
\(\tilde{\Delta}\equiv2\pmod3\), then
\(\rhobar_{E,q}(G_{\Q_3})\) is irreducible for every prime
\(q\geq5\), by Theorem~\ref{thm:F} and
Proposition~\ref{prop:G}. If \(e=12\), the restriction to inertia is
already irreducible for every \(q\geq5\). Together with the
assumption that \(E\) admits no \(2\)- or \(3\)-isogenies, this shows
that \(\rhobar_{E,q}(G_\Q)\) is irreducible for every prime~\(q\).
Hence \(E\) has no non-scalar \(\Q\)-isogenies also in
case~\textup{(3)}.

To complete the proof, we will show that $\Xp(\Q) = \emptyset$. We start by noting that Proposition~\ref{prop:cuspsQ} assures that the cusps of $\Xp$ are not defined over $\Q$. Thus, we are left with showing that $Y^-_E(p)(\Q) = \emptyset$. Assume there is a rational point $(E', \phi) \in Y^-_E(p)(\Q)$, where $E'/\Q$ is an elliptic curve and 
$\phi : E[p] \to E'[p]$ is an anti-symplectic $G_\Q$-isomorphism. Note that $E\neq E'$ as otherwise $E[p]$ would have an anti-symplectic automorphism, contradicting the first statement of Lemma~\ref{lem:irred}.

By the Frey--Mazur conjecture, the curves \(E\) and \(E'\) are
\(\Q\)-isogenous.
In cases~\textup{(2)}--\textup{(4)}, the curve \(E\) admits only
scalar \(\Q\)-isogenies, and therefore \(E=E'\), contradicting the
above observation.
In case~\textup{(1)}, choose a \(\Q\)-isogeny
$\psi\colon E\longrightarrow E'$ 
of minimal degree \(d\). As in the proof of
Lemma~\ref{lem:cyclic-isog}, the isogeny \(\psi\) is cyclic. If
\(p\mid d\), then the unique subgroup of order \(p\) in
\(\ker\psi\) is \(G_\Q\)-stable, contradicting the irreducibility of
\(\rhobar_{E,p}\). Hence \(p\nmid d\), and therefore
$\left(\frac{d}{p}\right)=1$
by hypothesis. Lemma~\ref{lem:irred} now implies that \(\phi\) is
symplectic, a contradiction.
\end{proof}

Observe that Corollary~\ref{cor:Frey-MazurSimple} is a direct consequence of part (1). 

Finally, since the counterexamples to the Hasse principle in Theorem~\ref{thm:conterHasseCM} are all in the irreducible case, 
we end this section with two unconditional counterexamples in which the residual representations are reducible.

\begin{example} \label{ex:final}
Consider the non-isogenous elliptic curves of conductor $2 \cdot 5^2 \cdot 17^2$
 \begin{align*}
E_1 &\; \colon y^2 + xy + y = x^3 - 190891\,x - 36002922,
  \quad j(E_1) = -\dfrac{17\cdot 373^3}{2^{17}}, \\[1em]
E_2 &\; \colon y^2 + xy + y = x^3 - 3041\,x + 64278,  
 \quad  j(E_2) = -\dfrac{17^2\cdot 101^3}{2}.
\end{align*}   
We will show that $X_{E_i}^{-}(17)$ are counterexamples to the Hasse principle.

From \cite[Theorem~1.10]{Zywina2015}, the quadratic twists of $E_1$ and~$E_2$ are the unique curves whose 17-torsion representation has image isomorphic to
the groups $G_1$ and~$G_2$ (defined in \S1.7 of {\it loc. cit }), respectively. Moreover, since the $G_i$ are reducible and non-abelian, the curve~$E_i$ is not mod~17 congruent to any of its non-trivial quadratic twists by \cite[Lemma~2.3]{CremonaFreitas}. 
Therefore, if $\phi : E_i[17] \to E'[17]$ is a $G_\Q$-isomorphism, then $E' = E$ and $\phi$ is scalar, hence a symplectic isomorphism.  Together with Proposition~\ref{prop:cuspsQ} this shows that $X_E^-(17)(\Q) = \emptyset$.

Note that the $E_i$ have multiplicative reduction at $\ell=2$ and
semistability defect $e(E_i/\Q_5)=4$. We conclude that $X_E^{-}(17)$ has local points everywhere away from $p$ by Theorems~\ref{thm:MAIN}. Finally, since the semistability defect $e(E_i/\Q_{17})=3$, it follows from Theorem~\ref{thm:l=p} that $X_E^{-}(17)$ has local points everywhere.
\end{example}
\vspace{-2mm}
\section{The Diophantine equation \texorpdfstring{$x^3+b^3=Cz^p$}{x³+b³=Czᵖ}}\label{sec:33p}

Consider the generalized Fermat equation
\begin{equation} \label{E:GFE1}
  x^r + y^q =Cz^p,
\end{equation}
where $p,q,r > 2$ are primes and $C\in \Z_{>0}$.
A solution $(a,b,c) \in \Z^3$ of~\eqref{E:GFE1} is called {\it primitive} if $\gcd(a,b,c)=1$ and it is called {\it non-trivial} if $abc \neq 0$. The triple of exponents $(r,q,p)$ are called the {\it signature} of the equation.

The folklore expectation (\textit{Beal's Conjecture, Fermat--Catalan Conjecture}, etc.) is that when $p,q,r$ are large enough, the equation~\eqref{E:GFE1} has no primitive non-trivial solutions, and this is well known to be a consequence of the \textit{abc} conjecture.
Wiles' famous proof of Fermat's Last Theorem \cite{Wiles} opened up the door for tackling families of signatures $(r,q,p)$, where two of the exponents are fixed and one is varying.
This approach is called {\it the modular method} and works by contradiction. It relies on the modularity of elliptic curves and can be summarized in a few steps: attaching a Frey curve, modularity, level lowering and elimination.
The
modular method is very powerful to prove the non-existence of solutions for large varying $p$
but often fails for small values of $p$. As an application of our main theorem, we will introduce
a new technique for the elimination step that applies for concrete small values of $ p$.

A survey on many solved signatures using the modular method can be found in \cite{BCDY}.
In this section, we consider the signature $(3,3,p)$, more precisely, the generalized Fermat equation
\begin{equation} \label{eq:33p}
  x^3 + y^3 = Cz^p,
\end{equation}
where $p > 3$ is a prime and $C\in \Z_{>0}$. Many authors (see~\cite{Bruin, ChenSiksek,Freitas,Kraus33p}) have studied the solutions of \eqref{eq:33p} for $C=1$ and their combined results show that there are no primitive, non-trivial solutions for a set of prime exponents $p$ with density greater than $0.844$. Moreover, many values of $C\neq 1$ are tackled in \cite{BBF}. In what follows,  we will prove Theorem~\ref{thm:33p} via the modular method.

\begin{remark} Note that $5 \in S_1 \subset S_0$ where the $S_i$ are defined in~\cite[Introduction]{BBF}, hence the non-existence of solutions given by Theorem~\ref{thm:33p} does not follow from Theorem~1.4 in \loccit; note also  that Theorem~\ref{thm:33p} is not covered by Theorems~1.8 and~1.9 in \loccit $\:\:$Finally, observe that for a fixed $p$, the theorem covers 50\% of the integers~$\alpha$.
\end{remark}

\begin{remark}
Allowing the algorithm in~\cite{git} to run longer, we can easily increase the set of primes $p$ covered by Theorem~\ref{thm:33p}.
\end{remark}

\subsection{Frey Curve}\label{Frey33p}
Let $p > 5$ and suppose that $(a,b,c)\in \Z^3$ is a non-trivial, primitive solution to
\begin{equation}\label{eq:33p5}
x^3+y^3=5^\alpha z^p,
\end{equation}
where $\alpha$ is a positive integer. Without loss of generality, $ac$ is even and $b \equiv (-1)^{c+1} \pmod{4}$.
Following \cite{BBF}, we associate to such a solution a {\it Frey elliptic curve} 
$F^{(i)}_{a,b}$ of the shape
$$
F^{(0)}_{a,b} \; \; : \; \; y^2+xy=x^3 + \frac{3( b-a) +2}{8} x^2 + \frac{3(a+b)^2}{64} x + \frac{ 9 (b-a)(a+b)^2}{512},
$$
or
$$
 F^{(1)}_{a,b} \; \; : \; \;  y^2=x^3+3abx+b^3-a^3,
$$
depending on whether $c$ is even or odd, respectively. We note that these are minimal models with discriminant given by
\[
  \Delta(F^{(i)}_{a,b})=-2^{12i-8} 3^3 5^{2\alpha} c^{2p}.
\]
Moreover, setting $F := F^{(i)}_{a,b}$, Lemma 2.1 in \loccit\: gives
\begin{equation}\label{eq:cond33p}
N_{F}=\begin{cases}
               90 \, \mathcal{R}&\text{ if $c$ even, }b\equiv -1\text{ {\rm (mod} } 4), \; 
               \text{or}\\
               180 \, \mathcal{R}&\text{ if $c$ odd, } v_2(a)\geq 2\text{ and }b\equiv1\text{ {\rm (mod} } 4), \; \text{or}\\
               360 \, \mathcal{R}&\text{ if $c$ odd, } v_2(a)=1\text{ and }b\equiv1\text{ {\rm (mod} } 4),
              \end{cases}\end{equation}
where 
$\mathcal{R}$ denotes the product of the primes $\ell$ satisfying $ \ell \mid c$ and $\ell \nmid 30$.
\subsection{Modularity and level lowering}\label{sec:MLL}  
We note that $\rhobar_{F,p}$ is odd, it has Serre weight 2 and Serre level $N_p := N_F / \mathcal R$. From standard irreducibility results and Serre's Conjecture, one gets that, for $p \geq 17$, there exists a newform $f\in
S_2^{new}(N_p)$ (the space of weight $2$ cuspidal newforms for the congruence subgroup $\Gamma_0(N_p)$)
and a prime $\fp \mid p$ in the field of coefficients of~$f$ such that 
$\rhobar_{F,p} \simeq \rhobar_{f,\fp}$
where $\rhobar_{f,\fp}$ denotes the
modulo~$\fp$ representation attached to $f$. A quick consultation of 
LMFDB~\cite{LMFDB} reveals that, at level $N_p \in \{90,180,360\}$,
all newforms have rational field of coefficients, so they correspond to the isogeny class of elliptic curves by the Eichler-Shimura correspondence. More precisely, we conclude that 
\begin{equation} \label{iso2}
\rhobar_{F,p} \simeq  \rhobar_{E_f,p},
\end{equation}
where $E_f$ belongs to the following set
of elliptic curves  (given by their Cremona reference)
\begin{equation}\label{eq:Ef}
    \{ 90a4,90b2,90c2,180a2, 360a2, 360b2, 360c2, 360d2,360e2 \}.
\end{equation}

\subsection{Elimination}\label{sec:elimination}
It follows from the isomorphism~\eqref{iso2} that,
for all primes $\ell \nmid pN_p$,
\[
\begin{cases}
    a_{\ell}(E_f) \equiv a_{\ell}(F) \pmod{p}, \:\: \text{ if } \ell \nmid  \mathcal{R},\\
    a_\ell(E_f) \equiv \pm (\ell+1) \pmod{p}, \:\: \text{ if } \ell \mid  \mathcal{R},
\end{cases}
\]
and consequently, \[p\mid B_\ell(E_f):=((\ell+1)^2-a_\ell(E_f)^2)\prod_{t\in S_\ell} (t-a_\ell(E_f)),\] where
\[
S_\ell:=\{t\in \Z: \:\: t=a_\ell(F_{a,b}), \text{ for all values }a,b \in \F_\ell \text{ such that }F^{(i)}_{a,b} \text{ has good reduction}  \}.
\]

With {\tt Magma} we check that the quantity $B_7(E)$ eliminates the curves
$E \in \{180a2,360b2,360c2 \}$ for $p > 7$; furthermore, all the other curves satisfy $B_\ell(E_f) = 0$ for all primes $\ell < 500$ of good reduction.
To eliminate the remaining isogeny classes in the list, we will introduce a symplectic argument  using primes of good reduction. Concretely, we show that the symplectic information coming from the multiplicative prime $5$ conflicts with our Theorem~\ref{thm:MAIN} at a suitably chosen prime of good reduction $\ell$.

We will need the following slightly modified version of \cite[Proposition 6.1]{BBF}.

\begin{proposition}\label{prop:sym5}
    Let $(a,b,c)\in\Z^3$ be a primitive solution to \eqref{eq:33p} with $C=5^\alpha$,~$(\alpha/p)=-1$. Suppose that~\eqref{iso2} holds with $$E_f \in \{ 90a4,90b2,90c2,360a2,360d2,360e2 \}.$$
    Then $X_{E_f}^-(p)(\Q)\neq \emptyset$.
\end{proposition}
\begin{proof}
    We have $v_5(N_F) = v_5(N_{E_f}) = 1$, so $5$ is a prime of multiplicative
reduction of both $F$ and~$E_f$; further, we have $v_5(\Delta(E_f)) =2$ for all $E_f$ in the statement.
Moreover, we have that~$p \nmid \alpha$ and
$v_5(\Delta(F)) = 2\alpha + 2pv_5(c)$.
A direct application of
\cite[Proposition~2]{KO} at $\ell = 5$ implies that the Galois isomorphism \eqref{iso2} is anti-symplectic. Therefore, the Frey curve $F = F_{a,b}$ gives rise to a point on $X_{E_f}^-(p)(\Q)$.
\end{proof}
\begin{proof}[Proof of Theorem~\ref{thm:33p}]
We now put together all the steps of the modular method. Suppose that $(a,b,c)\in \Z^3$ is a non-trivial, primitive solution to \eqref{eq:33p5}. Then we attach the Frey elliptic curve $F = F^{(i)}_{a,b}$ given in Section~\ref{Frey33p}. By Sections~\ref{sec:MLL}--\ref{sec:elimination}, this gives rise to the isomorphism~\eqref{iso2} where $E_f$ is one of the curves in Proposition~\ref{prop:sym5}.

The code provided in \cite{git} shows that for all primes $p$ in the statement and $$E\in \{ 90a4,90b2,90c2, 360a2,360d2,360e2\},$$ there exists an $\ell \neq p$ satisfying the conditions in first row of Table~\ref{table:main} and thus
$X_E^-(p)(\Q_\ell)= \emptyset$ implies that $X_{E}^-(p)(\Q)=\emptyset,$
contradicting Proposition~\ref {prop:sym5} since $(\alpha/p)=-1$ by assumption.
\end{proof}

\subsection{Declaration of interests}
On behalf of all authors, the corresponding author states that there is no conflict of interest.
\vspace{-2mm}
\subsection{Data availability }
All code used for the computations in this article is available 
at~\cite{git}. No additional datasets were generated or analysed.

\end{document}